\documentclass[a4paper,12pt]{amsart}
\usepackage[
    a4paper,
    textwidth=13.5cm,
    textheight=22.5cm,
    centering
]{geometry}
\usepackage[utf8]{inputenc}
\usepackage[T1]{fontenc}
\usepackage[english]{babel}
\usepackage{microtype}
\usepackage{amsmath}
\usepackage{amssymb}
\usepackage{amsthm}
\usepackage{amsfonts}
\usepackage{mathrsfs}
\usepackage{mathtools}
\usepackage{dsfont}
\usepackage[hidelinks]{hyperref}

\usepackage[
    backend=biber,
    style=numeric-comp,
    sorting=nyt,
    maxbibnames=99,
    maxcitenames=99
]{biblatex}

\DeclareDelimFormat{finalnamedelim}{\addcomma\space}
\numberwithin{equation}{section}

\theoremstyle{plain}
\newtheorem{theorem}{Theorem}[section]
\newtheorem{proposition}[theorem]{Proposition}
\newtheorem{lemma}[theorem]{Lemma}
\newtheorem{corollary}[theorem]{Corollary}
\theoremstyle{definition}
\newtheorem{definition}[theorem]{Definition}

\newtheorem{conjecture}[theorem]{Conjecture}

\theoremstyle{remark}
\newtheorem{remark}[theorem]{Remark}

\usepackage[normalem]{ulem}
\usepackage{comment}
\usepackage[dvipsnames]{xcolor}

\usepackage{tikz}
\usepackage{pgfplots}

\pgfplotsset{compat=newest}
\usepgfplotslibrary{fillbetween}
\usepackage{fancyhdr}
\title{$\mathbf{S}$-squareful values of binary quadratic forms}
\author{Sarah Sophie Pohl}
\date{}

\begin{document}

\maketitle

\begin{abstract}
Shute showed that the leading constant in the asymptotic formula for squareful values of the form $ax^2+by^2$ does not match with the conjecture established by Pieropan, Smeets, Tanimoto, and Várilly-Alvarado. In this paper, $S$-squareful values of this form are considered for $S$ the set of primes that divide $a$ and $b$. We prove an asymptotic formula for these numbers and show that the leading constant of the asymptotic formula still does not match the conjecture. Furthermore, we consider the updated conjecture established by Chow, Loughran, Takloo-Bighash, and Tanimoto regarding the leading constant for Campana points in the given setting and show that this constant is consistent with our computed constant. 
\end{abstract}

\tableofcontents 
\section{Introduction}
Quantitative questions regarding rational or integral points on algebraic varieties are central to arithmetic geometry. One may ask how rational points of bounded height on a projective variety are distributed. 
Manin's conjecture predicts the asymptotic behaviour of the number of rational points of bounded anticanonical height on suitable open subsets of Fano varieties. Peyre provided a geometric interpretation of the leading constant in the predicted asymptotic formula in \cite{13inShute}.
Related counting questions have also been considered for $S$-integral points, where $S$ is a finite set of places of $\mathbb{Q}$ containing the archimedean place. For example in \cite{Niedermowwe2010TheDO}, Niedermowwe establishes an asymptotic formula for the number of $S$-integral points of bounded height in the projective space with respect to a quadric. This demonstrates how imposing conditions at a prescribed set of primes can affect the asymptotic behaviour of the counting function.
Analogous counting problems can also be studied in the context of Campana points, which may be viewed as interpolating between rational and integral points.
In recent years, the quantitative study of Campana points has attracted considerable attention, with asymptotic formulas being established in a variety of settings. For example, Van Valckenborgh studies squareful points on linear hypersurfaces in \cite{22inshute}. This has been generalised in \cite{BY}, where the authors have established an asymptotic formula for $m$-ful points on linear hypersurfaces, where $m$ is a positive integer. In a related work, Balestrieri, Brandes, Kaesberg, Ortmann, Pieropan and Winter \cite{Brandes} consider $m$-ful points on diagonal hypersurfaces. 
Pieropan, Smeets, Tanimoto and Várilly-Alvarado \cite{16inshute} introduced a more general framework for the quantitative study of Campana points. In particular, they formulated a Manin-type conjecture predicting the asymptotic distribution of Campana points of bounded height on Fano Campana orbifolds. We refer to this conjecture as the PSTV-A conjecture.  
Since its formulation, several results have provided evidence for the predicted asymptotic behaviour of Campana points. Notable examples include the works of Santens \cite{Santens}, Xiao \cite{Xiao}, and Pieropan and Schindler \cite{Pieropan_Schindler}.
In \cite{StreeterNormforms}, Streeter establishes asymptotic formulas for Campana points and weak Campana points associated with norm forms. However, he also shows that the leading constant does not always agree with the prediction of the PSTV-A conjecture.
Further counterexamples have been established by Shute in \cite{shute}, where he computes the leading constant in an asymptotic formula for the number of squareful values represented by binary quadratic forms.
Squareful numbers arise naturally in this setting as a particular instance of Campana points.
More precisely, let $a,b \in \mathbb{N}$, coprime and squarefree, satisfying $a \equiv b \equiv 1 \mod 4$ and let $B\in \mathbb{R}_{>0}$. Shute defines the counting function
\begin{align*}
    N(B) \coloneqq  \frac{1}{2}\#\{(x,y)\in \mathbb{Z}_{\neq 0}^2 : \gcd(x,y)=1, |x|,|y|\leq B, ax^2+by^2 \text{ squareful}\}.
\end{align*}
Following the work of Browning and Van Valckenborgh in \cite{2inshute}, Shute proves an asymptotic, where the leading constant can be given explicitly and only depends on $a$ and $b.$
The counting problem may be viewed through the framework of Campana points.
Shute also shows that the leading constant differs from the constant obtained by the PSTV-A conjecture. More recently, Chow, Loughran, Takloo-Bighash and Tanimoto proposed a revised conjecture for the leading constant, incorporating Brauer-theoretic considerations, in \cite{newconjecture}.  We call this conjecture the CLT-BT conjecture.  This raises the question of whether the leading constant obtained by counting $S$-squareful values of binary quadratic forms agrees with the prediction of the revised conjecture. Motivated by \cite[Remark~6.6]{shute}, we investigate this question for the set $S$ of primes dividing $ab$.
We define $S$-squareful numbers in the following way.
\begin{definition}\label{def:S-squareful}
    Let $n \in \mathbb{N}$ and $S$ a set of primes. We say that $n$ is  ${S}$\emph{-squareful} if $p$ divides $n$ implies $p^2$ divides $n$ for all primes $p$ that are not in $S$. We denote the set of $S$-squareful numbers by $\mathbb{Z}_{S\text{-ful}}$.
\end{definition}
Let $a,b \in \mathbb{N}$ be squarefree and coprime such that $a\equiv b \equiv 1 \mod 4$. Furthermore, let $S$ be the set of primes dividing $ab$. 
For $B \in \mathbb{R}_{>0}$, we define the following counting function
\begin{align}\label{def:N_S(B)}
    N_S(B)\coloneqq  \frac{1}{2} \# \left\{ (x,y) \in \mathbb{Z}_{\text{prim}}^2 : |x|,|y| \leq B , ax^2 + by^2 \in \mathbb{Z}_{S\text{-ful}} \right\}.
\end{align}
Since the condition of being $S$-squareful is weaker than that of being squareful, the counting function $N_S(B)$ is at least as large as the corresponding counting function for squareful values. We therefore expect the leading constant in the asymptotic formula to be at least as large as in the original setting of Shute.
The main aim of this work is to determine the leading constant in the asymptotic formula for $N_S(B)$ by explicitly counting points on conics and to compare the resulting constant with the Brauer-theoretic prediction of the CLT-BT conjecture.
We establish the following theorem in Section~\ref{sec:asymptotic_for_N_S(B)}.
\begin{theorem}\label{thm:theorem_1}
For any $\epsilon>0$ there exists a constant $C\geq 0$ which depends only on $a,b$ such that $N_S(B) = CB + O(B^{43/44 + \epsilon})$. The constant $C$ can be given explicitly and only depends on $a$ and $b$.
The constant $C$ has the form 
\begin{align*}
    C= \sum_{\mathbf{p}^\mathbf{k}\vert a}\sum_{\mathbf{q}^\mathbf{l}\vert b} c_{(\mathbf{p}^\mathbf{k},\mathbf{q}^\mathbf{l})}
\end{align*}
with $c_{(\mathbf{p}^\mathbf{k},\mathbf{q}^\mathbf{l})}$ given explicitly in (\ref{eq:constant_c_k_l_2}).
\end{theorem}
As in Shute's original setting, the leading constant does not agree with the prediction of the PSTV-A conjecture. However, we show that it agrees with the constant predicted by the revised CLT-BT conjecture. We compute the predicted constant in Section~\ref{sec:new-conjecture} and establish the equality of the two constants in Section~\ref{seq:comparison_constant}.
\begin{theorem}\label{thm:constants_equal}
Fix $a,b$ as above. Then the leading constant of $N_S(B)$ agrees with the constant predicted by the CLT-BT conjecture.
\end{theorem}
Our approach builds on the framework developed in \cite{shute} and \cite{20inshute}.

\section{Notation}
Let $\mathbb{N} = \{1,2,3,....\}$.
For $k \in \mathbb{N}$ we define $\mathbb{Z}_{\text{prim}}^k\coloneqq \{\mathbf{x}\in \mathbb{Z}^k: \gcd(x_1,...,x_k)=1\}$. With $\sum_{(s,t)\mod n}^+ $ we denote a summation for $s,t \in \mathbb{N}, $ such that $ 1 \leq s,t \leq n $, subject to the condition $\gcd(s,t,n)=1$. We denote the set of squareful numbers by  $\mathbb{Z}_{\text{ful}}$ and the set of squarefree numbers by $\mathbb{Z}_{\text{free}}$.
For two functions $f(x),g(x)$ we write $f(x)=O(g(x))$ if there exists $x_1,c_1>0$ such that for all $x_1 < x$ we have $\vert f(x) \vert \leq c_1 g(x)$. Then we also write $f(x) \ll g(x).$\\
From now on let $K$ be a number field and let $\Omega_K$ denote the set of all places of $K$. For each non archimedean place $p \in \Omega_K$, we write $\nu_p: K \to \mathbb{Z} \cup \{\infty\}$ for the associated discrete valuation. Let $S_\infty \subset \Omega_K$ be the set of all archimedean places. Given a finite set $S \subset \Omega_K$ containing $S_\infty$, we denote the ring of $S$-integers in $K$ by   $\mathscr{O}_{K,S}$. For a non archimedean place $v\in\Omega_K$, we denote by $\mathscr{O}_v$ the ring of integers of the completion $K_v$ of $K$ at $v$.

\section{Campana Points}\label{sec:campana}
To make it self contained we introduce the concept of Campana points in this section. We use definitions from the first three sections of \cite{12inshute}. 

\begin{definition}
    A \emph{Campana orbifold} is a pair $(X,D)$, where $X$ is a smooth variety over $K$ and $$D=\sum_{\alpha \in \mathscr{A}}\epsilon_\alpha D_\alpha$$ is an effective Weil $\mathbb{Q}$-divisor of $X$ over $K$, where the $D_\alpha$  are prime divisors such that for all $\alpha \in \mathscr{A},$ either $\epsilon_\alpha = 1 $ or $\epsilon_\alpha$ takes the form $1-\frac{1}{m_\alpha}$ for some $m_\alpha \in \mathbb{Z}_{\geq 2}$ and the support $D_{\text{red}}=\sum_{\alpha \in \mathscr{A}}D_\alpha$ of $D$ has strict normal crossings on $X$.
    A Campana orbifold is called \emph{klt} if $\epsilon_\alpha \neq 1$ for all $\alpha \in \mathscr{A}$.
\end{definition}
\begin{definition}
    Let $(X,D)$ be a Campana orbifold.  A \emph{good integral model of} $(X,D)$ \emph{over} $\mathscr{O}_{K,S}$ is defined to be a pair $(\mathscr{X},\mathscr{D})$ where $\mathscr{X}$ is a regular, flat, proper $\mathscr{O}_{K,S}$-scheme which is a model of $X$ and $\mathscr{D}=\sum_{\alpha \in \mathscr{A}}\epsilon_\alpha \mathscr{D}_\alpha,$ where $\mathscr{D}_\alpha$ denotes the Zariski closure of $D_\alpha$ in $\mathscr{X}$. 
    \end{definition}
\begin{definition}
    Let $(X,D)$ be a Campana orbifold with  a good integral model $(\mathscr{X},\mathscr{D})$ over $\mathscr{O}_{K,S}$. A point $P\in (X \setminus D_{\text{red}})(K)$ is a \emph{Campana $\mathscr{O}_{K,S}$ point of $(\mathscr{X},\mathscr{D})$ } if for all $v \notin S$ and all $\alpha \in \mathscr{A}$, we have $n_v(\mathscr{D}_\alpha,P)=0$ if $\epsilon_\alpha = 1.$ If $\epsilon_\alpha \neq 1,$ so that $\epsilon_\alpha = 1- \frac{1}{m_\alpha}$ for some $m_\alpha \in \mathbb{Z}_{\geq 2}$, then either $n_v(\mathscr{D}_\alpha,P)=0$ or $n_v(\mathscr{D}_\alpha,P) \geq m_\alpha$.
    We denote the set of Campana $\mathscr{O}_{K,S}$-points of $(\mathscr{X},\mathscr{D})$ by $(\mathscr{X},\mathscr{D})(\mathscr{O}_{K,S})$. 
\end{definition}
Fix $a\equiv b \equiv 1 \mod 4$. We consider the Campana orbifold  $(X,D)=(\mathbb{P}^1,\frac{1}{2}V(ax^2+by^2))$ over $\mathbb{Q}$, together with the natural $\mathbb{Z}\left[\tfrac{1}{ab}\right]$-model 
$(\mathscr{X},\mathscr{D})$, by which we mean $\mathscr{X}=\mathbb{P}^1_{\mathbb{Z}[1/ab]} $ and $
\mathscr{D}=\tfrac{1}{2}V(ax^2+by^2) \subset \mathscr{X}.$
Furthermore we consider the height $H$ on $\mathbb{P}_{\mathbb{Q}}^1$ given by $H([x:y])=\max(\vert x \vert, \vert y \vert)$ for $(x,y)\in \mathbb{Z}_{\text{prim}}^2$.
In \cite{shute}, Shute works in the same setting but counts Campana $\mathbb{Z}$-points instead of Campana $\mathbb{Z}\left[\frac{1}{ab}\right]$-points.
For our setting we observe $D_{\text{red}}=V(ax^2+by^2)$ and therefore
$(X \setminus D_{\text{red}})(K)= \{[x:y]\in \mathbb{P}_{\mathbb{Q}}^1: ax^2+by^2 \neq 0\}$.
Thus we can write the set of Campana $\mathbb{Z}\left[ \frac{1}{ab}\right]$-points in the following way.
\begin{align*}
\left(\mathscr{X},\mathscr{D}\right)\left(\mathbb{Z}\left[ \frac{1}{ab}\right]\right) = \{[x:y]\in \mathbb{P}_{\mathbb{Q}}^1: ax^2+by^2 \neq 0, \text{ for all } v \notin S: \\n_v(\mathscr{D}_1,[x:y])=0 \text{ or }  n_v(\mathscr{D}_1,[x:y]) \geq 2  \} 
\end{align*}
Furthermore we obtain for all $p \notin S$ that $n_p(\mathscr{D}_1,P)= n_p(V(ax^2+by^2),P)=\nu_p(ax^2+by^2)$. 
Since $D=\frac{1}{2}V(ax^2+by^2) $ we have $ \epsilon_1=\frac{1}{2}, m_1=2$ and so we obtain that $\left(\mathscr{X},\mathscr{D}\right)\left(\mathbb{Z}\left[ \frac{1}{ab}\right]\right)$ is equal to the set of points $[x:y]\in \mathbb{P}_{\mathbb{Q}}^1$ such that $ ax^2+by^2\neq 0, ax^2+by^2 \in \mathbb{Z}_{S\text{-ful}}$.
Since we count these points with a bounded height, this motivates our counting problem
\begin{align*}
    &\#\left\{P\in \left(\mathscr{X},\mathscr{D}\right)\left(\mathbb{Z}\left[ \frac{1}{ab}\right]\right): H(P \leq B) \right\} \\&= \frac{1}{2}\#\{(x,y)\in \mathbb{Z}_{\text{prim}}^2: |x|,|y|\leq B, ax^2+by^2\neq 0, ax^2+by^2 \in \mathbb{Z}_{S\text{-ful}}\}\\
    &= N_S(B).
\end{align*}

\subsection{PSTV-A conjecture}
Before stating the PSTV-A conjecture, we recall some necessary background. In Section~\ref{sec:PSTVA}, we show that the constant predicted by the PSTV-A conjecture does not agree with the leading constant appearing in $N_S(B)$. Nevertheless, the PSTV-A conjecture provides an important context for understanding the CLT-BT conjecture, so we first review the relevant setup and notation.

\begin{definition}
Let $(X,D)$ be a Campana orbifold, let $S\subset \Omega_K$ be a finite set of places containing all archimedean places, and let $(\mathscr{X},\mathscr{D})$ be a good integral model of $(X,D)$ over $\mathscr{O}_{K,S}$. Furthermore, let $T\subset \Omega_K$ be a finite set of places. For every non-archimedean place $v\notin S$, let $(\mathscr{X},\mathscr{D})(\mathscr{O}_v)\subseteq X(K_v)$ denote the set of local Campana $\mathscr{O}_v$-points. We define the set of \emph{Campana $\mathbb{A}_K^T$-points} of $(\mathscr{X},\mathscr{D})$ by $(\mathscr{X},\mathscr{D})(\mathbb{A}_K^T)=\prod_{v\notin S\cup T}(\mathscr{X},\mathscr{D})(\mathscr{O}_v)\times\prod_{v\in S\setminus T}X(K_v),$
The set of \emph{adelic Campana points} of $(\mathscr{X},\mathscr{D})$ is defined by $(\mathscr{X},\mathscr{D})(\mathbb{A}_K)=\prod_{v\notin S}(\mathscr{X},\mathscr{D})(\mathscr{O}_v)\times\prod_{v\in S}X(K_v).$
\end{definition}
\begin{definition}
Let $V$ be a variety defined over a field $F$ of characteristic $0$ and let $A \subset V(F)$.
We say that $A$ is of \emph{type I} if $A \subset W(F)$ for some proper closed
subset $W$ of $V$ , i.e. $A$ is not dense in $V$.
We say that $A$ is of \emph{type II} if $A \subset \phi(V^\prime(F))$ for $V^\prime$ a variety and $\phi : V^\prime \rightarrow V$ a generically finite dominant morphism of degree $\geq 2$.
We say that $A$ is \emph{thin} if it is contained in a finite union of subsets of $V(F)$ of
types I and II.
\end{definition}
\begin{definition}\label{def:def2.6_in_Shute}
Let $(X,D)$ be a Campana orbifold over $K$ with a good integral model $(\mathscr{X},\mathscr{D})$ over $\mathscr{O}_{K,S}$ and $(\mathscr{L},\Vert \cdot \Vert )$ be an adelically metricized big and nef line bundle on $X$ with associated divisor class $[L]$. Let $H_\mathscr{L}:X(K)\rightarrow\mathbb{R}_{\geq 0}$ denote the corresponding height function. We recall that the \emph{effective cone} $\Lambda_{\text{eff}}$ of a variety $X$ is defined as $$\Lambda_{\text{eff}}=\{[D]\in \operatorname{Pic}(X) \otimes_\mathbb{Z} \mathbb{R}: [D]\geq 0\}.$$
Let $[K_X]$ denote the canonical divisor class.\\ Define $a = \inf\{t \in \mathbb{R}: t[L]+[K_X]+[D] \in \Lambda_{\text{eff}}\}$ which contains $a[L] + [K_X] + [D]$. Furthermore let $b$ be the codimension of the minimal supported face of $\Lambda_{\text{eff}}$ which contains $a[L] + [K_X] + [D].$
\end{definition}
With these preparations we can formulate the PSTV-A conjecture.
\begin{conjecture}\label{conj:pstva}
Let $(X,D)$ be a klt Campana orbifold, and suppose that $-(K_X+D)$ is ample. Assume that the set of Campana points $(\mathscr{X},\mathscr{D})(\mathscr{O}_{K,S})$ is not itself thin. Then there is a thin set $\mathscr{T}$ of Campana $\mathscr{O}_{K,S}$-points such that $$\#\{(\mathscr{X},\mathscr{D})(\mathscr{O}_{K,S})\setminus \mathscr{T}: H_\mathscr{L}(P)\leq B\} \sim c_{\text{PSTV-A}}B^{a}(\log B)^{b-1},$$ as $B \rightarrow \infty$, where $a,b$ are as in Definition~\ref{def:def2.6_in_Shute} and $c_{\text{PSTV-A}}>0$ is an explicit constant described in \cite{16inshute}.
\end{conjecture}
By \cite[Proposition~3.15 and Theorem~1.1]{12inshute}, the Campana $\mathbb{Z}\left[ \frac{1}{ab}\right]$-points of $\left(\mathbb{P}_\mathbb{Q}^1,\frac12 V(ax^2+by^2)\right)$
satisfy the Campana Hilbert property and hence do not form a thin subset of $\mathbb{P}_{\mathbb{Q}}^1$.
Thus the PSTV-A conjecture applies to our setting. 

\subsection{CLT-BT conjecture}
The leading constant of $N_S(B)$ can also be seen in the light of the CLT-BT conjecture. We give some more definitions to state the conjecture. For this, we follow \cite{newconjecture}.\\

As before let $(X,D)$ be a smooth klt Campana orbifold over $K$ with $D= \sum_{a \in \mathscr{A}} \epsilon_{\alpha}D_{\alpha}$ and let $(\mathscr{X},\mathscr{D})$ be a good integral model for $(X,D)$ over the ring of $S$-integers {$\mathscr{O}_{K,S}$} of $K$.

The constants $a((X, D),L), b(K,(X, D),L)$ and $\alpha((X, D),L)$ are the same as in Conjecture~\ref{conj:pstva} and are explained in \cite{16inshute}.

Let $\mathcal{L}$ be an adelically metricized log-adjoint rigid big and nef divisor so that $a((X, D),L)L+K_X+D$ is $\mathbb{Q}$-linearly equivalent to an effective $\mathbb{Q}$-divisor $E$ on $X$ with Iitaka dimension $0$. Furthermore, write $U=X\setminus \operatorname{Supp}(E)$ and $X^{\circ}= X\setminus \operatorname{Supp}(D) $.
For each $\alpha \in \mathscr{A}$ fix an adelic metrication for $D_{\alpha}$ and with $\mathscr{D}_\alpha$ we denote its adelically metricized divisor. Then $K_X$ is the adelic metrication for the canonical divisor, which is induced by adelic metrication of $\mathcal{L}$ and $\mathscr{D}_\alpha$. The space $U(F_v)$ can be endowed with a local Tamagawa measure $\tau_{U,v}$ by \cite{CT} for any $v \in \Omega_F$ and $H_v(D, \cdot)\tau_{U,v}$ is the associated local Tamagawa measure. Furthermore let $\lambda_v$$=L_v(\operatorname{Pic}(\overline{U}),1) $ and $\tau_{U, D}$ $=L^*(\operatorname{Pic}(\overline{U}),1)\prod_v \lambda_v^{-1}H_v(D,\cdot)\tau_{U,v}$ where $L^*(\operatorname{Pic}(\overline{U}),1)$ is the leading coefficient in the Laurent series expansion of $L(\operatorname{Pic}(\overline{U}),s)$ at $s=1$.
We define the set of local Campana points by
\begin{align*}
X^\circ(K_v)_\epsilon
=
\begin{cases}
\left\{x_v\in X^\circ(K_v):\substack{
n_v(\mathscr D_\alpha,x_v)=0
\text{ or }
n_v(\mathscr D_\alpha,x_v)\geq m_\alpha\\
\text{ for all }\alpha\in\mathscr A}
\right\}
& v\notin S\\
X^\circ(K_v)&v\in S,
\end{cases}
\end{align*}
with $n_v(\mathscr D_\alpha,x_v)$ the length of the $\mathcal O_v$-module $\mathcal O_v/x_v^*\mathcal I_{\mathscr D_\alpha}$, where $\mathcal I_{\mathscr D_\alpha}$ is the ideal sheaf of $\mathscr D_\alpha$.
At last, define $\hat{\tau}_{U,D}(b)$ as 
\begin{align}\label{def:tau}
     L^*(\operatorname{Pic}(\overline{U}),1)\prod_{v \in \Omega_K} \lambda_v^{-1}\int_{x_v \in X^\circ(K_v)_\epsilon} \frac{e^{2 \pi i \operatorname{inv}_v b (x_v)}H_v(D, x_v) d\tau_{U,v}}{H_v(a((X,D),L)L+K_X + D,x_v)}
\end{align}

\subsubsection{Brauer groups}
\begin{definition}
Let $D_\alpha$ be the divisor occurring as an irreducible component of $U\setminus X^\circ$ and  $\partial_D: \operatorname{Br} X^\circ \rightarrow H^1(\kappa(D),\mathbb{Q}/\mathbb{Z})$ the residue map at $D$. Then we de fine the \emph{Campana Brauer group of} $(X,D,L)$ as 
\begin{align*}
\operatorname{Br}((X, D),L)= \{\beta \in \operatorname{Br} X^\circ: m_\alpha \partial_{D_\alpha}(\beta)=0 \forall D_\alpha \subseteq U \setminus X^\circ\}.
\end{align*}
\end{definition}
Furthermore take $\operatorname{Br}_1((X, D),L)=\operatorname{Br}((X, D),L) \cap \operatorname{Br}_1 X^\circ$. 
Thus we can state the conjecture.
\begin{conjecture}\label{con:new_conjecture}
The conjectural leading constant $c(K,S,(\mathscr{X},\mathscr{D}),\mathscr{L})$ in the asymptotic formula for the
number of Campana points of bounded $\mathcal{L}$-height is
    \begin{align}
\label{eq:new-conjecture-constant}
\frac{\alpha((X,D),L)}{a((X,D),L)(b(K,(X,D),L)-1)!}
\sum_{\beta\in\operatorname{Br}_1((X,D),L)/\operatorname{Br}K}\hat{\tau}_{U,D}(\beta).
\end{align}
\end{conjecture}
 
\section{Asymptotic of $N_S(B)$}\label{sec:asymptotic_for_N_S(B)}

From now on fix $a,n\in \mathbb{N}$ with $a$ squarefree such that $a=\prod_{i=1}^n p_i$ with $p_i$ prime for $1\leq i \leq n$. Then for $\mathbf{k}\in \{0,1\}^n$ we set $\mathbf{p}^\mathbf{k}\coloneqq \prod_{i=1}^n p_i^{k_i}$. Analogously, fix $b,m \in \mathbb{N}$ with $b$ squarefree and coprime to $a$ and write $b$ as  $b=\prod_{i=1}^m q_i$ with $q_i$ prime for $1\leq i \leq m$. For $\mathbf{l}\in \{0,1\}^m$ we define $\mathbf{q}^\mathbf{l}\coloneqq \prod_{i=1}^m q_i^{l_i}$. We also always assume that $a \equiv b \equiv 1 \mod 4.$\\

In this section we find an asymptotic formula for $N_S(B)$ with $S$ the set of primes dividing $ab$. Furthermore, we prove Theorem~\ref{thm:theorem_1} and give an explicit form for the constant depending only on $a$ and $b$. 
In this subsection we compute an asymptotic formula for $N_S(B)$. We show the following identity. 
\begin{lemma} Let $S$ be the set of prime divisors of $a$ and $b$ and let $B\in \mathbb{R}_{>0}$.
Then we can write $N_S(B)$ in the following way
    \begin{align*}
    N_S(B) =  \sum_{\mathbf{k}\in \{0,1\}^n}\sum_{\mathbf{l}\in \{0,1\}^m} \frac{1}{2}\# \left\{ (x,y) \in \mathbb{Z}_{\operatorname{prim}}^2 :\substack{ |x|,|y| \leq B\\  ax^2 + by^2 = z \mathbf{p}^\mathbf{k} \mathbf{q}^\mathbf{l}, \    z  \in \mathbb{Z}_{\operatorname{ful}} } \right\}.
\end{align*}

\end{lemma}
\begin{proof}
By definition every $s \in \mathbb{Z}_{S\text{-ful}}$ can be written uniquely as $s=z \prod_{i=1}^n p_i^{k_{i}}\prod_{j=1}^m q_j^{l_{j}}$ where $k_i, l_j \in \mathbb{N}_{\geq0}$ for $1 \leq i \leq n, 1 \leq j \leq m$ and $z\in \mathbb{Z}_{\text{ful}}$ such that  $p_i \nmid z, q_j \nmid z$ for $1 \leq i \leq n, 1 \leq j \leq m$.
If $s$ is of the form $ax^2+by^2$ with $\gcd(x,y)=1$  it follows that $k_i,l_j \in \{0,1\}$ for $1\leq i \leq n$ and $1 \leq j \leq m$ because if there would be a prime $p$ such that $p\vert \mathbf{p}^\mathbf{k}$ with $p^2\vert z\mathbf{p}^\mathbf{k}\mathbf{q}^\mathbf{l}$ then we obtain that $p\vert x$ and $p\vert y$, which is a contradiction. The same holds if $p\vert \mathbf{q}^\mathbf{l} $.
By a similar argument observe that it is not possible that there exists a prime $p\vert ab$ which also divides $z$, since $z$ is squareful.
Thus, we can write any $S$-squareful number of the form $ax^2+by^2$ uniquely as $z \prod_{i=1}^n p_i^{k_{i}}\prod_{j=1}^m q_j^{l_{j}}$ with $z$ squareful and $k_i,l_j \in \{0,1\}$ for $1\leq i \leq n, 1 \leq j \leq m$.
\end{proof}
We define for $B \in \mathbb{R}_{>0}$ the counting function
\begin{align*}
\hat{N}_{(\mathbf{k},\mathbf{l})}(B) \coloneqq  \frac{1}{2}\# \{ (x,y) \in \mathbb{Z}_{\text{prim}}^2 : |x|,|y| \leq B ,  ax^2 + by^2 = z  \mathbf{p}^\mathbf{k} \mathbf{q}^\mathbf{l}, z \in \mathbb{Z}_{\text{ful}} \}.
\end{align*}
With this definition we can write $N_S(B)$ as
\begin{align}\label{eq:basic_N_S(B)}
     N_S(B) =  \sum_{\mathbf{k}\in \{0,1\}^n}\sum_{\mathbf{l}\in \{0,1\}^m} \hat{N}_{(\mathbf{k},\mathbf{l})}(B).
\end{align}
Every squareful number $z$ can be written uniquely as $u^2v^3$ with $u,v\in \mathbb{Z}$ and $v$ squarefree. Therefore we can write $ax^2+by^2$ uniquely as $u^2v^3\mathbf{p}^\mathbf{k} \mathbf{q}^\mathbf{l}$ with $v$ squarefree. Furthermore if a point $(x,y)$ is counted by $\hat{N}_{(\mathbf{k},\mathbf{l})}(B)$, the variable $u$ is uniquely determined up to the sign. Therefore, we can count the points as $(x,y,u)$ rather than just as $(x,y)$ and divide the number of points by two. Therefore
\begin{align*}
    \hat{N}_{(\mathbf{k},\mathbf{l})}(B)  =\frac{1}{4}\# \left\{ (x,y,u) \in \mathbb{Z}^3 : \substack{\gcd(x,y)=1, |x|,|y| \leq B,\\ ax^2 + by^2 = u^2v^3 \mathbf{p}^\mathbf{k}  \mathbf{q}^\mathbf{l}, \ v \in \mathbb{Z}_{\text{free}}} \right\}.
\end{align*}
Now we define for fixed $v\in \mathbb{Z}$ the counting function
\begin{align*}
    N_{(\mathbf{k},\mathbf{l}),v}(B) \coloneqq  \frac{1}{2} \# \left\{(x,y,u) \in \mathbb{Z}^3 : \substack{\gcd(x,y)=1, \vert x\vert, \vert y \vert \leq B,\\ ax^2+by^2 = u^2v^3 \mathbf{p}^\mathbf{k}  \mathbf{q}^\mathbf{l}} \right\}.
\end{align*}
Since $a,b,v$ are pairwise coprime and squarefree we get $\mu^2(abv)=1 $ if and only if $a,b,v$ are pairwise coprime and squarefree. Thus we obtain
\begin{align}\label{eq:hat_N}
    \hat{N}_{(\mathbf{k},\mathbf{l})}(B) = \frac{1}{2} \sum_{v=1}^\infty \mu^2(abv)N_{(\mathbf{k},\mathbf{l}),v}(B).
\end{align}
For a fixed $\delta >0$ we write $\hat{N}_{(\mathbf{k},\mathbf{l})}(B)$ as $Z_1+Z_2$ with 
$$Z_1 = \frac{1}{2} \sum_{v\geq B^{\delta}} \mu^2(abv)N_{(\mathbf{k},\mathbf{l}),v}(B) \text{ and } Z_2 = \frac{1}{2} \sum_{v< B^{\delta}} \mu^2(abv)N_{(\mathbf{k},\mathbf{l}),v}(B).$$
We show, that the size of $Z_1$ is smaller than the expected main term and we have to focus on the second sum. First we consider $Z_1$ and therefore $v \geq B^\delta$.
With the fact that $\vert x \vert, \vert y \vert \leq B $ we get that $u^2v^3 \mathbf{p}^\mathbf{k}  \mathbf{q}^\mathbf{l} = ax^2 +by^2 \ll B^2 $ and so $u \ll \frac{B}{(\mathbf{p}^\mathbf{k}  \mathbf{q}^\mathbf{l})^\frac{1}{2}v^\frac{3}{2}}$. This means we have $O\left( \frac{B}{(\mathbf{p}^\mathbf{k}  \mathbf{q}^\mathbf{l})^\frac{1}{2}v^\frac{3}{2}}\right)$ choices for $u$. \\
We consider Theorem \cite[Theorem~1.11]{1inshute}.

\begin{theorem}\label{thm:1_in_Shute}
    Let $Q\in \mathbb{Z}[x_1,...,x_{\nu}]$ be a quadratic polynomial, with $\nu \geq 2$ and $\epsilon>0$. Let $Q_0$ denote the quadratic part of $Q$. Assume that $Q$ is irreducible over $\mathbb{Q}$ and that $\operatorname{rank}(Q_0)\geq2$. Then 
    \begin{align}
        \#\{\mathbf{t}\in \mathbb{Z}^\nu : \vert \mathbf{t}\vert \leq B, Q(\mathbf{t})=0 \}=O(B^{\nu-2+\epsilon}).
    \end{align}
\end{theorem}
We define for fixed $u,v,\mathbf{k},\mathbf{l}$ the curve $\tilde{Q}_{u,v,\mathbf{k},\mathbf{l}}(\mathbf{x})\coloneqq  ax_0^2+bx_1^2-u^2v^3\mathbf{p}^\mathbf{k}\mathbf{q}^\mathbf{l}$.\\ Since $\tilde{Q}_{u,v,\mathbf{k},\mathbf{l}} \in \mathbb{Z}[x_1,x_2]$ is irreducible in $\mathbb{Q}$ and rank$(ax_0^2+bx_1^2)=2$ we obtain with Theorem~\ref{thm:1_in_Shute} for all $\epsilon>0$:
\begin{align}
    \# \{(x,y) \in \mathbb{Z}^2 : \vert x\vert, \vert y \vert  \leq B, ax^2 + by^2 = u^2v^3  \mathbf{p}^\mathbf{k}  \mathbf{q}^\mathbf{l} \} = O(B^\epsilon).
\end{align}
In our setting, we only consider $(x, y) \in \mathbb{Z}^2$ satisfying $|x|, |y| \leq B$. Since $u \ll \frac{B}{(\mathbf{p}^\mathbf{k}\mathbf{q}^\mathbf{l})^\frac{1}{2}v^\frac{3}{2}}$, when summing over $u$, we may drop the norm condition on $x$ and $y$ in our counting function, since it is automatically enforced by the size of the right-hand side. 
So we obtain
\begin{align*}
   \# \left\{ (x,y) \in \mathbb{Z}_{\mathrm{prim}}^2 : a x^2 + b y^2 = u^2 v^3 \mathbf{p}^{\mathbf{k}} \mathbf{q}^{\mathbf{l}} \right\} = O(B^\epsilon). 
\end{align*}
This leads to the estimate $ N_{(\mathbf{k},\mathbf{l}),v}(B) \ll \frac{B^{1+\epsilon}}{(\mathbf{p}^\mathbf{k}\mathbf{q}^\mathbf{l})^\frac{1}{2}\sqrt{v^3 }}$.
Since $v \geq B^\delta $ we get
\begin{align}\label{bound:Z_1}
    Z_1 \ll \frac{B^{1+\epsilon - \frac{1}{2}\delta }}{(\mathbf{p}^\mathbf{k}\mathbf{q}^\mathbf{l})^\frac{1}{2}}.
\end{align}
Thus the impact of $Z_1$ is very small and we consider $Z_2$.\\
We think of $ax^2+by^2=u^2v^3 \mathbf{p}^\mathbf{k} \mathbf{q}^\mathbf{l} $ as a conic with $a,b$ and $v^3 \mathbf{p}^\mathbf{k} \mathbf{q}^\mathbf{l}$ fixed.
We introduce the following definition from \cite{20inshute}:
\begin{definition}
    A norm $\Vert \cdot \Vert: \mathbb{R}^3 \rightarrow \mathbb{R}_{\geq 0}$ is called \emph{isometric to the supremum norm $\Vert \cdot \Vert_{\infty}$} if there exists an invertible matrix $g \in \operatorname{GL}_3(\mathbb{R})$ such that $\Vert \mathbf{x} \Vert = \Vert g\mathbf{x} \Vert_{\infty} $ for all $\mathbf{x} \in \mathbb{R}^3$. 
\end{definition}
We denote by $Q$ a non-singular quadratic form in three variables, coefficients in $\mathbb{Z}$ and $\Vert \cdot \Vert $ a norm isometric to the supremum norm. Furthermore we define
\begin{align*}
    N_{\Vert \cdot \Vert}(Q,B) = \# \{\mathbf{x} \in \mathbb{Z}^3 :\gcd(x_1,x_2)=1, Q(\mathbf{x})=0, \Vert \mathbf{x} \Vert \leq B \} 
\end{align*}
and $
    Q_{(\mathbf{k}, \mathbf{l}),v}(\mathbf{x}) =ax_1^2 + bx_2^2 - v^3 \mathbf{p}^\mathbf{k}  \mathbf{q}^\mathbf{l}  x_3^2.$\\
We set $\Vert \mathbf{x} \Vert_{\text{new}} \coloneqq \max\left(|x_1|,|x_2|,\sqrt{\frac{\mathbf{p}^\mathbf{k}  \mathbf{q}^\mathbf{l}v^3}{a+b}}|x_3|\right)$ and $\Vert \mathbf{x} \Vert \coloneqq \max(\vert x_1 \vert, \vert x_2 \vert)$. We obtain the following properties about the norm $\Vert \cdot \Vert_{\text{new}}$

\begin{lemma}
      The norm $\Vert \mathbf{x} \Vert_{\text{new}} $ is isometric to the supremum norm of $\emph{x}$ for all $\emph{x} \in \mathbb{R}^3$. And on the conic $Q_{(\mathbf{k}, \mathbf{l}),v}$ we have $\Vert \mathbf{x} \Vert_{\text{new}} = \Vert \mathbf{x} \Vert = \max(\vert x_1 \vert, \vert x_2 \vert)$  
\end{lemma}

\begin{proof}
 We set $C=  \sqrt{\frac{\mathbf{p}^\mathbf{k}  \mathbf{q}^\mathbf{l}}{a+b}}$ and  $g=\operatorname{diag}\left(1,1,Cv^{3/2}\right).$
    Then $g$ is invertible and $g \in \mathrm{GL}_3(\mathbb{R})$. Thus  
    \begin{align*}
        \Vert \mathbf{x} \Vert_{\text{new}} = \max\left(|x_1|,|x_2|,\sqrt{\frac{\mathbf{p}^\mathbf{k}  \mathbf{q}^\mathbf{l}v^3}{a+b}}|x_3|\right) =\Vert g\mathbf{x}\Vert_\infty.
    \end{align*}
     which gives us that $\Vert \cdot \Vert_{\text{new}}$ is isometric to the supremum norm. 
    By the choice of $C$ we obtain that $\max(\vert x_1 \vert, \vert x_2 \vert) = \max(\vert x_1 \vert, \vert x_2 \vert, C  v^{\frac{3}{2}} \vert x_3 \vert )$
     for all $\mathbf{x}$ on the conic $Q_{(\mathbf{k},\mathbf{l}),v}$. 
   \end{proof}
With the above notation we obtain $N_{(\mathbf{k},\mathbf{l}),v}(B) = N_{\Vert \cdot \Vert_{\text{new}}}(Q_{(\mathbf{k},\mathbf{l}),v},B).$
We continue with a linear change of variables that turns $Q_{(\mathbf{k},\mathbf{l}),v}$ into a quadratic form $Q$ that satisfies $Q(0,1,0)=0.$ To do this, we assume that there exists a nontrivial solution that gets counted by $N_{(\mathbf{k},\mathbf{l}),v}(B)$. We consider the smallest $(t_{12},t_{22},t_{32}) \in \mathbb{Z}^3 $ such that $ \gcd(t_{12},t_{22})=1, \vert \vert (t_{12},t_{22},t_{32}) \vert \vert_{\text{new}}  \leq B $ and $ a(t_{12})^2+b(t_{22})^2 = v^3\mathbf{p}^\mathbf{k}\mathbf{q}^\mathbf{l}(t_{32})^2. $ We can find $t_{11}, t_{21} \in \mathbb{Z}$ with $t_{11}t_{22}-t_{21}t_{12} = 1$ and $\vert t_{11} \vert, \vert t_{21} \vert \leq \max(\vert t_{12} \vert, \vert t_{22} \vert).$ Then we set 
\begin{align}\label{eq:M}
    M = \begin{pmatrix}
        t_{11} & t_{12} & 0 \\
        t_{21} & t_{22} & 0 \\
        0 & t_{32} & 1
    \end{pmatrix}.
\end{align}
We set $Q(\mathbf{x})= Q_{(\mathbf{k},\mathbf{l}),v}(M\mathbf{x})$ and 
$\Vert \mathbf{x}\Vert_{\text{new}}^{\prime}= \Vert M\mathbf{x}\Vert_{\text{new}}$.
By the choice of $t_{11}, t_{21}$ we get that the first $2 \times 2$ minor is in $\operatorname{SL}_2(\mathbb{Z})$ and so under this transformation the condition $\gcd(x_1,x_2)=1$ is preserved. Thus 
\begin{align}\label{eq:N_k,l,v,N_norm}
N_{(\mathbf{k},\mathbf{l}),v}(B) = \frac{1}{2}N_{\Vert \cdot \Vert_{\text{new}}}(Q_{(\mathbf{k},\mathbf{l}),v},B) =\frac{1}{2} N_{\Vert \cdot \Vert_{\text{new}}^{\prime}}(Q,B).
\end{align}
Explicitly we have
\begin{align*}   Q(\mathbf{x})=Q_{(\mathbf{k},\mathbf{l}),v}(M\mathbf{x})
 = \tilde{a}x_1^2 + \tilde{b}x_1x_2 + \tilde{e}x_2x_3+\tilde{f}x_3^2
 \end{align*}
with $ \tilde{a} = at_{11}^2+bt_{21}^2$, $\tilde{b} = 
2at_{11}t_{12}+2bt_{21}t_{22},\tilde{d}=0 ,\tilde{e}=- 2  v^3  \mathbf{p}^\mathbf{k}  \mathbf{q}^\mathbf{l}  t_{32} $ and $ \tilde{f}=- v^3 \mathbf{p}^\mathbf{k} \mathbf{q}^\mathbf{l}$. With this construction, we obtain further properties, which have to be fulfilled, that such a triple $(t_{12},t_{22},t_{32})$ exists. 
\begin{lemma}
    If there exists a solution $(t_{12},t_{22},t_{32}) \in \mathbb{Z}^3$ with the above conditions, then $2\nmid v$ and $v\mathbf{p}^\mathbf{k}\mathbf{q}^\mathbf{l}\equiv 1 \mod 4$
\end{lemma}
\begin{proof}
    Assume $2 \vert v$. By the construction of $t_{12},t_{22}$ and $t_{32}$ we obtain $   at_{12}^2 + bt_{22}^2 \equiv 0 \mod 8.$ Since $a \equiv b \equiv 1 \mod 4$ there are no possible $t_{12},$ $t_{22} \in \mathbb{Z}$ such that $\gcd(t_{12},t_{22})=1$ and $  at_{12}^2 + bt_{22}^2 \equiv 0 \mod 8.$
    It is left to show that $v\mathbf{p}^\mathbf{k}\mathbf{q}^\mathbf{l}\equiv 1 \mod 4$.
    Since $a\equiv b \equiv 1 \mod 4$ we obtain that  $v^3\mathbf{p}^\mathbf{k}\mathbf{q}^l\equiv 1 \mod 4$ or $v^3\mathbf{p}^\mathbf{k}\mathbf{q}^l\equiv 3 \mod 4 $. Assume that the latter is true, then by the construction of $t_{12},t_{22}$ and $t_{32}$ we obtain $at_{12}^2 + bt_{22}^2 \equiv  3 t_{32}^2 \mod 4$. Thus $t_{12}^2+t_{22}^2 \equiv 3 \mod 4$ or $t_{12}^2+t_{22}^2 \equiv 0 \mod 4$. Since $\gcd(t_{12},t_{22})=1$ no of the relations has a solutions in $t_{12},t_{22}$ and $t_{32}$. Thus we need $v^3\mathbf{p}^\mathbf{k}\mathbf{q}^\mathbf{l}\equiv1 \mod 4$.
\end{proof}
Furthermore set 
\begin{align*}
  \Pi \coloneqq  \begin{pmatrix}
    \tilde{b} & \tilde{e} & 0 \\
    -\tilde{a} & -\tilde{d} & -\tilde{f}\\
    0 & \tilde{b} & \tilde{e}
\end{pmatrix}.
 \end{align*}
Then we obtain $\operatorname{adj}(\Pi)\Pi \begin{pmatrix}
    s^2 \\ st \\ t^2
\end{pmatrix} = \operatorname{det}(\Pi) \begin{pmatrix}
    s^2 \\ st \\ t^2
\end{pmatrix},$ where $\operatorname{det}(\Pi) = \Delta_Q$.
We use the forms of 
\cite[Section~2]{20inshute} to describe $\underline{q}(s,t)\coloneqq \Pi \begin{pmatrix}
    s^2 \\ st \\ t^2
\end{pmatrix}$ for our case and define 
 \begin{align*}
     g(s,t) \coloneqq   \tilde{a}s^2  + \tilde{f}t^2  \text{ and }  L(s,t) \coloneqq  \tilde{b}s + \tilde{e}t. 
 \end{align*}
 Then we obtain $\underline{q} = (\underline{q_1},\underline{q_2},\underline{q_3})^T = (\underline{q_1}(s,t),\underline{q_2}(s,t),\underline{q_3}(s,t))^T $ with 
 \begin{align}\label{q_1}
     \underline{q_1}(s,t) &= sL(s,t) = (2at_{11}t_{12}+2bt_{21}t_{22})s^2  - 2  v^3 \mathbf{p}^\mathbf{k} \mathbf{q}^\mathbf{l}  t_{32}  s  t,
 \end{align}
 \begin{align}\label{q_2}
     \underline{q_2}(s,t) = -g(s,t) = -(at_{11}^2+bt_{21}^2)s^2 +  v^3   \mathbf{p}^\mathbf{k}  \mathbf{q}^\mathbf{l}  t^2 
 \end{align}
 and 
 \begin{align}\label{q_3}
     \underline{q_3}(s,t) = tL(s,t) = (2at_{11}t_{12}+2bt_{21}t_{22})s t -  2v^3 \mathbf{p}^\mathbf{k} \mathbf{q}^\mathbf{l}  t_{32}  t^2.
 \end{align}
 Set $\lambda = \gcd(\underline{q_1},\underline{q_2},\underline{q_3})$. 
 With these forms, we obtain the following lemma.
 \begin{lemma}\label{lemma:N_norm}
 Assume the previous notation. Then we obtain  
      \begin{align*}
     N_{\Vert \cdot \Vert_{\text{new}}^{\prime}}(Q,B) = 2\hat{N}(Q,B) + O(1)
 \end{align*} with 
 \begin{align}\label{lemma:N_hat}
     \hat{N}(Q,B) = \#\left\{(s,t) \in \mathbb{Z}_{\operatorname{prim}}^2 : t>0, \Vert \underline{q} \Vert_{\operatorname{new}}^{\prime} \leq \lambda B , \gcd\left(\frac{\underline{q_1}}{\lambda}, \frac{\underline{q_2}}{\lambda} \right)=1\right\}.
 \end{align}
 \end{lemma}
\begin{proof}
    We follow the proof of
    \cite[Lemma~3.1]{20inshute}. In $\mathbb{P}^2$ we define $C$ as the curve given by $Q=0$. We denote the point $(0,1,0)$ of $C$ by $\xi$. We define $L_\xi\coloneqq  \{\tilde{e}x_3 = \tilde{b}x_1\}$. This is the tangent line on $C$ in the point $\xi$. Furthermore, set $\mathscr{L}$ be the set of projective lines in $\mathbb{P}^2$ that go through $\xi$. By $\mathscr{L}(\mathbb{Q})$ we denote the corresponding subset of lines which are defined over $\mathbb{Q}$. By $U$ we denote the open subset formed by the deleting $\xi$ from $C$. Thus $U \subset C$. Finally set $\mathscr{U}\coloneqq  \mathscr{L}\setminus\{L_\xi\}$. Then $U(\mathbb{Q})$ and $\mathscr{U}(\mathbb{Q})$ are in bijection.
    We exclude the same cases as Sofos and follow the argument. These cases are $(s,t) \neq \frac{(b,e)}{\gcd(b,e)}$, $t=0$. These give us an error term of size $O(1)$. 
    Furthermore, since there is a bijection between the case $t>0$ and $t<0$ we only consider one case. We focus on $t>0$. 
    Now we consider the bijection between $U(\mathbb{Q})$ and $\mathscr{U}(\mathbb{Q})$. If we consider an element in $\mathscr{L}(\mathbb{Q})$ then it is given by $L_{s,t}\coloneqq  \{sx_1 = tx_3\}$ for $(s,t)\in \mathbb{Z}_{\text{prim}}^2$. The line $L_{s,t}$ intersects $C$ in the point $(x,y,z)$ if and only if $zg(s,t)+ytL(s,t)=0$ or $z=0$. If $z=0$ then $(x,y,z)=\xi$. We ignore this case. So we consider $zg(s,t)+ytL(s,t)=0$. By the equation of $L_{s,t}$ we get $-g(s,t)xt=-g(s,t)sz=syL(s,t)t$.
    If $\gcd(x,y)=1$ then $(x,y,z)\in \mathbb{Z}_{\text{prim}}^3$.
    Thus $(x,y,z)$ with $\gcd(x,y)=1$ represent a point in $C(\mathbb{Q})$ if and only if $(x,y,z)= \pm (sL(s,t)/\lambda,-g(s,t)/\lambda,tL(s,t)/\lambda)$ with $\gcd(sL(s,t)/\lambda,-g(s,t)/\lambda)$=1.  Therefore, we obtain the lemma.
\end{proof}
We take a sum over the possible values of $\lambda$ that can appear in (\ref{lemma:N_hat}). These are the values $\lambda=\gcd(\underline{q_1}(s,t),\underline{q_2}(s,t),\underline{q_3}(s,t))$ for coprime $s,t$. By
\cite[Lemma~2.2(i)]{20inshute}
we know that every $\lambda$ that can appear in $(\ref{lemma:N_hat})$ must fulfil $\lambda \vert  \frac{\Delta_Q}{\gcd(\tilde{b},\tilde{e})} $.  However, we must be careful, as we still wish to apply the concepts discussed in
\cite[Section~2 and Section~3]{20inshute}. Therefore, we obtain that $\hat{N}(Q,B)$ is equal to 
\begin{align*}
   \sum_{\lambda \vert  \frac{\Delta_Q}{\gcd(\tilde{b},\tilde{e})} } \#\left\{(s,t) \in \mathbb{Z}_{\text{prim}}^2 : \substack{t>0, \Vert \underline{q}(s,t) \Vert_{\text{new}}^{\prime} \leq \lambda B, \\   \gcd\left(\frac{\underline{q_1}(s,t)}{\lambda}, \frac{\underline{q_2}(s,t)}{\lambda} \right)=1, \lambda \vert \underline{q}}\right\}.
\end{align*}
By a Möbius inversion we obtain 
\begin{align*}
    \hat{N}(Q,B) = \sum_{\lambda \vert \Delta_Q}\sum_{r=1}^\infty \mu(r) \#\big\{(s,t) \in \mathbb{Z}_{\text{prim}}^2 : \substack{ t>0, \  \Vert \underline{q} \Vert_{\text{new}}^{\prime} \leq \lambda B ,\\ \lambda r \vert \gcd(\underline{q_1},\underline{q_2}), \  \lambda \vert \underline{q}}\big\}.
\end{align*}
 We define for $T \geq 1$ and $n \in \mathbb{N}$ the function
 \begin{align*}
     M^*_{\sigma, \tau }(T,n) \coloneqq  \# \{(s,t) \in \mathbb{Z}_{\text{prim}}^2 :
     \substack{ (s,t) \equiv (\sigma, \tau) \mod n,  t>0,\\\Vert \underline{q}(s,t) \Vert_{\text{new}}^{\prime} \leq T}  \}.
 \end{align*}
Then we obtain the following analogously to \cite[Lemma~3.2]{20inshute}.
 \begin{align*}
     \hat{N}(Q,B) = \sum_{\lambda | \frac{\Delta_Q}{\gcd(\tilde{b},\tilde{e})} } \sum_{r=1}^\infty \mu(r) \sum_{\substack{(\sigma, \tau) \mod r\lambda \\ \lambda \vert \underline{q}(\sigma, \tau),\\ r\lambda \vert (\underline{q_1}(\sigma, \tau),\underline{q_1}(\sigma, \tau))} }^+ M_{\sigma, \tau}^*(B\lambda , r \lambda ). 
 \end{align*} 
In the following we explain why we do not need to sum over all $r$ that divide $\left(\frac{\underline{q_1}}{\lambda}, \frac{\underline{q_2}}{\lambda} \right)$, it is enough to sum over all divisors of $\frac{\Delta_Q}{\gcd(\tilde{b},\tilde{e})}$. To do so we argue if $r \vert (\underline{q_1},\underline{q_2})$ then $r | (\underline{q_1},\underline{q_2},\underline{q_2})$. 
We only consider those $r$, which are squarefree since otherwise $\mu(r)$ is zero. 
 Therefore, it is sufficient to demonstrate that, for any prime $p$ with $p \vert (\underline{q_1},\underline{q_2})$, the prime $p$ also divides $\underline{q_3}$.
Assume there exist such a prime $p$ with $p|(\underline{q_1},\underline{q_2})$. Since $\underline{q_1}= sL(s,t)$ we get that either $p|s$ or $p|L(s,t)$. Assume $p|L(s,t)$. Then $p|tL(s,t)=\underline{q_3}$. So we know assume the other case, $p|s.$ Since $p|\underline{q_2}=-(at_{11}^2+bt_{21}^2)s^2 +  v^3t^2 $ and $\gcd(s,t)=1$ we get $p|v^3$. 
But both parts of $L(s,t)= (2at_{11}t_{12}+2bt_{21}t_{22})s -  2  v^3   t_{32} t $ are then divisible by $p$. This means $p|L(s,t)$ and therewith we are again in the first case. In total this means $p|\underline{q_3}$. 
Thus $r\vert (\underline{q_1}(\sigma, \tau), \underline{q_2}(\sigma, \tau),\underline{q_3}(\sigma, \tau))$. Therefore $r \vert \gcd(\underline{q_1}(\sigma, \tau), \underline{q_2}(\sigma, \tau),\underline{q_3}(\sigma, \tau))$. by the same argumentation as earlier we know $\gcd(\underline{q_1}(\sigma, \tau), \underline{q_2}(\sigma, \tau),\underline{q_3}(\sigma, \tau)) \vert \frac{\Delta_Q}{\gcd(\tilde{b},\tilde{e})}$.
Thus we ca restrict the $r$-sum to divisors of $\frac{\Delta_Q}{\gcd(\tilde{b},\tilde{e})}$.
Therefore, 
\begin{align}\label{eq:N_hat_2}
    \hat{N}(Q,B) = \sum_{\lambda |  \frac{\Delta_Q}{\gcd(\tilde{b},\tilde{e})} } \sum_{r\vert \frac{\Delta_Q}{\gcd(\tilde{b},\tilde{e})}} \mu(r) \sum_{\substack{(\sigma, \tau) \mod r\lambda \\ \lambda \vert \underline{q}(\sigma, \tau),\\ r\lambda \vert (\underline{q_1}(\sigma, \tau),\underline{q_1}(\sigma, \tau))} }^+ M_{\sigma, \tau}^*(B\lambda , r \lambda ).
\end{align}
Furthermore, we can determine the $\gcd(\tilde{b},\tilde{e})$ explicitly with the following Lemma.
\begin{lemma}\label{lemma:gcd}Let 
     $\tilde{b} = 
2at_{11}t_{12}+2bt_{21}t_{22},\tilde{e}=- 2  v^3  \mathbf{p}^\mathbf{k}  \mathbf{q}^\mathbf{l}  t_{32} $ with 
$(t_{12},t_{22},t_{32}) \in \mathbb{Z}^3 $ such that $ \gcd(t_{12},t_{22})=1,$ and $ a(t_{12})^2+b(t_{22})^2 = v^3\mathbf{p}^\mathbf{k}\mathbf{q}^\mathbf{l}(t_{32})^2. $ Furthermore, let $t_{11}, t_{21} \in \mathbb{Z}$ with $t_{11}t_{22}-t_{21}t_{12} = 1$ and $\vert t_{11} \vert, \vert t_{21} \vert \leq \max(\vert t_{12} \vert, \vert t_{22} \vert).$ Then
$\gcd(\tilde{b},\tilde{e})=2\mathbf{p}^\mathbf{k}\mathbf{q}^\mathbf{l}.$
\end{lemma}
\begin{proof}
Since $\mathbf{p}^\mathbf{k}\vert t_{22}$ and $\mathbf{q}^\mathbf{l}\vert t_{12}$ we get that $2\mathbf{p}^\mathbf{k}\mathbf{q}^\mathbf{l} \vert \gcd(\tilde{b},\tilde{e})$. It is left to show that there is no $p\vert v^3t_{32}^2$ such that $p\vert \tilde{b}$.
Assume contra wise. We know that $v$ and $t_{32}$ are both coprime to $a$ and $b$.
We observe 
\begin{align*}
    t_{12}(2at_{11}t_{12}+2bt_{21}t_{22})  &= 2(at_{11}t_{12}^2+b(t_{11}t_{22}-1)t_{22})\\
    &= 2(t_{11}v^3\mathbf{p}^\mathbf{k}\mathbf{q}^\mathbf{l}(t_{32})^2-bt_{22})
\end{align*} and in the same way we obtain
\begin{align*}
    t_{22}(2at_{11}t_{12}+2bt_{21}t_{22}) =2(t_{21}v^3\mathbf{p}^\mathbf{k}\mathbf{q}^\mathbf{l}(t_{32})^2+at_{12}). 
\end{align*}
Thus $t_{12}\tilde{b}=2(-t_{11}\tilde{e} -bt_{22})$ and $t_{22}\tilde{b}=2(-t_{21}\tilde{e}+at_{12})$\\
Since $p\vert \tilde{b}$ and $p\vert \tilde{e}$ we get that $p \vert t_{22}$ and $p \vert t_{12}$, this is a contradiction to $\gcd(t_{12},t_{22})=1.$
\end{proof}
In the following we find a bound for $M_{\sigma, \tau}^*(B\lambda , r \lambda )$. Let \begin{align*}
    V\coloneqq  \{(s,t) \in \mathbb{R}^2: t>0, \Vert \underline{q} \Vert_{\text{new}}^\prime \leq 1\}
\end{align*}
and 
$K_0= \sup_{\mathbf{x} \neq 0} \left( 1 + \frac{\Vert \mathbf{x} \Vert_{\infty}}{\Vert \mathbf{x} \Vert_{\text{new}}^{\prime}}\right)$. 
With $\langle Q \rangle$ we denote the maximum modulus of the coefficients of $Q$. 
Then by \cite[Lemma~4.1]{20inshute}, $V$ is contained in a rectangle with 
\begin{align}\label{bound:s_t}
    \vert s \vert, \vert t \vert \ll \langle Q \rangle \left( \frac{K_0}{\Delta_Q} \right)^\frac{1}{2}. 
\end{align}
Furthermore, if $\partial V$ is the boundary of $V$, then 
\begin{align*}
    \vert \partial V\vert \ll \langle Q \rangle \left( \frac{K_0}{\Delta_Q} \right)^\frac{1}{2}. 
\end{align*}
And any line which is parallel to the coordinate axes intersects $V$ in a set of points, if not empty,  of most $O(1)$ intervals. All the implied constants are absolute. 
We define for $T \in \mathbb{R}_{\geq 1}$, $n,\sigma, \tau \in \mathbb{N}$ with $\gcd(\sigma,\tau,n)=1$
\begin{align*}
     M_{\sigma, \tau }(T,n) \coloneqq  \# \{(s,t) \in \mathbb{Z}^2 : (s,t) \equiv (\sigma, \tau) \mod n,  t>0, \Vert \underline{q} \Vert_{\text{new}}^{\prime} \leq T  \}.
\end{align*}
With this definition we obtain the following Lemma.
\begin{lemma}\label{lemma:sofos}
    For any $T, n, \sigma, \tau$ as above with $\gcd(\sigma, \tau,n)=1$ and\\ $n \vert (\underline{q_1}(\sigma, \tau), \underline{q_2}(\sigma, \tau))$, we obtain 
    \begin{align*}
        M_{\sigma,\tau}^*(T,n)= \sum_{\substack{1 \leq m \leq (2TK_0/n)^\frac{1}{2}\\\gcd(m,n)=1}} \mu(m)M_{\overline{m}\sigma, \overline{m}\tau}\left( \frac{T}{m^2},n \right),
    \end{align*}
    where $\overline{m}$ denotes the inverse of $m \mod n$.
\end{lemma}\begin{proof}
    By a Möbius inversion of (\ref{eq:N_hat_2}) we obtain 
    \begin{align*}
        M_{\sigma,\tau}^*(T,n)= \sum_{\substack{m=1\\\gcd(m,n)=1}}^\infty \mu(m)M_{\overline{m}\sigma, \overline{m}\tau}\left( \frac{T}{m^2},n \right).
    \end{align*}
    Now we assume $m> \left( \frac{2K_0T}{n}\right)^\frac{1}{2}.$\\
    Since $M_{\overline{m}\sigma, \overline{m}\tau}\left( \frac{T}{m^2},n \right)$ is equal to 
    \begin{align*}
        \#\left\{ (s,t) \in \mathbb{Z}^2: (s,t) \equiv (\overline{m}\sigma,\overline{m}\tau ) \mod n, t>0, \Vert \underline{q}(s,t)\Vert_{\text{new}}^{\prime} \leq \frac{T}{m^2}\right\} 
    \end{align*}
    the condition $\Vert \underline{q}(s,t)\Vert_{\text{new}}^{\prime} \leq \frac{T}{m^2}$ implies $\Vert \underline{q}(s,t)\Vert_{\text{new}}^{\prime} < \frac{n}{2K_0}\leq \frac{n}{2}$.\\
    But $\Vert \underline{q}(s,t)\Vert_{\text{new}}^{\prime} < \frac{n}{2}$ together with $\gcd(\sigma, \tau, n)=1$ and $n \vert (\underline{q_1}(\sigma, \tau), \underline{q_2}(\sigma, \tau)) $ implies $t=0$, which concludes the proof.
\end{proof}

The last preparation for the bound of $M_{\sigma, \tau}^*(B\lambda , r \lambda )$ is a bound for $ M_{\sigma, \tau }(T,n)$. We obtain this bound by \cite[Lemma~4.3]{20inshute}.
\begin{lemma}\label{lemma:M_sigma_tau}
    For any $T\in \mathbb{R}_{>1}, n, \sigma, \tau \in \mathbb{N}$ with $\gcd(\sigma, \tau, n)=1$ we obtain 
    \begin{align*}
         M_{\sigma, \tau }(T,n)= \operatorname{vol}(V)\frac{T}{n^2}+ O\left(1+ \frac{(K_0T)^\frac{1}{2}}{n}\frac{\langle Q \rangle }{\vert \Delta_Q\vert^\frac{1}{2} }\right). 
    \end{align*}
\end{lemma}
With these preparations we obtain with Lemma~\ref{lemma:sofos} for (\ref{eq:N_hat_2}) that $\hat{N}(Q,B)$ is equal to 
\begin{align*}
      \sum_{\lambda |  \frac{\Delta_Q}{2\mathbf{p}^\mathbf{k}\mathbf{q}^\mathbf{l}} } \sum_{r\vert \frac{\Delta_Q}{2\mathbf{p}^\mathbf{k}\mathbf{q}^\mathbf{l}}} \mu(r) \sum_{\substack{(\sigma, \tau) \mod r\lambda \\ \lambda \vert \underline{q}(\sigma, \tau),\\ r\lambda \vert (\underline{q_1}(\sigma, \tau),\underline{q_1}(\sigma, \tau))} }^+ \sum_{\substack{1 \leq m \leq (2B K_0/r)^\frac{1}{2}\\\gcd(m,r\lambda)=1}} \mu(m)M_{\overline{m}\sigma, \overline{m}\tau}\left( \frac{B\lambda}{m^2},r\lambda \right).
\end{align*}
As in \cite{20inshute} we set $\mathscr{L}\coloneqq  \frac{(K_0B)^\frac{1}{2}}{r(\lambda)^\frac{1}{2}}\frac{\langle Q \rangle }{\vert \Delta_Q\vert^\frac{1}{2} } $. 
For $M_{\overline{m}\sigma, \overline{m}\tau}\left( \frac{B \lambda}{m^2},r\lambda  \right)$ we obtain together with Lemma~\ref{lemma:M_sigma_tau} the following bound. 
\begin{align*}
    M_{\overline{m}\sigma, \overline{m}\tau}\left( \frac{B\lambda}{m^2},r\lambda \right) = \begin{cases}
        \frac{\text{vol}(V)B}{m^2r^2\lambda}+ O \left( \frac{\mathscr{L}}{m}\right) & \text{ if } m \leq \mathscr{L}\\ 
        O(1) & \text{ otherwise.}
    \end{cases}
\end{align*}
Furthermore, we define the function 
\begin{align*}
    \rho^{**}(r,\lambda)=\#\left\{(s,t)\in [0,r\lambda)^2: \substack{ r\lambda \vert (\underline{q_1}(s,t),\underline{q_2}(s,t)),\\ \lambda \vert \underline{q}(s,t), \  \gcd(r\lambda,s,t)=1}\right\}.
\end{align*}
Then we obtain that $\hat{N}(Q,B)$ is equal to
\begin{align}\label{eq:N_hat_3}
      \sum_{\lambda |  \frac{\Delta_Q}{2\mathbf{p}^\mathbf{k}\mathbf{q}^\mathbf{l}} } \sum_{r\vert \frac{\Delta_Q}{2\mathbf{p}^\mathbf{k}\mathbf{q}^\mathbf{l}}} \mu(r) \rho^{**}(r,\lambda) \sum_{\substack{1 \leq m \leq  \mathscr{L}\\\gcd(m,r\lambda )=1}} \left( \frac{\mu(m)\text{vol}(V)B}{m^2r^2\lambda} +  O\left( \frac{\mu(m)\mathscr{L}}{m}\right)\right)\\ +\sum_{\lambda | \Delta_Q} \sum_{r\vert \lambda} \mu(r) \rho^{**}(r,\lambda) \sum_{\substack{\mathscr{L} < m \leq (2BK_0/r)^\frac{1}{2}\\\gcd(m,r\lambda)=1}} O(1). \notag
\end{align}

To obtain the error terms we need several preparations, we need to compute the different values for $\rho^{**}$ and to compute the volume of $V$.
\subsubsection{Properties of the function $\rho^{**}$}
\begin{lemma}\label{lemma:rho_star_star_multiplicative}
The function $\rho^{**}$ is multiplicative in $(r, \lambda)$, specifically,
\begin{align*}
   \rho^{**}(r_1 r_2, \lambda_1 \lambda_2) = \rho^{**}(r_1, \lambda_1) \rho^{**}(r_2, \lambda_2) 
\end{align*}
if $\gcd(r_1 \lambda_1, r_2 \lambda_2) = 1$.
\end{lemma}
\begin{proof}
Consider $(r_1, \lambda_1), (r_2, \lambda_2) \in \mathbb{N}^2$ such that $\gcd(r_1 \lambda_1, r_2 \lambda_2) = 1$. 
By the Chinese Remainder Theorem, for every pair $(s, t) \bmod r_1\lambda_1r_2\lambda_2$, there is a unique pair
$((s_1, t_1) \bmod r_1\lambda_1, (s_2, t_2) \bmod r_2\lambda_2)$
such that $s \equiv s_1 \bmod r_1 \lambda_1$, $s \equiv s_2 \bmod r_2 \lambda_2$ and similarly for $t$.
We record the following properties of $\rho^{**}$. 
First, the condition
$r_1 \lambda_1 r_2 \lambda_2 \vert  \bigl(\underline{q_1}(s,t), \underline{q_2}(s,t)\bigr)
$ holds if and only if $
    r_1 \lambda_1 \vert  \bigl(\underline{q_1}(s_1,t_1),  \underline{q_2}(s_1,t_1)\bigr)
   $ and $r_2 \lambda_2 \vert  \bigl(\underline{q_1}(s_2,t_2),  \underline{q_2}(s_2,t_2)\bigr).$\\
Moreover,$\lambda_1 \lambda_2 \vert  \underline{q}(s,t)$ if and only if $ \lambda_1 \vert  \underline{q}(s_1,t_1)$ and $\lambda_2 \vert  \underline{q}.(s_2,t_2).$\\
Finally, $\gcd(r_1 \lambda_1 r_2 \lambda_2, s, t) = 1$
if and only if $\gcd(r_1 \lambda_1, s_1, t_1) = 1$ and $\gcd(r_2 \lambda_2, s_2, t_2) = 1.$
Thus we obtain a bijection between the set counted by $\rho^{**}(r_1 r_2, \lambda_1 \lambda_2)$ and the Cartesian product of the sets counted by $\rho^{**}(r_1, \lambda_1)$ and $\rho^{**}(r_2, \lambda_2)$. Thus 
$\rho^{**}(r_1 r_2, \lambda_1 \lambda_2) = \rho^{**}(r_1, \lambda_1) \rho^{**}(r_2, \lambda_2)$ whenever\\ $\gcd(r_1 \lambda_1, r_2 \lambda_2) = 1$.
\end{proof}
In order to compute the constant explicitly, we consider the different possible primes dividing $\frac{\Delta_Q}{\gcd(\tilde{b},\tilde{e})} = \frac{4abv^3\mathbf{p}^\mathbf{k}\mathbf{q}^\mathbf{l}}{2\mathbf{p}^\mathbf{k}\mathbf{q}^\mathbf{l}  } = 2abv^3$ for the value of $\rho^{**}$.\\
To compute $\rho^{**}$, we introduce some notation. Let 
\begin{align*}
    \rho^*(n)\coloneqq  \#\{(s,t)\in [0,n)^2: n\vert \underline{q}(s,t), \gcd(n,s,t)=1\}.
\end{align*}
This function is considered in \cite{20inshute}. By the definition of $\underline{q}(s,t)$ we obtain 
\begin{align*}
    \rho^*(n)= \#\{(s,t)\in [0,n)^2: n\vert(L(s,t),g(s,t)), \gcd(n,s,t)=1\}.
\end{align*}
To compute the values of $\rho^{**}$ we need the following lemma as a preparation.
\begin{lemma}\label{lemma:preparation}
    Let $p$ be a prime and assume $p\nmid \tilde{a}, p \nmid \tilde{b}, p \nmid \tilde{e}$ and $p \nmid \tilde{f}$. If $p\nmid s$ and $p\nmid t$ then  $p^{i} \vert (L(s,t),g(s,t))$ has no solution in $(s,t)\in [0,p^{i})^2$ for $i \geq 2$. 
\end{lemma}
\begin{proof}
    By the assumptions we know that $\tilde{a},\tilde{b},\tilde{e},\tilde{f}$ are invertible modulo $p^{i}$. Assume there exists $(s,t)\in [0,p^i)^2$ with $p\nmid s, p \nmid t$ such that $p^{i}\vert (L(s,t),g(s,t))$. Since $p^{i} \vert L(s,t)$ we obtain 
    \begin{align}\label{eq:L}
        s \equiv -\tilde{e}t\tilde{b}^{-1} \mod p^{i}.
    \end{align}
 Since $p^{i} \vert g(s,t)$ we get $\tilde{a}s^2+\tilde{f}t^2 \equiv 0 \mod p^{i}$. Using (\ref{eq:L}) we obtain 
 $\tilde{a}(-\tilde{e}t\tilde{b}^{-1})^2+\tilde{f}t^2 \equiv 0 \mod p^{i}. $
 Since $p \nmid t$ we get 
$ \tilde{a}\tilde{e}^2+\tilde{b}^2\tilde{f} \equiv 0 \mod p^{i}.$
Using the definitions of $\tilde{a},\tilde{b},\tilde{e}$ and $\tilde{f}$ we get 
  $(at_{11}^2+bt_{21}^2)(2v^3\mathbf{p}^\mathbf{k}\mathbf{q}^\mathbf{l}t_{32})^2 -(2at_{11}t_{12}+2bt_{21}t_{22})^2 (v^3\mathbf{p}^\mathbf{k}\mathbf{q}^\mathbf{l})\equiv 0 \mod p^{i}.$
 We recall the definition of $t_{12},t_{22}$ and $t_{32}$ and get that $at_{12}^2+bt_{22}^2\equiv v^3\mathbf{p}^\mathbf{k}\mathbf{q}^\mathbf{l}t_{32}^2 \mod p^{i}$. Thus we obtain
$(at_{11}^2+bt_{21}^2)4v^3\mathbf{p}^\mathbf{k}\mathbf{q}^\mathbf{l}(at_{12}^2+bt_{22}^2) -(2at_{11}t_{12}+2bt_{21}t_{22})^2 (v^3\mathbf{p}^\mathbf{k}\mathbf{q}^\mathbf{l})\equiv 0 \mod p^{i}.$
 Since $p \nmid 2v^3\mathbf{p}^\mathbf{k}\mathbf{q}^\mathbf{l}$ we obtain $ab(t_{11}t_{22}-t_{12}t_{21})^2 \equiv 0 \mod p^{i}.$
 Since $t_{11}t_{22}-t_{12}t_{21}=1$ by the definition of $t_{11},t_{21}$ we conclude $p^{i} \vert ab$. Since $a$ and $b$ are coprime and squarefree this is a contradiction. 
\end{proof}
In the following lemmas we compute $\rho^{**}$ for the different primes  which divide $2abv^3$.
In the proofs of the lemmas we use the construction (\ref{eq:M}) of $t_{11},t_{12},t_{12},t_{22}$ and $t_{32}$ and the definitions (\ref{q_1}),(\ref{q_2}) and (\ref{q_3}) of $\underline{q}(s,t)$. 
We start with the case that $p \vert \frac{a}{\mathbf{p}^\mathbf{k}}$ and obtain the following lemma for the values of $\rho^{**}$. 
\begin{lemma}\label{lemma:p_mid_a}
    Assume $p \vert \frac{a}{\mathbf{p}^\mathbf{k}}$. Then we get if $\left( \frac{v^3 \mathbf{p}^\mathbf{k}\mathbf{q}^\mathbf{l}b}{p} \right) = 1$ that 
    \begin{align*}
        \rho^{**}(p^j,p^k)= \begin{cases}
            1 & \text{ if } (j,k)=(0,0)\\
            p-1 & \text{ if } (j,k) \in \{(1,0),(0,1)\}\\0 & \text{ if } (j,k)=(1,1).
        \end{cases}
    \end{align*}
    The same holds for $p \vert \frac{b}{\mathbf{q}^\mathbf{l}}$ if  $\left( \frac{v^3 \mathbf{p}^\mathbf{k}\mathbf{q}^\mathbf{l}a}{p} \right) = 1$. 
\end{lemma}
\begin{proof}
  We need $\left( \frac{v^3 \mathbf{p}^\mathbf{k}\mathbf{q}^\mathbf{l}b}{p} \right) = 1$, otherwise there exist no $t_{12},t_{22},t_{32}\in \mathbb{Z}$ such that $at_{12}^2 + bt_{22}^2 = v^3 \mathbf{p}^\mathbf{k}\mathbf{q}^\mathbf{l}t_{32}^2$ with $\gcd(t_{12},t_{22})=1$.
  We consider $\rho^{**}(1,p)=\rho^{*}(p)$. We count $(s,t)\in [0,p)^2 $ such that $p \vert L(s,t),g(s,t)$ with $\gcd(p,s,t)=1$. We assume $p \nmid s, p \nmid t$. Then the relations
  \begin{align}\label{F1GL1}
      bt_{21}^2s^2 \equiv v^3\mathbf{p}^\mathbf{k}\mathbf{q}^\mathbf{l}t^2 \mod p
  \end{align} and 
  \begin{align}\label{F1GL2}
      2bt_{21}t_{22}s \equiv 2v^3\mathbf{p}^\mathbf{k}\mathbf{q}^\mathbf{l}t_{32}t \mod p 
  \end{align} has to be fulfilled.  We get that $p \nmid t_{21}$. By the definition of $t_{22},t_{32}$ we observe that $p \nmid t_{22}$ and $p \nmid t_{32}$. Thus we get from (\ref{F1GL2}) that $s \equiv v^3 \mathbf{p}^\mathbf{k}\mathbf{q}^\mathbf{l}t_{32}t_{21}^{-1}t_{22}^{-1}b^{-1}t \mod p$. Since $p \nmid  v^3 \mathbf{p}^\mathbf{k}\mathbf{q}^\mathbf{l}t_{32}t_{21}^{-1}t_{22}^{-1}b^{-1}$ we get for every $s \in (\mathbb{Z}/p\mathbb{Z})^\times $ a fixed $t$ which is different for every $s$. With (\ref{F1GL1}) we obtain $v^3\mathbf{p}^\mathbf{k}\mathbf{q}^\mathbf{l} b^{-1}\equiv t_{32}^2(t_{22}^2)^{-1} \mod p$ which is fulfilled by the definition of $t_{12},t_{22},t_{32}$.  Thus if $p \nmid t_{21}$ we obtain $p-1$ solutions. 
  If $p\nmid s$ and $p \vert t$ we obtain $p-1$ solutions. 
  If $p \vert s$ and $p \nmid t$ there are no possible solutions.
  In total we obtain $\rho^{**}(1,p)=p-1$. 
  The same holds for $\rho^{**}(p,1)$ because we have already shown, that if a prime $p$ divides $\underline{q_1}(s,t)$ and $\underline{q_2}(s,t)$, then $p$ also divides $\underline{q_3}(s,t)$.  
  Next we consider $\rho^{**}(p,p)$. As before we start with the case $p \nmid s, p \nmid t$. Then our counting problem is equivalent to count $\rho^{*}(p^2)$. With Lemma~\ref{lemma:preparation} we know that not all $\tilde{a},\tilde{b},\tilde{e}$ and $\tilde{f}$ are invertible modulo p. But by definition we know that $p \nmid \tilde{f}$ and $\tilde{f}$. If $p \vert \tilde{a}$ this implies $p \vert t_{21}$. But then we do not get any solutions. Thus $p \vert \tilde{b}$ and therefore $p \vert t_{22}$. But this is a contradiction to the choice of $t_{12},t_{22}$ and $t_{32}$. Thus in total all $\tilde{a},\tilde{b},\tilde{e}$ and $\tilde{f}$ are invertible modulo $p$ and with Lemma~\ref{lemma:preparation} we do not obtain any solutions.
  For the remaining cases we also obtain no solutions.
  Combining everything we obtain the lemma.
  \end{proof}
Next we consider the cases if $p \vert \mathbf{p}^\mathbf{k}$ and $p\vert \mathbf{q}^\mathbf{l}$
\begin{lemma}
    Assume $p\vert \mathbf{p}^\mathbf{k}$. Then we obtain if $\left( \frac{ \frac{a\mathbf{p}^{\mathbf{k}}\mathbf{q}^\mathbf{l}v^3}{p^2}}{p} \right) = 1$  
    \begin{align*}
        \rho^{**}(p^j,p^k)= \begin{cases}
            1 & \text{ if } (j,k)=(0,0)\\
            p-1 & \text{ if } (j,k) \in \{(1,0),(0,1)\}\\0 & \text{ if } (j,k)=(1,1).
        \end{cases}
    \end{align*}
    The same holds if $p \vert \mathbf{q}^\mathbf{l}$ and $\left( \frac{ \frac{b\mathbf{p}^\mathbf{k}\mathbf{q}^{\mathbf{l}}v^3}{p^2}}{p} \right) = 1$
\end{lemma}
\begin{proof}
    Assume $p\vert \mathbf{p}^\mathbf{k}$.\\
    We write $a$ as $a = p \hat{a}$ with $\hat{a} \in \mathbb{Z},p \nmid \hat{a}, \mathbf{p}^\mathbf{k}$ as  $\mathbf{p}^\mathbf{k} = \mathbf{p}^{\hat{\mathbf{k}}} p $ with $\mathbf{p}^{\hat{\mathbf{k}}} \in \mathbb{Z}, p \nmid \mathbf{p}^{\hat{\mathbf{k}}}$ and $t_{22}$ as $t_{22} = p \hat{t}_{22}$ with $\hat{t}_{22} \in \mathbb{Z},p \nmid \hat{t}_{22}$.
Then we get 
$ p \hat{a}(t_{12})^2+b(p \hat{t}_{22})^2 = v^3\mathbf{p}^{\hat{\mathbf{k}}} p\mathbf{q}^\mathbf{l}(t_{32})^2$
This is equivalent to
$\hat{a}(t_{12})^2+bp( \hat{t}_{22})^2 = v^3\mathbf{p}^{\hat{\mathbf{k}}} \mathbf{q}^\mathbf{l}(t_{32})^2$.
Since $\hat{a},bp $ and $ v^3\mathbf{p}^{\hat{\mathbf{k}}} \mathbf{q}^\mathbf{l}$ are pairwise coprime and squarefree we are now in the same situation as in Lemma~\ref{lemma:p_mid_a}. This means, if $(t_{12},t_{22},t_{32})$ exists then we need $$\left( \frac{v^3 \mathbf{p}^{\hat{\mathbf{k}}}\mathbf{q}^\mathbf{l}\hat{a}}{p} \right) = \left( \frac{ \frac{a\mathbf{p}^{\mathbf{k}}\mathbf{q}^\mathbf{l}v^3}{p^2}}{p} \right) = 1 $$
and we conclude in the same way as in the proof of Lemma~\ref{lemma:p_mid_a}. 
\end{proof}
With the same ideas we get the next two lemmas.
\begin{lemma}\label{lemma:p_mid_v}
    Assume $p \vert v$ and $p\neq 2$. Then we get if $\left( \frac{-ab}{p} \right) = 1$: 
    \begin{align*}
        \rho^{**}(p^j,p^k)= \begin{cases}
            1 & \text{ if } (j,k)=(0,0)\\
            p-1 & \text{ if } (j,k) \in \{(1,0),(0,1)\}\\p(p-1)& \text{ if } (j,k) =(0,2)\\
            p^2(p-1) & \text{ if } (j,k)\in \{(1,1), (0,3)\}\\
            p^3(p-1) & \text{ if }(j,k)=(1,2)\\
            0 & \text{if } (j,k)=(1,3)
        \end{cases}
    \end{align*}
\end{lemma}
The last possible value we have to consider is $p=2$.
\begin{lemma}
    If $v^3\mathbf{p}^\mathbf{k}\mathbf{q}^\mathbf{l}\equiv1 \mod 4$ we obtain
    \begin{align*}
        \rho^{**}(2^j,2^k)= \begin{cases}
        1 & \text{ if } (j,k) \in \{(0,0),(1,0),(0,1)\}\\0 & \text{ if } (j,k)=(1,1).
        \end{cases}
    \end{align*}
\end{lemma}
\subsubsection{Computation of the volume of $V$}
In this subsection we compute the volume of $V$, where\\ $V= \{(s,t)\in \mathbb{R}^2: t>0, \Vert \underline{q}(s,t)\Vert_{\text{new}}^\prime \leq 1\}$.
Since $\Vert \underline{q}(s,t) \Vert_{\text{new}}^\prime = \Vert M \underline{q}(s,t) \Vert_{\text{new}} $ consider the following definition:
\begin{align*}
    M\underline{q}(s,t) = \begin{pmatrix}
        t_{11}q_1(s,t)+t_{12}q_{2}(s,t)\\
        t_{21}q_1(s,t)+t_{22}q_2(s,t)\\
        t_{32}q_2(s,t)+q_3(s,t)
    \end{pmatrix} =: \begin{pmatrix}
        \theta_1(s,t)\\
        \theta_2(s,t)\\
        \theta_3(s,t)
    \end{pmatrix}.
\end{align*}
We know that $Q(\underline{q}(s,t))=0 $ for all $(s,t)\in \mathbb{R}^2$.  
Since $\Vert \underline{q}(s,t) \Vert_{\text{new}}^\prime = \Vert M \underline{q}(s,t) \Vert_{\text{new}} $  and $Q(\mathbf{x}) = Q_v(M\mathbf{x})$ we get $
    0=Q(\underline{q}(s,t))=Q_v(M \underline{q}(s,t)).$
Thus $M \underline{q}(s,t)$ is a point on the conic $Q_v$ and therefore $
    \Vert \underline{q}(s,t)\Vert_{\text{new}}^\prime = \Vert M\underline{q}(s,t)\Vert_{\text{new}} = \max(|\theta_1(s,t)|,|\theta_2(s,t)|)$. 
 With the definition of $\theta_1(s,t)$ and $\theta_2(s,t)$ we obtain $$V= \{(s,t)\in \mathbb{R}^2: t>0, \vert \theta_1(s,t)\vert, \vert \theta_2(s,t)\vert \leq 1\}.$$
To compute the volume of $V$ we observe $\theta_1(s,t)
  =t_{11}(\tilde{b}s^2+\tilde{e}st)+t_{12}(-\tilde{a}s^2-\tilde{f}t^2)$ and $\theta_2(s,t) 
 =t_{21}(\tilde{b}s^2+\tilde{e}st)+t_{22}(-\tilde{a}s^2-\tilde{f}t^2).$
 Furthermore we define the map 
 \begin{align*}
     F: \mathbb{R}^2 \rightarrow \mathbb{R}^2, (s,t)\mapsto (\theta_1(s,t),\theta_2(s,t)).
 \end{align*}
This map is smooth. Thus we obtain 
\begin{align*}
  \frac{\partial(\theta_1(s,t),\theta_2(s,t))}{\partial(s,t)} 
  &=\begin{pmatrix}
t_{11} & t_{12} \\ t_{21} & t_{22}
  \end{pmatrix}
  \begin{pmatrix}
      2s\tilde{b}+\tilde{e}t & \tilde{e}s\\
     -2\tilde{a}s &  -2\tilde{f}t.\\
  \end{pmatrix}
\end{align*}
Furthermore we obtain for the determinant: 
\begin{align}\label{eq:det}
     \left\vert \operatorname{det}\left( \frac{\partial(\theta_1(s,t),\theta_2(s,t))}{\partial(s,t)} \right) \right\vert  =4 (v^3\mathbf{p}^\mathbf{k}\mathbf{q}^\mathbf{l})^\frac{1}{2}\sqrt{a\theta_1^2+b\theta_2^2}.
\end{align}
Therefore 
\begin{align*}
    \operatorname{vol}(V)  = \frac{1}{4(v^3\mathbf{p}^\mathbf{k}\mathbf{q}^\mathbf{l})^\frac{1}{2}}\int_{\vert x\vert, \vert y \vert \leq 1} \frac{G(x,y)}{\sqrt{ax^2+by^2}}dxdy,
\end{align*}
with $G(x,y)\coloneqq  \#\{(s,t)\in V : F(s,t)=(x,y)\}$.
For $G(x,y)$ we obtain that $G(x,y)=1$ for almost all $(x,y)\in [-1,1]^2.$
Thus in total we obtain 
\begin{align*}
    \operatorname{vol}(V)  = \frac{1}{4(v^3\mathbf{p}^\mathbf{k}\mathbf{q}^\mathbf{l})^\frac{1}{2}}\int_{\vert x\vert, \vert y \vert \leq 1} \frac{1}{\sqrt{ax^2+by^2}}dxdy.
\end{align*}
This is the same integral as in 
\cite[Section~6.2]{shute}. Therefore, we get 
\begin{align}
    \operatorname{vol}(V) 
&= \frac{\overline{c}_{a,b}}{(v^3   \mathbf{p}^{\mathbf{k}}  \mathbf{q}^{\mathbf{l}})^{\frac{1}{2}}}
\end{align}
with 
$$ \overline{c}_{a,b} \coloneqq  \frac{\sinh^{-1} \left( \sqrt{\frac{a}{b}}\right) } {\sqrt{a}}+\frac{\sinh^{-1} \left( \sqrt{\frac{b}{a}}\right) } {\sqrt{b}}. $$

\subsubsection{Error terms}
Considering the different values of $\rho^{**}$ together with (\ref{eq:N_hat_3}) we get
\begin{align*}
    \hat{N}(Q,B) = \operatorname{vol}(V)\sum_{r\lambda | \frac{\Delta_Q}{2\mathbf{p}^\mathbf{k}\mathbf{q}^\mathbf{l}} } \frac{\mu(r) \rho^{**}(r,\lambda)}{r^2\lambda}  B\sum_{\substack{1 \leq m \\\gcd(m,r\lambda )=1}}  \frac{\mu(m)}{m^2} + E_1+E_2-E_3
\end{align*}
with
\begin{align*}
    E_1\coloneqq  \sum_{r\lambda |  \frac{\Delta_Q}{2\mathbf{p}^\mathbf{k}\mathbf{q}^\mathbf{l}} } \mu(r) \rho^{**}(r,\lambda) \sum_{\substack{1 \leq m \leq  \mathscr{L}\\\gcd(m,r\lambda )=1}}O\left( \frac{\mu(m)\mathscr{L}}{m}\right),
\end{align*}
\begin{align*}
    E_2\coloneqq  \sum_{r\lambda |  \frac{\Delta_Q}{2\mathbf{p}^\mathbf{k}\mathbf{q}^\mathbf{l}} } \mu(r) \rho^{**}(r,\lambda) \sum_{\substack{\mathscr{L} < m \leq (2BK_0/r)^\frac{1}{2}\\\gcd(m,r\lambda)=1}} O(1)
\end{align*}
and 
\begin{align*}
    E_3\coloneqq  \sum_{r\lambda |  \frac{\Delta_Q}{2\mathbf{p}^\mathbf{k}\mathbf{q}^\mathbf{l}} }  \mu(r) \rho^{**}(r,\lambda) \sum_{\substack{ m >  \mathscr{L}\\\gcd(m,r\lambda )=1}} \left( \frac{\mu(m)\text{vol}(V)B}{m^2r^2\lambda}\right).
\end{align*}
For $E_1$ we obtain
\begin{align*}
    E_1 
&\ll \left\vert \frac{(K_0B)^\frac{1}{2}}{1} \frac{\langle Q \rangle }{\vert \Delta_Q\vert^\frac{1}{2} }\right\vert\left|
\log\left(K_0B\right)\right|\sum_{r\lambda |  \frac{\Delta_Q}{2\mathbf{p}^\mathbf{k}\mathbf{q}^\mathbf{l}} } \left\vert \frac{\mu(r) \rho^{**}(r,\lambda) }{r\lambda^\frac{1}{2}} \right\vert,
\end{align*}
where 
\begin{align*}
    &\sum_{r\lambda |  \frac{\Delta_Q}{2\mathbf{p}^\mathbf{k}\mathbf{q}^\mathbf{l}} } \left\vert \frac{\mu(r) \rho^{**}(r,\lambda)}{r\lambda^\frac{1}{2}}  \right\vert\\
    &= \sum_{\substack{(r_p,\lambda_p)\\ \prod_p p^{r_p+\lambda_p}\vert 2abv^3}} \left( \prod_{p \vert \prod_p p^{r_p+\lambda_p}} \left\vert  \frac{\mu(p^{r_p})\rho^{**}(p^{r_p},p^{\lambda_p})}{p^{r_p+\frac{\lambda_p}{2}}}  \right\vert \right)\\
    &\leq  \prod_{p \vert 2abv^3} \left( 1+ \sum_{\substack{i,j\geq 0 \\ i +j >0 \\ p^{i+j}\vert 2abv^3\mathbf{p}^{\mathbf{k}}\mathbf{q}^{\mathbf{l}}}} \left\vert \frac{\mu(p^{i})\rho^{**}(p^{i},p^j)}{p^{i+\frac{j}{2}}} \right\vert \right)\\
    &\leq  \prod_{p \vert 2abv^3} \tau_p
\end{align*}
with $\tau_p =  \left( 1+ \sum_{\substack{i,j\geq 0 \\ i +j >0 \\ p^{i+j}\vert 2abv^3\mathbf{p}^{\mathbf{k}}\mathbf{q}^{\mathbf{l}}}} \left\vert \frac{\mu(p^{i})\rho^{**}(p^{i},p^j)}{p^{i+{\frac{j}{2}}}} \right\vert  \right)$. \\
Computing $\tau_p$ for the different primes we get 
\begin{align*}
    \tau_p = \begin{cases}
        1+ \frac{1}{2^{\frac{1}{2}}}+ \frac{1}{2} & \text{ if } p=2\\
        \frac{p^\frac{3}{2}+2p-p^\frac{1}{2}-1}{p} & \text{ if } p \vert ab\\
        \frac{p^3 + 2p^\frac{5}{2}-p^\frac{3}{2}+p-p^\frac{1}{2}-1}{p} & \text{ if } p \vert v.
    \end{cases}
\end{align*}
Therefore,
\begin{align*}
    \prod_{p \vert 2abv^3} \tau_p 
    &\ll \prod_{p \vert v} \frac{p^3 + 2p^\frac{5}{2}-p^\frac{3}{2}+p-p^\frac{1}{2}-1}{p}
\end{align*} and we get for the error term
\begin{align*}
    E_1 
&\ll \left\vert \frac{(K_0B)^\frac{1}{2}}{1} \frac{\langle Q \rangle }{\vert \Delta_Q\vert^\frac{1}{2} }\right\vert\left|
\log\left(K_0B\right)\right|\operatorname{rad}(v)^{2+\epsilon}.
\end{align*}
We use the same idea to obtain for the error term $E_2$ that
\begin{align*}
    E_2 \ll (BK_0)^\frac{1}{2}\operatorname{rad}(v)^{\frac{7}{2}+\epsilon}.
\end{align*}
For $E_3$ we obtain 
\begin{align*}
    E_3 &\ll \frac{\vert \Delta_Q\vert^\frac{1}{2} }{(K_0)^\frac{1}{2} \langle Q \rangle}\frac{B^\frac{1}{2}}{v^\frac{3}{2}} \operatorname{rad}(v)^{\frac{5}{2}+\epsilon}.
\end{align*}
Since $E_2\ll E_1$ and $E_3\ll E_1$ we conclude 
\begin{align}\label{eq:N_norm}
     N_{\Vert \cdot \Vert_{\text{new}}^{\prime}}(Q,B) &=\hat{N}(Q,B) + O(1)\\
&=c_{(\mathbf{k},\mathbf{l}),v}B + O\left(  \frac{(K_0B)^\frac{1}{2}\langle Q \rangle }{\vert \Delta_Q\vert^\frac{1}{2} }\left| \log\left(K_0B\right)\right|\operatorname{rad}(v)^{2+\epsilon}\right) \notag
 \end{align}
 for the constant \begin{align}\label{eq:constant_c}
     c_{(\mathbf{k},\mathbf{l}),v} \coloneqq   \text{vol}(V) \sum_{r\lambda |  \frac{\Delta_Q}{2\mathbf{p}^\mathbf{k}\mathbf{q}^\mathbf{l}} }  \frac{\mu(r)\rho^{**}(r,\lambda)}{r^2\lambda}  \sum_{\substack{m \in \mathbb{N}\\\gcd(m,r\lambda )=1}} \frac{\mu(m)}{m^2}.
 \end{align}
 The constant will be explained in Subsection~\ref{sec:computing_constant}.\\ 
We observe   
$\Delta_Q = 4 a  b v^3 \mathbf{p}^\mathbf{k}\mathbf{q}^\mathbf{l}  \ll \mathbf{p}^\mathbf{k}\mathbf{q}^\mathbf{l}v^3  $
and
\begin{align*}
    \langle Q \rangle &= \max(|at_{11}^2+bt_{21}^2|,|2at_{11}t_{12}+2bt_{21}t_{22}|,|- 2  \mathbf{p}^\mathbf{k}\mathbf{q}^\mathbf{l}v^3t_{32}|, |-\mathbf{p}^\mathbf{k}\mathbf{q}^\mathbf{l}v^3|)
\end{align*}

Thus $ \langle Q \rangle \ll \Vert M \Vert_\infty\left( \Vert M \Vert_\infty + v^3\mathbf{p}^\mathbf{k}\mathbf{q}^\mathbf{l} \right) $ where $\Vert M \Vert_\infty = \max_{i,j} \vert m_{i,j}\vert $ .\\
Furthermore let $ \Vert M \Vert_{\text{new}, \infty}$ be the row sum norm. Furthermore we observe
\begin{align*}
    K_0= \sup_{\mathbf{x} \neq 0} \left( 1 + \frac{\Vert \mathbf{x} \Vert_{\infty}}{\Vert \mathbf{x} \Vert_{\text{new}}^{\prime}}\right)=\sup_{\mathbf{x} \neq 0} \left( 1 + \frac{\Vert (gM)^{-1}\mathbf{x} \Vert_{\infty}}{\Vert \mathbf{x} \Vert_{\infty}}\right) \ll \Vert (gM)^{-1} \Vert_{\text{new},\infty}.
\end{align*}
And we observe 
\begin{align*}
   &\Vert (gM)^{-1} \Vert_{\text{new},\infty}\\
   &= \max\left(\vert t_{22} | + |t_{12}|, |t_{21}| + |t_{11}|, |t_{32}t_{21}|+|t_{32}t_{11}|+ \sqrt{\frac{v^3 \mathbf{p}^\mathbf{k}\mathbf{q}^\mathbf{l}}{a+b}}\right)\\
   &\ll  \Vert M \Vert_{\infty}^2 + (v^3 \mathbf{p}^\mathbf{k}\mathbf{q}^\mathbf{l})^{\frac{1}{2}}.
\end{align*}
Together with (\ref{eq:N_norm}) we obtain 
\begin{align}\label{eq:N_norm_2}
     N_{\Vert \cdot \Vert_{\text{new}}^{\prime}}(Q,B) 
=c_{(\mathbf{k},\mathbf{l}),v}B
+ O\left(   E_4\right)
 \end{align}
where
 \begin{align*}
     E_4 = \frac{\left(\left(  \Vert M \Vert_{\infty}^2 + (v^3 \mathbf{p}^\mathbf{k}\mathbf{q}^\mathbf{l})^{\frac{1}{2}}\right)B)^\frac{1}{2}(\Vert M \Vert_\infty\left( \Vert M \Vert_\infty + v^3 \mathbf{p}^\mathbf{k}\mathbf{q}^\mathbf{l} \right)\right) }{\vert v^3 \mathbf{p}^\mathbf{k}\mathbf{q}^\mathbf{l}\vert^\frac{1}{2} }\\ \times \left|
\log\left(\left(  \Vert M \Vert_{\infty}^2 + (v^3 \mathbf{p}^\mathbf{k}\mathbf{q}^\mathbf{l})^{\frac{1}{2}}\right)B\right)\right|\operatorname{rad}(v)^{2+\epsilon}. \end{align*}
In the following, we estimate $\Vert M \Vert_{\infty}$. For that we need a bound for the smallest integral zero of the conic $Q_{(\mathbf{k},\mathbf{l}),v}$.
\begin{lemma}\label{lemma:bound_for_small_solution}
We use the notation above. If the equation $Q_{(\mathbf{k},\mathbf{l}),v}(\mathbf{x}) = 0$
has a nontrivial solution  with $\mathbf{x} \in \mathbb{Z}^3, \gcd(x_1,x_2)=1$, it has a solution satisfying  $\vert \vert \mathbf{x} \vert \vert_\infty \ll v^3  \mathbf{p}^{\mathbf{k}} \mathbf{q}^{\mathbf{l}}  $ with $\mathbf{x} \in \mathbb{Z}^3, \gcd(x_1,x_2)=1$.              
\end{lemma}
\begin{proof}
 Assume the equation $Q_{(\mathbf{k},\mathbf{l}),v}(\mathbf{x}) = 0$
has a nontrivial solution  with $\mathbf{x} \in \mathbb{Z}^3, \gcd(x_1,x_2)=1$. We fix the form $Q_0(x_1,x_2,x_3)\coloneqq ax_1^2 + bx_2^2 - v\mathbf{p}^{\mathbf{k}}\mathbf{q}^\mathbf{l}x_3^2$. Then $Q_0$ has an integral zero. By 
\cite[Theorem]{Cassels_1955}
there exist $\mathbf{y} \in \mathbb{Z}_{\text{prim}}^3$ with $Q_0(\mathbf{y})=0$ and $\Vert \mathbf{y} \Vert_{\infty} \leq \max(|a|,|b|,|v\mathbf{p}^{\mathbf{k}}\mathbf{q}^\mathbf{l}|) \ll v\mathbf{p}^{\mathbf{k}}\mathbf{q}^\mathbf{l}$. Furthermore, we fix a prime $p$ such that $p\vert v$. Then $p\nmid y_1y_2$ and $y_2$ is invertible modulo $p$.
Modulo $p$ we obtain $a(y_1y_2^{-1})^2 +b \equiv 0 \mod p$.
By the Chinese Remainder Theorem there exist $r\in \mathbb{Z}$ with $0 \leq r < v$ and $r \equiv -y_1y_2^{-1} \mod p$ for all $p\vert v$. 
Therefore, $ar^2+b \equiv 0 \mod v.$  
Fix $\mathbf{P}=(r,1,0)$ and let $B_0$ be the bilinear form corresponding to $Q_0$. Then $B_0(\mathbf{y},\mathbf{P}) = ay_1r+by_2$.
We consider the following quadratic parametrization 
\begin{align*}
\mathbf{N}&\coloneqq Q_0(\mathbf{P})\mathbf{y}-2B_0(\mathbf{y},\mathbf{P})\mathbf{P}\\   &= \begin{pmatrix}
        (ar^2+b)y_1 \\ (ar^2+b)y_2 \\ (ar^2+b)y_3
    \end{pmatrix} - \begin{pmatrix}
        2(ay_1r +by_2)r\\
        2(ay_1r +by_2)\\ 0
    \end{pmatrix} =: \begin{pmatrix}
            N_1 \\ N_2 \\ N_3 
        \end{pmatrix}
\end{align*}
Since $Q_0(\mathbf{N})=Q_0(\mathbf{P})^2Q_0(\mathbf{y})=0$, we obtain that $\mathbf{N}$ is an integral zero of $Q_0$.
Furthermore since $v\vert ar^2+b$ we get $v \vert N_3$. Thus, $\left( N_1, N_2, \frac{N_3}{v} \right) $ is an integral zero of $Q_{(\mathbf{k},\mathbf{l}),v}$. 
Next, we observe that for every prime $p\vert v$ we have 
$ B_0(\mathbf{y},\mathbf{P}) \equiv  -ay_1^2y_2^{-1} +by_2 \equiv 2by_2  \mod p$. 
Since $p\nmid 2by_2$ we get $p\nmid B_0(\mathbf{y},\mathbf{P})$. 
Thus we obtain $N_2 \equiv -2B_0(\mathbf{y},\mathbf{P}) \equiv -4by_2 \mod p$ and $\gcd(N_2,v)=1$.
Now fix $\mathbf{x} = \left( \frac{N_1}{\lambda}, \frac{N_2}{\lambda}, \frac{\frac{N_3}{v} }{\lambda} \right)$
with $\lambda = \gcd(N_1,N_2,\frac{N_3}{v})$. Observe that $\gcd(x_1,x_2)=1$. Thus, $\mathbf{x}$ is an integral solution of $Q_{(\mathbf{k},\mathbf{l}),v}$ with $\gcd(x_1,x_2)=1$.
It remains to show the bound for $\Vert \mathbf{x} \Vert_{\infty}$. Observe that  $\Vert \mathbf{y} \Vert_{\infty} \ll \mathbf{p}^\mathbf{l}\mathbf{q}^\mathbf{l}$ and since $0 \leq r <v$ we have $\Vert P \Vert_{\infty}\ll v$.
Furthermore, $|Q_0(\mathbf{P})| = |ar^2+b|  \ll v^2$ and 
 $|B_0(\mathbf{y},P)|=  ay_1r +by_2 \ll v^2\mathbf{p}^\mathbf{k}\mathbf{q}^\mathbf{l}$. Therefore, we obtain in total \begin{align*}
     \Vert \mathbf{N}\Vert_{\infty} \leq  |Q_0(\mathbf{P})|\Vert\mathbf{y}\Vert_{\infty}+2\vert B_0(\mathbf{y},\mathbf{P})\vert \Vert \mathbf{P}\Vert_{\infty} \ll v^3 \mathbf{p}^\mathbf{k}\mathbf{q}^\mathbf{l}.
 \end{align*} 
\end{proof}
Now we apply Lemma~\ref{lemma:bound_for_small_solution} and obtain $\Vert M \Vert_{\infty} \ll  v^3 \mathbf{p}^{\mathbf{k}} \mathbf{q}^{\mathbf{l}}$.
Thus, we obtain together with (\ref{eq:N_k,l,v,N_norm}), Lemma \ref{lemma:N_norm} and (\ref{eq:N_norm_2}) that
\begin{align}\label{eq:N_k_l_v}
N_{(\mathbf{k},\mathbf{l}),v}(B)
    &= c_{(\mathbf{k},\mathbf{l}),v}B + O(B^{\frac{1}{2}}v^{\frac{15}{2}}(\mathbf{p}^\mathbf{k}\mathbf{q}^\mathbf{l})^{\frac{5}{2}}\log\left(Bv^{6}(\mathbf{p}^\mathbf{k}\mathbf{q}^\mathbf{l})^2 \right)\operatorname{rad}(v)^{2+\epsilon})\\
    &= c_{(\mathbf{k},\mathbf{l}),v}B + O(B^{\frac{1}{2}+\epsilon}v^{\frac{19}{2}+\epsilon}(\mathbf{p}^\mathbf{k}\mathbf{q}^\mathbf{l})^{\frac{5}{2}+\epsilon}).\notag
\end{align}

\section{Computation of  $c_{(\mathbf{k},\mathbf{l}),v}$}\label{sec:computing_constant}Combining (\ref{eq:constant_c}) and Lemma~\ref{lemma:gcd} we write the constant $ c_{(\mathbf{k},\mathbf{l}),v}$ as
\begin{align}
    c_{(\mathbf{k},\mathbf{l}),v} =  \text{vol}(V) \sum_{r\lambda | 2abv^3} \frac{\mu(r)\rho^{**}(r,\lambda)}{r^2\lambda}  \sum_{\substack{m \in \mathbb{N}\\\gcd(m,r\lambda )=1}} \frac{\mu(m)}{m^2}.
\end{align}
We observe that we can write $\sum_{n=1}^\infty \frac{\mu(n)}{n^s}$ as an Euler product. Explicitly we obtain $\sum_{n=1}^\infty \frac{\mu(n)}{n^s} = \prod_p (1-p^{-s})$. Therefore we get 
\begin{align}
     c_{(\mathbf{k},\mathbf{l}),v} =  \frac{6\text{vol}(V)}{\pi^2} \sum_{r\lambda |2abv^3 } \frac{\mu(r)\rho^{**}(r,\lambda)}{r^2 \lambda}  \prod_{\substack{p \mid r^2\lambda}} \left(1 - \frac{1}{p^2}\right)^{-1}.
\end{align}
In the following, we write 
$r = \prod_{p} p^{r_p},$
where $r_p \in \{0,1\}$ for every prime $p$, with only finitely many $r_p$ being nonzero.  
Similarly, we set
$\lambda = \prod_{p} p^{\lambda_p},$
where $\lambda_p \in \{0,1,2,3\}$ for every prime $p$, and only finitely many of the exponents $\lambda_p$ are nonzero.
Therefore we obtain together with Lemma~\ref{lemma:rho_star_star_multiplicative}:
\begin{align*}
         c_{(\mathbf{k},\mathbf{l}),v} = \frac{6\text{vol}(V)}{\pi^2}\sum_{\substack{(r_p,\lambda_p)_p\\\prod_p p^{r_p+\lambda_p} \vert 2abv^3 }} \left(\prod_{\substack{p \vert \prod_p p^{r_p+\lambda_p} }} \frac{\mu(p^{r_p})\rho^{**}(p^{r_p},p^{\lambda_p})}{p^{2r_p+\lambda_p}}\left( 1-\frac{1}{p^2}\right)^{-1}  \right).
\end{align*}
We can decompose the full sum as an Euler product. Therefore we obtain
\begin{align}\label{eq:c_k_l_v}
     c_{(\mathbf{k},\mathbf{l}),v} =  \frac{6\text{vol}(V)}{\pi^2}\prod_{\substack{p\vert 2abv^3}}\left(1+  \sum_{\substack{i,j \geq  0\\i+j>0 \\ p^{i+j}\vert 2abv^3}} \frac{\mu(p^{i})\rho^{**}(p^{i},p^{j})}{p^{2i+j}}\left(1-\frac{1}{p^2}\right)^{-1} \right).
\end{align}
Let $\lambda_p =  \left( 1+ \sum_{\substack{i,j\geq 0 \\ i +j >0 \\ p^{i+j}\vert 2abv^3}}  \frac{\mu(p^{i})\rho^{**}(p^{i},p^j)}{p^{2i+j}} \left(1-\frac{1}{p^2}\right)^{-1}  \right)$. \\
We compute $\lambda_p$ for the different primes.\\
For $p=2$ we compute $\lambda_2 = \frac{4}{3}$ and if $p\vert a$ or $p\vert b$ we obtain $\lambda_p = \frac{2p}{p+1}.$
For the last case $p\vert v$ we get $\lambda_p= \frac{2p}{p+1}.$
In total we get $\prod_{p \vert 2abv^3} \lambda_p = \frac{4}{3}\prod_{p\vert abv}\frac{2p}{p+1}.$
Define 
$\gamma_{(\mathbf{k},\mathbf{l}),v} \coloneqq   \prod_{p \vert abv} \frac{p}{p+1}$
and 
\begin{align*}
\varrho_{(\mathbf{k},\mathbf{l}),v}\coloneqq  \prod_{p \vert \frac{a}{\mathbf{p}^\mathbf{k}}}\left(1 + \left(\frac{bv\mathbf{p}^k\mathbf{q}^\mathbf{l}}{p} \right)\right)\prod_{p \vert \frac{b}{\mathbf{q}^\mathbf{l}}}\left(1 + \left(\frac{av\mathbf{p}^k\mathbf{q}^\mathbf{l}}{p} \right)\right)\prod_{p \vert v}\left(1 + \left(\frac{-ab}{p} \right)\right)\\ \times \prod_{p \vert \mathbf{q}^\mathbf{l}}\left(1 + \left( \frac{ \frac{bv\mathbf{p}^{\mathbf{k}}\mathbf{q}^\mathbf{l}}{p^2}}{p} \right)\right)\prod_{p \vert \mathbf{p}^\mathbf{k}}\left(1 + \left( \frac{ \frac{av\mathbf{p}^{\mathbf{k}}\mathbf{q}^\mathbf{l}}{p^2}}{p} \right)\right) .
\end{align*}
With this definitions we obtain
\begin{align*}
c_{(\mathbf{k},\mathbf{l}),v}
  &= \frac{8\overline{c}_{a,b}\gamma_{(\mathbf{k},\mathbf{l}),v} \varrho_{(\mathbf{k},\mathbf{l}),v}}{\pi^2(v^3   \mathbf{p}^{\mathbf{k}}  \mathbf{q}^{\mathbf{l}})^{\frac{1}{2}}}.
\end{align*}
With the asymptotic formula (\ref{eq:N_k_l_v}) and the bound for $Z_1$ in (\ref{bound:Z_1}) we obtain
\begin{align*}
    \hat{N}_{(\mathbf{k},\mathbf{l})}(B)= \frac{B}{2} \sum_{v<B^\delta}\mu^2(abv)c_{(\mathbf{k},\mathbf{l}),v}
    + O\left(B^{\frac{1}{2}+\frac{21\delta}{2}+\epsilon}(\mathbf{p}^\mathbf{k}\mathbf{q}^\mathbf{l})^{\frac{5}{2}+\epsilon}\right)\\
    + O\left(\frac{B^{1+\epsilon - \frac{1}{2}\delta }}{(\mathbf{p}^\mathbf{k}\mathbf{q}^\mathbf{l})^\frac{1}{2}}\right).
\end{align*}

We choose $\delta$ such that the two error terms are balanced, so we take $\delta = 1/22$ and get 
\begin{align*}
    \hat{N}_{(\mathbf{k},\mathbf{l})}(B)
    &= \frac{B}{2} \sum_{v<B^{\frac{1}{22}}}\mu^2(abv)c_{(\mathbf{k},\mathbf{l}),v}
    + O\left(B^{\frac{43}{44}+\epsilon}(\mathbf{p}^\mathbf{k}\mathbf{q}^\mathbf{l})^{\frac{5}{2}+\epsilon}\right).
\end{align*}
We observe that $c_{(\underline{k},\underline{l}),v} \ll (v^3\mathbf{p}^\mathbf{k}\mathbf{q}^\mathbf{l})^{-\frac{1}{2}+\epsilon}$ and therefore we can use an infinite sum over $v$. Thus we get  
\begin{align*}\hat{N}_{(\mathbf{k},\mathbf{l})}(B) &= c_{(\mathbf{p}^\mathbf{k},\mathbf{q}^\mathbf{l})}B + O(B^{\frac{43}{44}+\epsilon}(\mathbf{p}^\mathbf{k}\mathbf{q}^\mathbf{l})^{\frac{5}{2}+\epsilon})
\end{align*}
with 
\begin{align}\label{eq:constant_c_k_l_2}
    c_{(\mathbf{p}^\mathbf{k},\mathbf{q}^\mathbf{l})}&=\frac{4\overline{c}_{a,b}}{\pi^2}\sum_{\substack{v\\ v\mathbf{p}^\mathbf{k}\mathbf{q}^\mathbf{l} \equiv 1 \mod 4 }}  \frac{\mu^2(abv)\gamma_{(\mathbf{k},\mathbf{l}),v} \varrho_{(\mathbf{k},\mathbf{l}),v}}{(v^3   \mathbf{p}^{\mathbf{k}}  \mathbf{q}^{\mathbf{l}})^{\frac{1}{2}}}.
\end{align}
We conclude the proof of Theorem~\ref{thm:theorem_1} with obtaining
\begin{align*}
  N_S(B) &=  \sum_{\mathbf{k}\in \{0,1\}^n}\sum_{\mathbf{l}\in \{0,1\}^m} \hat{N}_{(\mathbf{k},\mathbf{l})}(B)\\
     &=CB + O\left( B^{\frac{43}{44}+\epsilon} \right),
\end{align*}
where
\begin{align}\label{eq:C_final}
C&=\sum_{\mathbf{k}\in \{0,1\}^n}\sum_{\mathbf{l}\in \{0,1\}^m} c_{(\mathbf{p}^\mathbf{k},\mathbf{q}^\mathbf{l})}\\
    &=\frac{4 \overline{c}_{a,b}}{\pi^2} \sum_{\mathbf{k}\in \{0,1\}^n}\sum_{\mathbf{l}\in \{0,1\}^m} \frac{1}{(\mathbf{p}^{\mathbf{k}}  \mathbf{q}^{\mathbf{l}})^{\frac{1}{2}}}\sum_{\substack{v\\ v\mathbf{p}^{\mathbf{k}}  \mathbf{q}^{\mathbf{l}} \equiv 1 \text{ mod }4}} \frac{\mu^2(abv)\gamma_{(\mathbf{k},\mathbf{l}),v}\varrho_{(\mathbf{k},\mathbf{l}),v}}{v^\frac{3}{2}}. \notag
\end{align}

\begin{remark}
The situation is the same in the special case of $(\mathbf{k},\mathbf{l}) = (0, 0)$, which was studied by Shute.
In this case, the constant $C$ is explicitly known from
\cite[equation~(6.19)]{shute}
and takes the form
\begin{align*}
    C= \frac{4 \overline{c}_{a,b}}{\pi^2 } \sum_{\substack{v\\ v\equiv 1 \text{ mod }4}} \left(  \prod_{\substack{p \vert abv\\ p>2}} \left( 1+ \frac{1}{p} \right)^{-1} \right)\frac{\mu^2(abv)\varrho_{(\mathbf{k},\mathbf{l}),v}}{v^\frac{3}{2}}.
\end{align*} 
\end{remark}

\section{The PSTV-A conjecture}\label{sec:PSTVA}
As explained at the end of Section~\ref{sec:campana}, we can apply the PSTV-A conjecture to our setting.
We follow the first part of \cite[Section~6]{shute}.
There Shute considers the Campana orbifold $(X,D)$ with $X=\mathbb{P}^1$ and $D=\frac{1}{2}V(ax^2+by^2)$ and counts Campana-$\mathbb{Z}$-points. Instead we count Campana-$\mathbb{Z}[\frac{1}{ab}]$-points. 
By Section~\ref{sec:campana}, the PSTV-A conjecture applies since the set of Campana-points is not itself thin.
We compute $c_{\text{PSTV-A}}$ for the height $H([x:y])=\max(|x|,|y|)$ and $(x,y)\in \mathbb{Z}_{\text{prim}}^2$. 
We recall Definition~\ref{def:def2.6_in_Shute}. To avoid confusion, we will denote the quantities $a$ and $b$, as defined in Definition~\ref{def:def2.6_in_Shute}, as ${\tilde{a}}$ and ${\tilde{b}}$, respectively.
With \cite[Section~3]{shute} we get 
$$c_{\text{PSTV-A}} = \frac{\alpha \beta \tau }{\tilde{a}(\tilde{b}-1)!}.$$
We will explain what $\alpha,\beta$ and $\tau$ are in our case, a general definition is given in \cite{16inshute}. 
We observe as in \cite[Section~6.1]{shute} that $\tilde{b}=\beta = 1$. 
For $a$ we obtain $\deg(L)=1$, $\deg[K_{\mathbb{P}^1}]=-2$ and $\deg(D)=\frac{1}{2}\deg(V(ax^2+by^2)) = \frac{1}{2}\cdot 2 = 1$. Therefore
$$\tilde{a} = \inf \{t\in \mathbb{R}: t\cdot 1 -2 >0\} = \inf\{t\in \mathbb{R}: t>1\}=1.$$
For $\alpha $ we obtain 
$$\alpha = \left(1-\frac{1}{2}\right)\int_0^\infty e^{-x}dx = \frac{1}{2}.$$
Therefore we obtain 
\begin{align*}
    c_{\text{PSTV-A}} &=  \frac{\tau}{2}. 
\end{align*}
We observe that the divisor has degree $2$ and that $D$ corresponds to the line bundle $\mathscr{O}_{\mathbb{P}^1}(2).$ The line bundle determines the height function $\max(|x^2|,|y^2|)$ for $(x,y)\in \mathbb{Z}_{\text{prim}}^2$.
Let $v$ range over all primes and $v=\infty$. Then we use the same height function as Shute
$$H_{D} = \prod_v H_{D_{},v},$$
with 
$$H_{D,v} = \frac{\max(|x|_v, |y|_v)}{\left\vert ax^2+by^2\right\vert_v^{\frac{1}{2}}}.$$
Since we count $S$-integers we have the local Campana condition that\\ $\nu_p(ax^2+by^2)-2\min(\nu_p(x),\nu_p(y)) \neq 1$ only if $p \notin S$. Let $y \neq 0$ and set $z = x/y$. Then the condition that $\nu_p(ax^2+by^2)-2\min(\nu_p(x),\nu_p(y)) \neq 1$ is equivalent to $\nu_p(az^2+b)$ is not $1$ or a negative odd integer.\\
Furthermore we define $\Omega_p$ to be the elements in $\mathbb{Q}_p$ that are satisfying the local Campana condition. We set $\Omega_\infty = \mathbb{R}.$ As in Shute we obtain 
\begin{align}
    c_{\text{PSTV-A}} = \frac{1}{2}\sigma_{\infty}\prod_p \left( 1- \frac{1}{p}\right)\sigma_{p}
\end{align}
with
\begin{align}\label{eq:sigma_v}
    \sigma_{v} = \int_{\Omega_v}\frac{dz}{\max(|z|_v,1) \left\vert az^2+b\right\vert_v^\frac{1}{2}}. 
\end{align}
To compute $\sigma_\infty$ we follow   
\cite[(6.3)]{shute} and get 
\begin{align}\label{eq:sigma_infinity}
    \sigma_{\infty} &= \int_{\vert z \vert > 1} \frac{dz}{\left( az^2+b\right)^\frac{1}{2}} + \int_{\vert z \vert \leq 1} \frac{dz}{|z|\left( az^2+b\right)^\frac{1}{2}}= 2 \overline{c}_{a,b}.
\end{align}
For the densities at the finite primes we obtain the following lemma. 
\begin{lemma}\label{eq:sigma_p}
We have 
\begin{align}
    \sigma_{p}= \begin{cases}
        1+p^{-1}+\left( 1+ \left( \frac{-ab}{p} \right)\right)p^{-\frac{3}{2}} & \text{ if } p \nmid 2ab \\
        1 & \text{ if }p\vert 2ab \text{ and } p \notin S\\
        1 + p^{-\frac{1}{2}} & \text{ if } p \in S.
    \end{cases}
\end{align}
\end{lemma}
\begin{proof}
Let $\mu_p$ be the $p$-adic measure. \\
For the primes $p \notin S$ we have the same results as in
\cite[Lemma~6.2]{shute}. Therefore, we consider the case $p \in S$. For these primes, we have no local Campana condition and, therefore $\Omega_p = \mathbb{Q}_p$.\\ Thus we get for (\ref{eq:sigma_v}) if $p\in S$ :
\begin{align*}
    \sigma_{p} = \ \int_{\mathbb{Q}_p}\frac{dz}{\max(|z|_v,1) \left\vert az^2+b \right\vert_v^\frac{1}{2}}. 
\end{align*}
As in the proof of
\cite[Lemma~6.2]{shute} we split $\Omega_{p}$ into the regions $R_1,R_2,R_3$ respectively by $v_p(z)\geq 1,$, $v_p(z)<0,$ and $v_p(z)=0.$
We divide in the two cases $p \mid a$ and the case $p \mid b$.
To compute the integrals we follow the proof of \cite[Lemma~6.2]{shute}. In every case we compute $\vert az^2+b \vert_p$ and $\max(\vert z \vert_p,1)$. \\
In total we obtain if $p\vert ab$ that 
$$\sigma_{p} =p^{-1} + p^{-\frac{1}{2}}+1-p^{-1} = 1+p^{-\frac{1}{2}}.$$
\end{proof}
Together with (\ref{eq:sigma_infinity}) and Lemma~\ref{eq:sigma_p} we obtain that $c_{\text{PSTV-A}}$ is equal to 
\begin{align}
    c_{\text{PSTV-A}} 
    &= \frac{2\sigma_{\infty}}{\pi^2} \prod_{p \nmid 2ab} \left( 1+ \frac{ \left( 1 + \left( \frac{-ab}{p}\right)\right)}{(1+ p^{-1})p^{\frac{3}{2}}} \right) \prod_{p \in S }\frac{1+p^{-\frac{1}{2}}}{1+p^{-1}}.
\end{align}
\begin{remark}
    Using the same ideas as in \cite[Section~6.3]{shute} we observe that if $(a,b)=(1,q)$ for $q$ a prime congruent to $1 \mod 4$ that $c_{\text{PSTV-A}}<C$. For the values $(a,b)=(2437,2269)$ we get that $c_{\text{PSTV-A}}>C.$
\end{remark}

\section{The CLT-BT conjecture }\label{sec:new-conjecture}
In this section we compute the leading constant in the asymptotic formula of $N_S(B)$ for $S$ the set of primes dividing $a$ and $b$ following the CLT-BT conjecture in \cite{newconjecture}. First, we determine 
$\operatorname{Br}_1((X, D),L)/\operatorname{Br}(\mathbb{Q})$, and then we compute the respective summands $\hat{\tau}_{U,D}(b)$ for $b\in \operatorname{Br}_1((X, D),L)/\operatorname{Br}(\mathbb{Q})$. From now on, we denote the constant $c(K,S,(\mathscr{X},\mathscr{D}),\mathscr{L})$ in Conjecture~\ref{con:new_conjecture} by $c_{\text{CLT-BT}}$.

\subsection{Determination of $\operatorname{Br}_1((X, D),L)/\operatorname{Br}(\mathbb{Q})$}

Since $a((X, D),L)$,\\ $b(F,(X,D),L)$ and $\alpha((X, D),L) $ are the same as in the PSTV-A conjecture, in our case, the constant has the following form 
\begin{align*}
c_{\text{CLT-BT}} = \frac{1}{2}\sum_{b \in \operatorname{Br}_1((X, D),L)/\operatorname{Br}(\mathbb{Q})} \hat{\tau}_{U,D}(b).
\end{align*}
In our setting we obtain $U=  \mathbb{P}^1$, $X^\circ = \mathbb{P}^1 \setminus V(ax^2+by^2)$ and thus $U \setminus X^\circ = D_1$. We deduce
\begin{align*}
\operatorname{Br}((X, D),L)&= \{\beta \in \operatorname{Br} (X^\circ): 2 \partial_{D_1}(\beta)=0 \}.
\end{align*}
Since $a,b\in \mathbb{N}$ we get that $-ab \notin \mathbb{Q}^{\times 2}$  and therefore $ax^2+by^2$ is irreducible over $\mathbb{Q}$. Thus $D_1 = V(ax^2+by^2) \subset \mathbb{P}^1$ is a closed point of degree $2$.
Since $\kappa(D_1)= \mathbb{Q}\left( \sqrt{\frac{-b}{a}}\right)= \mathbb{Q}(\sqrt{-ab}) $ we fix from now on $\theta = \sqrt{\frac{-b}{a}}$. 
We observe that $D_1$ splits into two points over $\overline{\mathbb{Q}}$. 
In homogeneous coordinates we can write them as $P_1=[i\sqrt{b}: -\sqrt{a}]$ and $P_2=[-i\sqrt{b}: -\sqrt{a}]$. 
Thus  
\begin{align*}
\overline{X^\circ} = \mathbb{P}_{\overline{\mathbb{Q}}}^1 \setminus \{P_1,P_2\} \cong \mathbb{P}_{\overline{\mathbb{Q}}}^1 \setminus \{0, \infty\} = \mathbb{G}_{m, \overline{\mathbb{Q}}}. 
\end{align*}
With $\overline{X^\circ} =\mathbb{G}_{m, \overline{\mathbb{Q}}}$ and $\operatorname{Br}(\mathbb{G}_{m, \overline{\mathbb{Q}}})=0$ we get 
\begin{align*}
\operatorname{Br}_1(X^\circ)=\ker(\operatorname{Br}(X^\circ) \to \operatorname{Br}(\overline{X^\circ}))=\operatorname{Br}(X^\circ).
\end{align*}
Since $\operatorname{Br}_1(X^\circ)=\operatorname{Br}(X^\circ)$
and $\operatorname{Br}((X,D),L)\subseteq \operatorname{Br}(X^\circ)$
we obtain that $\operatorname{Br}_1((X,D),L)
=\operatorname{Br}((X,D),L).$
And therefore $\operatorname{Br}_1((X,D),L)
/\operatorname{Br}(\mathbb{Q}) = \operatorname{Br}((X,D),L)
/\operatorname{Br}(\mathbb{Q})$. Thus we obtain 
\begin{align}\label{eq:Br_1} 
\operatorname{Br}_1((X,D),L)
/\operatorname{Br}(\mathbb{Q}) &= \operatorname{Br}((X, D),L) /\operatorname{Br}(\mathbb{Q})\\ &= \{\beta \in \operatorname{Br} (X^\circ): 2 \partial_{D_1}(\beta)=0 \}/\operatorname{Br}(\mathbb{Q}). \notag
\end{align}
\subsubsection{Sequences}
To further compute $\operatorname{Br}_1((X,D),L)
/\operatorname{Br}(\mathbb{Q})$, we look at some exact sequences.
By \cite[Chapter III, Example 2.22]{milne} we obtain the exact sequence
\begin{align}\label{seq:milne}
0\longrightarrow \operatorname{Br}(X^\circ) \longrightarrow \operatorname{Br}(\mathbb{Q}(t)) \stackrel{\oplus_{v\in X^{\circ(1)}}\partial_v}{\longrightarrow} \bigoplus_{v\in X^{\circ(1)}} H^1(\kappa(v),\mathbb{Q}/\mathbb{Z}),
\end{align}
where $v$ runs over all closed points of $X^\circ$. These are all closed points of $\mathbb{P}_{\mathbb{Q}}^1$ besides $D_1$.

\subsubsection{Remarks on the function $\partial_D$}
For every prime divisor $v \in (\mathbb{P}^1)^{(1)}$ exists a canonic map $\partial_v: \operatorname{Br}(\mathbb{Q}(t)) \to H^1(\kappa(v),\mathbb{Q}/\mathbb{Z}).$
This map measures the ramification. This means $\partial_v(\beta)=0$ if $\beta$ is unramified at $v$ and $\partial_v(\beta)\neq 0$ if $\beta$ is ramified at $v$.
Since $X^{\circ(1)}=(\mathbb{P}^1_{\mathbb{Q}})^{(1)}\setminus\{D_1\},$
it follows by (\ref{seq:milne}) that 
$\beta\in\operatorname{Br}(X^\circ)$ if and only if $\partial_v(\beta)=0$ for all $v\neq D_1.$
Note that the divisor $D_1$ does not belong to $X^\circ$, and therefore
$\partial_{D_1}$ does not occur in the above exact sequence.
Nevertheless, the residue map at $D_1$
\begin{align*}
    \partial_{D_1}:\operatorname{Br}(\mathbb{Q}(t))\longrightarrow H^1(\kappa(\theta),\mathbb{Q}/\mathbb{Z})
\end{align*}
is defined on the Brauer group of the function field.
Restricting this map along the natural inclusion $\operatorname{Br}(X^\circ)
\hookrightarrow
\operatorname{Br}(\mathbb{Q}(t)),$
we obtain a homomorphism 
\begin{align*}
    \partial_{D_1}:\operatorname{Br}(X^\circ)\longrightarrow H^1(\kappa(\theta),\mathbb{Q}/\mathbb{Z}),
\end{align*}
which is not forced to vanish by the sequence in (\ref{seq:milne}), since $D_1$
is not a point of $X^\circ$. Since $\mathbb{Q}$ is a perfect field we obtain by \cite[Theorem~1.5.2]{brauer_book}
    \begin{align*}
        0 \to \operatorname{Br}(\mathbb{Q}) \to \operatorname{Br}(\mathbb{Q}(t)) \to \bigoplus_{x\in(\mathbb{P}^1_{\mathbb{Q}})^{(1)}} H^1(\mathbb{Q}(x),\mathbb{Q}/\mathbb{Z}) \to H^1(\mathbb{Q}, \mathbb{Q} / \mathbb{Z}) \to 0.
    \end{align*}
Since only the residue at $D_1$ may be non-zero we get $\operatorname{cores}_{\kappa(\theta)/\mathbb{Q}}(\partial_{D_1}(\beta))=0$ and by the Campana condition we have $2\partial_{D_1}(\beta)=0$. This lies in the residue $H^1(\kappa(\theta),\mathbb{Z}/2\mathbb{Z})$ and we get 
\begin{align*}
\operatorname{Br}((X,D),L)/\operatorname{Br}\mathbb Q \cong \ker(\operatorname{cores}_{\kappa(\theta)/\mathbb{Q}}:H^1(\kappa(\theta),\mathbb Z /2) \to H^1(\mathbb Q, \mathbb Z/2)).
\end{align*}
Using Kummer theory we have 
$H^1(\kappa(\theta),\mathbb{Z}/2\mathbb{Z})\cong\kappa(\theta)^\times/\kappa(\theta)^{\times2},$
and the corestriction corresponds to the norm map.
Therefore we obtain together with (\ref{eq:Br_1})
\begin{align*}
    \operatorname{Br}_1((X,D),L)/\operatorname{Br}\mathbb{Q} \cong\ker\left( N_{\kappa(\theta)/\mathbb Q}: \kappa(\theta)^\times/\kappa(\theta)^{\times2} \to \mathbb{Q}^\times/\mathbb Q^{\times2}\right).
\end{align*}

\subsubsection{Quaternion algebras}
We give explicit representatives for the elements of
$\operatorname{Br}((X,D),L)/\operatorname{Br}(\mathbb Q)$.
On the affine chart $y\neq 0$, let $ t=\frac{x}{y}$ and $ K=\mathbb Q(\theta).$
Then $\mathbb Q(X^\circ)=\mathbb Q(t)$ and
$\kappa(\theta)=K$.
\begin{proposition}
Let $ m=r+s\theta\in K^\times,
 r,s\in\mathbb Q,$
and assume that $ N_{K/\mathbb Q}(m)=r^2+\frac ba s^2
    \in\mathbb Q^{\times2}.$ Then $ \beta_m \coloneqq  \left(t^2+\frac ba,r+st\right) $ defines an element of $\operatorname{Br}((X,D),L)$ and satisfies $\partial_{D_1}(\beta_m) = [m] \in K^\times/K^{\times2}.$ Moreover, every class in $\operatorname{Br}((X,D),L)/ \operatorname{Br}(\mathbb Q)$ is represented by a quaternion algebra of this form. 
\end{proposition}

\begin{proof}
For a quaternion algebra $(f,g)$, recall that
\begin{align*}
        \partial_v(f,g)  =(-1)^{v(f)v(g)} \frac{f^{v(g)}}{g^{v(f)}} \pmod{\kappa(v)^{\times2}}.
\end{align*}
At $D_1$ the function $t^2+\frac ba$ has a valuation one, while
$r+st$ is a unit. Since $t=\theta$ in
$\kappa(\theta)=K$, we obtain $\partial_{D_1}(\beta_m)= (r+s\theta)^{-1} = [m] \text{ in }K^\times/K^{\times2}.$
Suppose first that $s\neq0$. At the divisor $r+st=0$,
\begin{align*}
\partial_{r+st}(\beta_u)=\left. \left(t^2+\frac ba\right)\right|_{t=-r/s} = \frac{r^2+\frac ba s^2}{s^2},
\end{align*}
which is trivial in $\mathbb Q^\times/\mathbb Q^{\times2}$ by the
assumption on $N_{K/\mathbb Q}(m)$. If $s=0$, there is no such divisor.
At infinity we have $ v_\infty\left(t^2+\frac ba\right)=-2$ and $ v_\infty(r+st)=-1$,
when $s\neq0$, and the residue is represented by $s^2$, hence is
trivial. The case $s=0$ is immediate. All remaining residues vanish.
Thus $ \beta_m\in\operatorname{Br}(X^\circ)$
and, since $\partial_{D_1}(\beta_m)$ is $2$-torsion, $2\partial_{D_1}(\beta_m)=0.$
Hence $    \beta_m\in\operatorname{Br}((X,D),L).$
Finally, the residue map identifies
$\operatorname{Br}((X,D),L)/\operatorname{Br}(\mathbb Q)$
with the above norm kernel. Since $\partial_{D_1}(\beta_m)=[m]$,
every element of this kernel is represented by some $\beta_m$.
\end{proof}

Consequently, the conjectural leading constant becomes
\begin{align}\label{eq:new_constant_1}
c_{\text{CLT-BT}}=\frac{1}{2} \sum_{\substack{[m]\in K^\times/K^{\times2}\\
N_{K/\mathbb Q}(m)\in\mathbb Q^{\times2}}}\hat{\tau}_{U,D}(\beta_m),
\end{align}
where, writing $m=r+s\theta$,
\begin{align*}
\beta_m=\left(t^2+\frac ba,r+st\right).
\end{align*}
In particular, for $t_v\in X^\circ(\mathbb Q_v)_\epsilon$, the corresponding
local Brauer factor is
\begin{align*}
e^{2\pi i\operatorname{inv}_v(\beta_u(t_v))}=\left(t_v^2+\frac{b}{a},r+st_v\right)_v,
\end{align*}
where $(\cdot,\cdot)_v$ denotes the Hilbert symbol over $\mathbb Q_v$.

\subsubsection{New Parametrization}\label{sec:new_parametrization}
The elements $\theta^2$ and $-ab$ determine the same class in
$\mathbb{Q}^\times/\mathbb{Q}^{\times 2}$, since $ -\frac{b}{a}=(-ab)\left(\frac{1}{a}\right)^2$.
Define
\begin{align*}
M\coloneqq \left\{[m]\in K^\times/K^{\times 2}:N_{K/\mathbb{Q}}(m)\in\mathbb{Q}^{\times 2}\right\}. 
\end{align*}
$M$ is the index set of the sum in (\ref{eq:new_constant_1}). We obtain the following lemma. 
\begin{lemma}\label{lemma:M}
Every class $[m]\in M$ admits a representative in
$\mathbb{Q}^\times$.
\end{lemma}
\begin{proof}
Let $[m]\in M$ with $m \in K^{\times}$. 
Write $u$ as $m=m_1 + \theta m_2$ with $m_1,m_2 \in \mathbb{Q}$.
By definition, there exists
$q\in\mathbb{Q}^\times$ such that $N_{K/\mathbb{Q}}(m)=q^2.$
Set $v=\frac{m}{q}\in K^\times.$
Then $N_{K/\mathbb{Q}}(v) =1$. 
Furthermore, consider the function 
$\phi \in \operatorname{Gal}(K/\mathbb{Q}), \phi: K \to K$, $m_1+\theta m_2 \mapsto m_1 - \theta m_2$. With this function observe $v\phi(v) =1 $ and so $\phi(v)= \frac{1}{v}.$
For every $v$ there exists a $z\in K^\times$ such that $v = \frac{z}{\phi(z)}$.
Then since $N_{K/\mathbb{Q}}(z)= z\phi(z)$\
we get $ v = \frac{z^2}{N_{K/\mathbb{Q}}(z)}.$
Since $u =qv$ we conclude
\begin{align*}
    [m]= \left[ \frac{q}{N_{K/\mathbb{Q}}(z)}  \right].
\end{align*}
\end{proof}
Define the map $\psi: \mathbb{Q}^\times/\mathbb{Q}^{\times 2} \to M, [m] \mapsto [m] $. By Lemma~\ref{lemma:M} the map $\psi$ is surjective. 

\begin{lemma}\label{lemma:square}
Let $c\in\mathbb{Q}^\times$. Then $c$ is a square in $K^\times$ if
and only if $c\in\mathbb{Q}^{\times 2}$ or $c\in(-ab)\mathbb{Q}^{\times 2}.$
\end{lemma}

\begin{proof} 
If $c \in K^{\times 2}$, we can write $c$ as $\left(c_1^2 -\frac{b}{a}c_2^2 \right) +2c_1c_2\theta$ with $c_1,c_2 \in \mathbb{Q}$, not both zero. 
For the other direction assume $c\in (-ab)\mathbb{Q}^{\times 2}$. Then there exists $c_3 \in \mathbb{Q}$ such that $c=(ac_3\theta)^2 \in K^{\times 2}$ 
\end{proof}
By Lemma~\ref{lemma:square} we get that $M \cong \mathbb{Q}^{\times}/(\mathbb{Q}^{2\times} \cup (-ab)\mathbb{Q}^{\times 2}).$
For every rational number $c$ there exists a unique squarefree integer $\operatorname{sf}(c)$ such that $c = \operatorname{sf}(c) d^2 $ with $d \in \mathbb{Q}$. 
Then $\operatorname{sf}(c)$ is unique and has the same sign as $c$. Furthermore, we observe $[c]=[\operatorname{sf}(c)] \in \mathbb{Q}^{\times }/\mathbb{Q}^{\times2}$.
Define 
\begin{align*}
\mathbb{Z}^\times_{free}= \{z \in \mathbb{Z}^\times: z \text{ is squarefree}\}.
\end{align*}
We consider the map
$$I:\mathbb{Z}^\times_{\mathrm{free}}\longrightarrow \mathbb {Z}^\times_{\mathrm{free}}, \quad
z\longmapsto \operatorname{sf}(-abz).$$
This map is a fixed-point-free involution. Hence every orbit of the action generated by $I$ has cardinality two and is of the form
$\mathcal O(z)=\{z,I(z)\}.$
Moreover, for any $z,z'\in\mathbb{Z}^\times_{\mathrm{free}}$, the corresponding orbits are either equal or disjoint; that is, $\mathcal O(z)=\mathcal O(z')$ or 
$\mathcal O(z)\cap\mathcal O(z')=\varnothing.$
Consequently, the set $\mathbb{Z}^\times_{\mathrm{free}}$ admits a partition into two-element $I$-orbits.
We now choose a set of representatives $Z\subseteq \mathbb{Z}^\times_{\mathrm{free}}$
containing exactly one element from each $I$-orbit. Equivalently, $Z$ is chosen such that
$$\mathbb {Z}^\times_{\mathrm{free}}=\bigsqcup_{z\in Z}\{z,I(z)\}.$$
In particular $Z\cap I(Z)=\varnothing$ and $Z\cup I(Z)=\mathbb{Z}^\times_{\mathrm{free}}.$
Observe that $Z$ is not unique, there are infinity many possibilities to choose $Z$. The next results hold for every choice of $Z$ and later we choose a specific $Z$.
\begin{proposition}\label{prop:bijection}
The map $F:Z \to M, z \mapsto [z]$
is bijective.
\end{proposition}
\begin{proof}
We prove surjectivity and injectivity separately.
For surjectivity, let $[m]\in M$. By definition of $M$, there exists $\tilde{m}\in\mathbb{Q}^\times$ such that $[m]=[\tilde{m}].$
Writing $\tilde{m}=\operatorname{sf}(\tilde{m})\hat{m}^2$ with 
$\hat{m}\in\mathbb{Q}^\times$
we obtain $[m]=[\operatorname{sf}(\tilde{m})],$
where $\operatorname{sf}(\tilde{m})\in\mathbb{Z}^\times_{\mathrm{free}}$. By the choice of $Z$, either $\operatorname{sf}(\tilde{m})\in Z$ or $I(\operatorname{sf}(\tilde{m}))\in Z$. In the first case, $F(\operatorname{sf}(\tilde{m}))=[m]$. In the second case, using $I(\operatorname{sf}(\tilde{m}))=\operatorname{sf}(-ab\operatorname{sf}(\tilde{m}))$
and $-ab = (a\theta)^2$, we have in $K^\times/K^{\times 2}$
$$[I(\operatorname{sf}(\tilde{m}))]
=[-ab\operatorname{sf}(\tilde{m})]
=[(a\theta)^2\operatorname{sf}(\tilde{m})]
=[\operatorname{sf}(\tilde{m})]
=[m].$$
Thus $F$ is surjective.
For injectivity, suppose that $z_1,z_2\in Z$ satisfy $F(z_1)=F(z_2).$
Then $z_1/z_2\in K^{\times 2}$. By Lemma~\ref{lemma:square}, either
$\frac{z_1}{z_2}\in\mathbb{Q}^{\times 2}$ or $\frac{z_1}{z_2}\in(-ab)\mathbb{Q}^{\times 2}.$
In the first case, $z_1$ and $z_2$ represent the same class in
$\mathbb{Q}^\times/\mathbb{Q}^{\times 2}$. Since both are squarefree integer representatives, uniqueness gives $z_1=z_2$.
In the second case, $[z_1]=[-abz_2]$ in $\mathbb{Q}^\times/\mathbb{Q}^{\times 2}.$
Hence by uniqueness of squarefree representatives we get $z_1=\operatorname{sf}(-abz_2)=I(z_2).$
This is impossible, since $Z$ contains exactly one representative from each $I$-orbit. Therefore only the first case can occur, and $z_1=z_2$.
Hence $F$ is bijective.
\end{proof}
With Proposition~\ref{prop:bijection} we obtain 
\begin{align}\label{eq:new_conj_sum}
 c_{\text{CLT-BT}} = \frac{1}{2}\sum_{z\in Z}\hat{\tau}_{U,D}(\beta_z).
\end{align}

\subsection{Computation of $\hat{\tau}_{U,D}(\beta_u)$}
Recall the definition of $\hat{\tau}_{U,D}(\beta_u)$ in (\ref{def:tau}).
Since $U=\mathbb{P}^1$, we have $\operatorname{Pic}(\overline U)=\operatorname{Pic}(\mathbb{P}^1_{\overline{\mathbb{Q}}})\cong\mathbb{Z}.$
Hence $L(\operatorname{Pic}(\overline{U}),s)=\zeta_{\mathbb{Q}}(s)$ and $L^*(\operatorname{Pic}(\overline{U}),1)=1.$
Moreover, for every finite prime $p$ we have $\lambda_p=
L_p(\operatorname{Pic}(\overline{U}),1)=(1-p^{-1})^{-1}$ and thus $\lambda_p^{-1}=1-p^{-1}.$ At the infinite place we obtain $\lambda_\infty=1$.
It remains to simplify the height appearing in the denominator.
Let $H$ denote the class of a divisor of degree one on $\mathbb{P}^1$.
Then $L\sim H,K_X\sim-2H,D_1\sim2H$
and therefore $D=\frac12D_1\sim_{\mathbb{Q}}H.$
Since $a((X,D),L)=1$, it follows that
\begin{align*}
    a((X,D),L)L+K_X+D\sim_{\mathbb{Q}}H-2H+H\sim_{\mathbb{Q}}0.
\end{align*}
With the chosen compatible adelic metrics, the corresponding local
height is therefore trivial:
\begin{align*}
    H_v(a((X,D),L)L+K_X+D,t_v)=1.
\end{align*}

\subsubsection{The Hilbert Symbols}

In order to compute the local densities, we first simplify the evaluation of the Brauer classes at local points. More precisely, we show that the Hilbert symbol associated with a class $m\in K^\times$ depends only on a rational representative of its square class.

\begin{proposition}
Let $m=r+s\theta\in K^\times$, where $r,s\in\mathbb{Q}$, and suppose that $q\in\mathbb{Q}^\times$ satisfies
$[m]=[q]$ in $K^\times/K^{\times 2}$. Then 
$\beta_m=\left(x^2+\frac ba,q\right)$ is in $\operatorname{Br}(X^\circ).$
In particular, for every place $v\in\Omega_{\mathbb {Q}}$ and every
$x_v\in X^\circ(\mathbb{Q}_v)$ we have $\beta_m(x_v)=\left(x_v^2+\frac ba,q\right)_v.$
At every point for which $r+sx_v\neq0$, this agrees with the original presentation:
$$\left(x_v^2+\frac ba,r+sx_v\right)_v=\left(x_v^2+\frac ba,q\right)_v.$$
\end{proposition}

\begin{proof}
Since $[m]=[q]$ in $K^\times/K^{\times2}$, there exists
$k=k_1+k_2\theta\in K^\times$, with $k_1,k_2\in\mathbb {Q}$, such that
$m=qk^2.$ With $\theta^2=-b/a$ and if we compare the coefficients we get
$r=q\left(k_1^2-\frac ba k_2^2\right),
$ and $s=2qk_1k_2.$
Thus 
\begin{align*}
r+sx=q\left(k_1^2-\frac ba k_2^2+2k_1k_2x\right)=q\left((k_1+k_2x)^2-k_2^2\left(x^2+\frac ba\right)\right).
\end{align*}
Now let $A=x^2+\frac ba.$
The second factor is a norm from the quadratic algebra
$\mathbb{Q}(x)(\sqrt A)/\mathbb{Q}(x),$ since
$(k_1+k_2x)^2-k_2^2A=N_{\mathbb{Q}(x)(\sqrt A)/\mathbb{Q}(x)}\bigl((k_1+k_2x)+k_2\sqrt A\bigr)$.
Therefore, $\left(A,(k_1+k_2x)^2-k_2^2A\right)=0$ in the Brauer group, and therefore
$\left(A,r+sx\right) =\left(A,q\right)$
wherever the former presentation is defined. Thus the two quaternion algebras define the same class over the function field of $X^\circ$.
The presentation $(A,q)$ is defined at every point of $X^\circ$. In particular, if $s\neq0$, it provides an extension of the original presentation across the point $x=-r/s$, where $r+sx$ vanishes. Hence
$\beta_m=\left(x^2+\frac ba,q\right)
\text{ in }\operatorname{Br}(X^\circ).$
At $x_v\in X^\circ(\mathbb{Q}_v)$ we get 
$\beta_m(x_v)=\left(x_v^2+\frac ba,q\right)_v,$
and, if $r+sx_v\neq0$ we obtain 
$\left(x_v^2+\frac ba,r+sx_v\right)_v=
\left(x_v^2+\frac ba,q\right)_v.$
\end{proof}
Therefore, we can express all relevant Brauer classes in terms of the representatives introduced above.
\begin{corollary}
Let $Z\subseteq\mathbb{Z}^\times_{\mathrm{free}}$ be the set of representatives chosen above. For every class
$[m]\in M$, there exists a unique $z\in Z$ such that
$[m]=[z] $ in $K^\times/K^{\times2}.$
For this representative we have 
$\beta_m=\beta_z=\left(x^2+\frac ba,z\right)
$ in $\operatorname{Br}(X^\circ).$
For every place $v\in\Omega_{\mathbb{Q}}$ and every
$x_v\in X^\circ(\mathbb{Q}_v)$ we have 
$\beta_m(x_v)=
\left(x_v^2+\frac ba,z\right)_v.$
\end{corollary}

\begin{proof}
The existence and uniqueness of $z$ follow from the bijection
$Z\xrightarrow{\sim}M,z\longmapsto[z].$
The statement about the Brauer classes then follows from the previous proposition.
\end{proof}
It follows by the next corollary that the local factors appearing in the Tamagawa-type constant can be computed entirely in terms of the rational representatives $z\in Z$.
\begin{corollary}
Let $z\in Z$ be the unique representative corresponding to
$[m]\in M$. Then
$\hat{\tau}_{U,D}(\beta_m)=\hat{\tau}_{U,D}(\beta_z),$
where
$$\hat{\tau}_{U,D}(\beta_z)=\prod_{v\in\Omega_{\mathbb{Q}}}\lambda_v^{-1}\int_{X^\circ(\mathbb{Q}_v)_\epsilon}\left(x_v^2+\frac ba,z\right)_vH_v(D,x_v)d\tau_{U,v}.$$
\end{corollary}

\begin{proof}
By the preceding corollary, for every place $v$ and every
$x_v\in X^\circ(\mathbb{Q}_v)$,
$$\beta_m(x_v)=\beta_z(x_v)=
\left(x_v^2+\frac ba,z\right)_v.$$
Since the remaining factors in the local integrals are independent of the chosen representative of the Brauer class, the corresponding local integrals agree at every place. Taking the product over all places proves the assertion.
\end{proof}
Consequently, using the bijection $Z\simeq M$, the sum over the relevant Brauer classes may be rewritten as
\begin{align*}
    c_{\text{CLT-BT}}= \frac{1}{2} \sum_{z\in Z} \prod_{v\in\Omega_{\mathbb{Q}}} \lambda_v^{-1} \int_{X^\circ(\mathbb{Q}_v)_\epsilon}\left(x_v^2+\frac ba,z\right)_v H_v(D,x_v)d\tau_{U,v}.
\end{align*}

\subsubsection{Local height functions and Campana points}
The local height associated with $D$ is $$H_{D,p}([x:y])
=\frac{\max(|x|_p,|y|_p)}
{|ax^2+by^2|_p^{1/2}}$$ at every finite place.
On the affine chart $y \neq 0$
we write $t=\frac{x}{y}.$
Since $ax^2+by^2=y^2(at^2+b),$
the local height becomes $$H_{D,p}(t)=\frac{\max(1,|t|_p)}{|at^2+b|_p^{1/2}}.$$
We next describe the domain of integration.
Since $X^\circ=X\setminus\operatorname{Supp}(D)=\mathbb{P}^1\setminus V(ax^2+by^2),$
the local Tamagawa integrals are taken over the Campana points of
$X^\circ(\mathbb{Q}_v)_\epsilon$.
On the affine chart $y\neq0$, writing $t = \frac{x}{y}$
we identify $X^\circ(\mathbb{Q}_v)=\{t\in \mathbb{Q}_v:at^2+b\neq0\}.$
The Campana condition imposes the additional local restriction only at
finite places $v\notin S$. There we define on the finite places, which ar enot in $S: X^\circ(\mathbb{Q}_p)_{\text{Campana}} \coloneqq \{t\in \mathbb{Q}_p:at^2+b\neq0, \nu_p(at^2+b) \text{not one or negative odd integer}\}$. Therefore we get
\begin{align*}
X^\circ(\mathbb{Q}_v)_\epsilon= 
\begin{cases} X^\circ(\mathbb{Q}_p)_{\text{Campana}}  ,& v\notin S, $ v \text{ finite}$\\
X^\circ(\mathbb{Q}_v),& v\in S\\ 
\mathbb{R} & v=\infty \end{cases}
\end{align*}
Therefore we obtain 
\begin{align*}
\hat{\tau}_{U,D}(\beta_z)=I_\infty(z)\prod_{p\in S}I_p^{S}(z)\prod_{p\notin S}I_p(z).
\end{align*}
with 
\begin{align*}
I_p(z):=(1-p^{-1})\int_{X^\circ(\mathbb{Q}_p)_{\text{Campana}} }\left(t^2+\frac ba,z\right)_p\frac{\max(1,|t|_p)}{|at^2+b|_p^{1/2}}d\tau_{U,p},
\end{align*}
\begin{align*}
I_p^{S}(z):=(1-p^{-1})\int_{X^\circ(\mathbb{Q}_p)}\left(t^2+\frac ba,z\right)_p\frac{\max(1,|t|_p)}{|at^2+b|_p^{1/2}}d\tau_{U,p}
\end{align*}    
 and 
\begin{align*}
I_\infty(z):=\int_{\mathbb{R}}\left(t^2+\frac ba,z\right)_\infty\frac{\max(1,|t|)}{|at^2+b|^{1/2}}d\tau_{U,\infty},
      \end{align*}

\subsubsection{$I_\infty(z)$ - the infinite place }\label{sec:infinit}
We observe that $\left( t^2 +\frac{b}{a},z\right)_{\infty}=1$ because $t^2 + \frac{b}{a}$ is never zero. 
Therefore we get the same as in the PSTV-A conjecture
\begin{align*}
    I_\infty(u) &= \int_\mathbb{R} \frac{dt}{\max(1,|t|) \sqrt{at^2+b}}\\
    &= 2 \left( \frac{\sinh^{-1}(\sqrt{a/b})}{\sqrt{a}}+ \frac{\sinh^{-1}(\sqrt{b/a})}{\sqrt{b}}\right)\\
    &= 2\overline{c}_{a,b}
\end{align*}

\subsubsection{$I_p(z)$-  The good primes }\label{sec:good}
We compute the integrals in a similar way as in \cite[Section~6.1]{shute}. 
For $p \notin S, p \neq 2$,  we have: 
\begin{align*}
        \Omega_p & \coloneqq  X^\circ(\mathbb{Q}_p)_{\text{Campana}}\\
&=\left\{t\in\mathbb{Q}_p: \nu_p(at^2+b)\neq1,\;\nu_p(at^2+b)\notin-2\mathbb{N}-1, at^2+b \neq 0\right\}.
    \end{align*}
Because of the Campana condition and since $\nu_p\left(t^2 + \frac{b}{a}\right)= \nu_p(at^2+b)$ it is not allowed that $\nu_p(at^2+b)$ is equal to one or an odd negative integer.
We normalize the additive Haar measure $dt$ on $\mathbb{Q}_p$ by $\operatorname{vol}(\mathbb{Z}_p)dt=1$ and
for $z\in Z$ we  write $  \chi_{z,p}(t)=\left(t^2+\frac ba,z\right)_p.  $
The local Tamagawa measure on $\mathbb{P}_{\mathbb{Q}_p}^1$ is
\begin{align*}
d\tau_{U,p}(t)=\frac{dt}{\max(1,|t|_p)^2}.
\end{align*}
On the affine chart, with coordinate $\tilde{z}=\frac1t = \frac{y}{x},$
we have $d\tilde{z}=-\frac{dt}{t^2},$
and hence $|d\tilde{z}|_p=\frac{|dt|_p}{|t|_p^2}.$ Thus we obtain
\begin{align*}
 H_{D,p}(t)d\tau_{U,p}(t)=\frac{\max(1,|t|_p)}{|at^2+b|_p^{1/2}}\frac{dt}{\max(1,|t|_p)^2}=\frac{dt}{\max(1,|t|_p)|at^2+b|_p^{1/2}}.
\end{align*}
Therefore we get
\begin{align*}
I_p(z)=(1-p^{-1})\int_{\Omega_p}\chi_{z,p}(t)\frac{dt}{\max(1,|t|_p)|at^2+b|_p^{1/2}}.
\end{align*}
We decompose the integral domain according to the valuation of $t$ in $
\Omega_p=\Omega_p^{+}\sqcup\Omega_p^{0}\sqcup\Omega_p^{-}$
with $\Omega_p^{+}=\{t\in\Omega_p:\nu_p(t)\geq1\}, $ $\Omega_p^{0}=\{t\in\Omega_p:\nu_p(t)=0\}$ and $\Omega_p^{-}=
\{t\in\Omega_p:\nu_p(t)<0\}.$
Define $A(t)\coloneqq  t^2+\frac{b}{a}$ and let $A_{0,p}(t) $ such that $A(t) = p^{\nu_p(A(t))}A_{0,p}(t)$ and $B(t)\coloneqq  z$ and let $B_{0,p} $ such that $B(t) = p^{\nu_p(B(t))}B_{0,p}(t)$.
Then we have 
\begin{align*}
    \chi_{z,p}(t)=(-1)^{\nu_p(A(t))\nu_p(z)(\frac{p-1}{2})}\left( \frac{A_{0,p}(t)}{p} \right)^{\nu_p(z)}\left( \frac{B_{0,p}(t)}{p} \right)^{\nu_p(A(t))}.
\end{align*}
On $\Omega_p^+$ we have $\nu_p(t)\geq 1$, so $\nu_p(A(t))=0$. Thus the Campana condition is fulfilled.
We observe $\chi_{z,p}(t)
= \left( \frac{ab}{p} \right)^{\nu_p(z)}$
and obtain 
\begin{align*}
    \int_{\Omega_p^+} \frac{\chi_{z,p}(t)dt}{\max(1,|t|_p)|at^2+b|_p^\frac{1}{2}} &= \left( \frac{ab}{p} \right)^{\nu_p(z)}\int_{\substack{\mathbb{Q}_p \\ 
    \nu_p(t)\geq 1\\ t^2+\frac{b}{a}\neq 0}} \frac{dt}{\max(1,|t|_p)|at^2+b|_p^\frac{1}{2}}\\
    &= \left( \frac{ab}{p} \right)^{\nu_p(z)}p^{-1}. 
\end{align*}

On $\Omega_p^-$ we have $\nu_p(t)<0$. Therefore $\nu_p(A(t))$ is always an negative even integer and so the Campana condition is fulfilled. 
We also observe if $\nu_p(t)=-k<0$ and since $p$ is a odd prime that 
Therefore, we get 
\begin{align*}
    \chi_{z,p}(t) =\left( \frac{\frac{t^2 + \frac{b}{a}}{p^{-2k}}}{p} \right)^{\nu_p(z)}. 
\end{align*}
Together with $\frac{t^2 + \frac{b}{a}}{p^{-2k}} = \left( t_1^2 + \frac{b}{a}p^{2k}\right)$
and $t_1^2 + \frac{b}{a}p^{2k}\equiv t_1^2$ $\mod p$ we obtain  
\begin{align*}
    \chi_{z,p}(t) =\left( \frac{\frac{t^2 + \frac{b}{a}}{p^{-2k}}}{p} \right)^{\nu_p(z)}  = \left( \frac{ t_1^2 + \frac{b}{a}p^{2k}}{p} \right)^{\nu_p(z)}  = \left( \frac{ t_1^2 }{p} \right)^{\nu_p(z)}  =1.
\end{align*}
Thus we get 
\begin{align*}
    \int_{\Omega_p^-} \frac{\chi_{z,p}(t)dt}{\max(1,|t|_p)|at^2+b|_p^\frac{1}{2}} =  \int_{\substack{\mathbb{Q}_p\\\nu_p(t)<0 \\ t^2+ \frac{b}{a}\neq 0}} \frac{dt}{\max(1,|t|_p)|at^2+b|_p^\frac{1}{2}}   
    = p^{-1}. 
\end{align*}

We consider the last integration domain $\Omega_p^0$.
Here we have $\nu_p(t)=0 $ so $t \in \mathbb{Z}_p^\times$.
Then $\nu_p\left(t^2 + \frac{b}{a}\right) \geq 0$.
We obtain with the Campana condition that $\nu_p(at^2+b)=1$ is forbidden and decompose the integral by the value of $\nu_p(at^2+b)$. 
\begin{align*}
    \int_{\Omega^0}
\frac{\chi_{z,p}(t)dt}{\max(1, |t|_p)|at^2+b|_p^\frac{1}{2}}=\int_{\substack{\mathbb{Q}_p\\ \nu_p(at^2+b)=0\\ \nu_p(t)=0\\ t^2 + \frac{b}{a}\neq 0}}
\frac{\chi_{z,p}(t)dt}{\max(1, |t|_p)|at^2+b|_p^\frac{1}{2}}\\+\int_{\substack{\mathbb{Q}_p\\ \nu_p(at^2+b)\geq 2 \\ \nu_p(t)=0\\ t^2+\frac{b}{a}\neq 0}}
\frac{\chi_{z,p}(t)dt}{\max(1, |t|_p)|at^2+b|_p^\frac{1}{2}}.
\end{align*}
For the first integral we obtain 
\begin{align*}
    \int_{\substack{\mathbb{Q}_p\\ \nu_p(at^2+b)=0\\ \nu_p(t)=0\\ t^2+\frac{b}{a}\neq 0}}
\frac{\chi_{z,p}(t)dt}{\max(1, |t|_p)|at^2+b|_p^\frac{1}{2}}
&= \int_{\substack{\mathbb{Z}_p^{\times}\\ at^2 + b \in \mathbb{Z}_p^\times }}\left( \frac{t^2 + \frac{b}{a}}{p} \right)^{\nu_p(z)}
dt\\
&= \frac{1}{p}\sum_{\substack{t \in \mathbb{F}_p^\times\\ at^2 +b \neq 0}}\left( \frac{t^2 + \frac{b}{a}}{p} \right)^{\nu_p(z)}\\
&= \begin{cases}
    \frac{p-2-\left(\frac{-ab}{p} \right)}{p} & \text{ if } p \nmid z \\
    -\frac{1+\left(\frac{ab}{p} \right) }{p} & \text{ if } p|z.
\end{cases}
\end{align*}
For the second integral we have 
\begin{align*}
    \chi_{u,p}(t)
=(-1)^{\nu_p(at^2+b)\nu_p(z)(\frac{p-1}{2})}\left( \frac{\frac{t^2+\frac{b}{a}}{p^{\nu_p(at^2+b)}}}{p} \right)^{\nu_p(z)}\left( \frac{\frac{z}{p^{\nu_p(z)}}}{p} \right)^{\nu_p(at^2+b)}
\end{align*}
And thus 
\begin{align*}
    \int_{\substack{\mathbb{Q}_p\\ \nu_p(at^2+b)\geq 2 \\ \nu_p(t)=0\\ t^2+\frac{b}{a}\neq 0 }}
\frac{\chi_{z,p}(t)dt}{\max(1, |t|_p)|at^2+b|_p^\frac{1}{2}} 
&= \begin{cases}
    \frac{2(1-p^{-1})p^{-1}}{1-\left( \frac{z}{p}\right)p^{-\frac{1}{2}}} & \text{ if } \left( \frac{-ab}{p} \right)=1, p\nmid z\\
     0 & \text{ otherwise. }
\end{cases}
\end{align*}
So we conclude 
\begin{align*}
            \int_{\Omega^0}
\frac{\chi_{z,p}(t)dt}{\max(1, |t|_p)|at^2+b|_p^\frac{1}{2}} &= \begin{cases} \frac{p-3}{p}  + 
    \frac{2(1-p^{-1})p^{-1}}{1-\left( \frac{z}{p}\right)p^{-\frac{1}{2}}} & \text{ if } \left( \frac{-ab}{p} \right)=1, p\nmid z\\
    \frac{p-1}{p}  & \text{ if }\left( \frac{-ab}{p}\right)=-1, p \nmid z\\
     -\frac{1+\left( \frac{ab}{p} \right)}{p} & \text{ if } p \vert z.
\end{cases} 
\end{align*}
And in total total we get
\begin{align}\label{eq:I_p}
    &
I_p(z)= \begin{cases}
    (1-p^{-1})\left( 1+p^{-1}+\left( 1 + \left( \frac{-ab}{p}\right)\right)\left( \frac{z}{p}\right)p^{\frac{-3}{2}}\right) & \text{ if } p \nmid z \\
    0 & \text{ if }p\vert z. 
\end{cases}
\end{align}

\subsubsection{$I_p^{S}(u)$
- The bad primes   }\label{sec:bad}
For the primes in $S$ we get 
      \begin{align*}
     I_p^{S}(z) = (1-p^{-1})\int_{\substack{\mathbb{Q}_p\\ at^2+b \neq 0}}\left(t^2+\frac ba,z\right)_p \frac{dt}{\max(1,|t|_p)|at^2+b|_p^\frac{1}{2}},
      \end{align*}
where $X^\circ(\mathbb{Q}_v)=\{t\in \mathbb{Q}_v:at^2+b\neq0\}.$
In the same way as we computed the integrals for the good primes we do it for the bad primes. We obtain if $p\vert a$
\begin{align*}
    I_p^{S}(z) = \begin{cases}
        (1-p^{-1})\left(\left( \frac{\frac{z}{\nu_p(z)}}{p}\right)+p^\frac{-1}{2}\right) & \text{ if } p \nmid z\\
        (1-p^{-1})\left(\left( \frac{-b(a/p)}{p}\right)\left( \frac{\frac{z}{\nu_p(z)}}{p}\right)+p^\frac{-1}{2}\right) & \text{ if } p \vert z. 
    \end{cases}
\end{align*}
If $p\vert b $ we get 
\begin{align*}
     I_p^{S}(z) = \begin{cases}
        (1-p^{-1})\left(\left( \frac{\frac{z}{\nu_p(z)}}{p}\right)p^\frac{-1}{2} +1 \right) & \text{ if } p \nmid z\\
        (1-p^{-1})\left(\left( \frac{-a(b/p)}{p}\right)\left( \frac{\frac{z}{\nu_p(z)}}{p}\right)p^\frac{-1}{2} + 1\right) & \text{ if } p \vert z.
    \end{cases}
\end{align*}
It is left to consider $p=2.$ Since $2$ is not in $S$ we have 
\begin{align*}
    I_2(z)=\left(1-\frac{1}{2}\right)\int_{\Omega_2}\left( t^2 + \frac{b}{a},z\right)_2 \frac{dt}{\max(1,|t|_2)|at^2+b|_2^{\frac{1}{2}}}
\end{align*} with
\begin{align*}
        \Omega_2=\left\{t\in\mathbb{Q}_2:
at^2+b \neq 0, \nu_2(at^2+b)\neq 1, \nu_2(at^2+b)\notin -(2\mathbb{N}+1)
\right\}.
    \end{align*}
    As in the proof for the good primes we decompose $\Omega_2$ into $\Omega_2^-,\Omega_2^0 $ and $\Omega_2^+$. Because of the Campana condition $\Omega_2^0$ is forbidden. For $\Omega_2^-$ we write $\nu_2(at^2+b)$ as $-2m$ and $m\in \mathbb{N}$. Furthermore observe $\left( t^2 + \frac{b}{a},z\right)_2 = \left( t^2\left( 1+ \frac{b}{at^2} \right),z\right)_2 = \left(  1+ \frac{b}{at^2},z\right)_2$. If $m \geq 2$ we get $1+\frac{b}{at^2} \in \mathbb{Q}_2^{\times 2}$. If $m=1$ we have $\left( t^2 + \frac{b}{a},z\right)_2=(5,z)_2$. Thus 
    \begin{align*}
         \int_{\Omega_2^-} \frac{\chi_{z,2}(t)dt}{\max(1,|t|_2)|at^2+b|_2^\frac{1}{2}} = \frac{1}{4}((5,z)_2 + 1).
    \end{align*}
Next consider $\Omega_2^+.$ With a similar idea as above we get 
    \begin{align*} \int_{\Omega_2^+} \frac{\chi_{z,2}(t)dt}{\max(1,|t|_2)|at^2+b|_2^\frac{1}{2}}
    = \frac{1}{4}\left(\left(\frac{5b}{a},z\right)_2+\left(\frac{b}{a},z\right)_2\right). 
        \end{align*}
In total we obtain
\begin{align}\label{eq:I_2}
    I_2(z)&= \frac{1}{8} ((5,z)_2 + 1)\left( 1 +  \left(\frac{b}{a},z\right)_2 \right) = \begin{cases}
    \frac{1}{2} & \text{ if } 2\nmid z\\
    0 & \text{ if } 2\vert z.
    \end{cases}
\end{align}

\subsubsection{Further classification of the summation domain}
We fix a set $Z$. First we observe that for every orbit $\{z,I(z)\}$ one is negative and the other one is positive. Therefore we choose $Z$ such that it contains all positive elements.
Furthermore by (\ref{eq:I_2}) we get that $\hat{\tau}_{U,D}(\beta_z)=0$ if $2\vert z$. By (\ref{eq:I_p}) we get that $\hat{\tau}_{U,D}(\beta_z)$ is zero if there is a prime that divides $z$ which is not in $S$. Thus we can reduce our sum to the set 
\begin{align*}
    Z_{\text{fix}} = \{\mathbf{p}^{\mathbf{k}}\mathbf{q}^{\mathbf{l}} \in \mathbb{Z}_{\text{free}}^{\times }, \mathbf{p}^{\mathbf{k}} \vert a, \mathbf{q}^{\mathbf{l}} \vert b  \}.
\end{align*}

\subsubsection{Combining the local densities}
In this section we use similar ideas as in \cite[Section~6.3]{shute}. 
Now fix $z\in Z_{\text{fix}}$ and write $z$ as $\mathbf{p}^\mathbf{k}\mathbf{q}^\mathbf{l}$. Then we obtain with the results of Sections~\ref{sec:infinit}, \ref{sec:good} and \ref{sec:bad} and the fact $\prod_p (1-p^{-1})(1+p^{-1})= \frac{6}{\pi^2}$ that 
\begin{align*}
\hat{\tau}_{U,D}(\beta_z) = \frac{8\overline{c}_{a,b}\sigma_2}{\pi^2}\prod_{p\nmid 2ab}  \left( 1+\frac{\left( 1 + \left( \frac{-ab}{p}\right)\right)\left( \frac{\mathbf{p}^\mathbf{k}\mathbf{q}^\mathbf{l}}{p}\right)p^{\frac{-3}{2}}}{1+p^{-1}}\right)\\
    \times \prod_{p\vert \mathbf{p}^{\mathbf{k}}}  \frac{\left( \frac{\frac{-ba}{p}}{p}\right)\left( \frac{\frac{\mathbf{p}^\mathbf{k}\mathbf{q}^\mathbf{l}}{p}}{p}\right)+p^\frac{-1}{2}}{1+p^{-1}}  
    \prod_{p\vert \frac{a}{\mathbf{p}^{\mathbf{k}}}}\frac{\left( \frac{\mathbf{p}^\mathbf{k}\mathbf{q}^\mathbf{l}}{p}\right)+p^\frac{-1}{2}}{1+p^{-1}}\\
    \times \prod_{p\vert \mathbf{q}^{\mathbf{l}}} \frac{\left( \frac{\frac{-ab}{p}}{p}\right)\left( \frac{\frac{\mathbf{p}^\mathbf{k}\mathbf{q}^\mathbf{l}}{p}}{p}\right)p^\frac{-1}{2} + 1}{1+p^{-1}}  \prod_{p\vert \frac{b}{\mathbf{q}^{\mathbf{l}}}}  \frac{\left( \frac{\mathbf{p}^\mathbf{k}\mathbf{q}^\mathbf{l}}{p}\right)p^\frac{-1}{2} +1 }{1+p^{-1}} .
\end{align*}
Furthermore we observe
\begin{align*}
    \prod_{p\nmid 2ab}  \left( 1+\frac{\left( 1 + \left( \frac{-ab}{p}\right)\right)\left( \frac{\mathbf{p}^\mathbf{k}\mathbf{q}^\mathbf{l}}{p}\right)p^{\frac{-3}{2}}}{1+p^{-1}}\right)\\ = \sum_{\substack{v \text{ squarefree}\\ (v,2ab)=1}}\frac{1}{v^{\frac{3}{2}}}\prod_{p \vert v}(1+p^{-1})^{-1}\prod_{p\vert v}\left( 1 + \left( \frac{-ab}{p}\right) \right)\left( \frac{\mathbf{p}^{\mathbf{k}}\mathbf{q}^{\mathbf{l}}}{p}\right). 
\end{align*}
And we can write all the other Euler products as sums and obtain together with $\left( \frac{ \frac{-ab\mathbf{q}^\mathbf{l}}{\mathbf{p}^{\mathbf{k}}}}{\mathbf{p}^\mathbf{k}}\right)\left( \frac{\mathbf{p}^\mathbf{k}\mathbf{q}^\mathbf{l}}{\frac{a}{\mathbf{p}^{\mathbf{k}}}}\right) = \left( \frac{\frac{b}{\mathbf{q}^\mathbf{l}}}{\mathbf{p}^\mathbf{k}} \right) \left( \frac{\mathbf{q}^\mathbf{l}}{\frac{a}{\mathbf{p}^\mathbf{k}}}\right)$ that $\hat{\tau}_{U,D}(\beta_z)$ is equal to 
\begin{align*}
  \frac{8\overline{c}_{a,b}}{\pi^2}\left( \frac{\frac{b}{\mathbf{q}^\mathbf{l}}}{\mathbf{p}^\mathbf{k}} \right) \left( \frac{\mathbf{q}^\mathbf{l}}{\frac{a}{\mathbf{p}^\mathbf{k}}}\right) \sum_{d \vert \frac{a}{\mathbf{p}^\mathbf{k}}}\sum_{e \vert \frac{b}{\mathbf{q}^\mathbf{l}}}\sum_{r\vert \mathbf{p}^{\mathbf{k}}} \sum_{s\vert \mathbf{q}^{\mathbf{l}}}\sum_{\substack{v \text{ squarefree}\\ \gcd(v,2ab)=1}}\frac{(ders)^{\frac{-1}{2}}}{v^{\frac{3}{2}}}\prod_{p \vert abv}(1+p^{-1})^{-1}\\
\times \prod_{p\vert v}\left( 1 + \left( \frac{-ab}{p}\right) \right)\left( \frac{\mathbf{p}^{\mathbf{k}}\mathbf{q}^{\mathbf{l}}}{p}\right)
       \left( \frac{ \frac{-ab\mathbf{q}^\mathbf{l}}{\mathbf{p}^{\mathbf{k}}}}{r}\right)  
    \left( \frac{\mathbf{p}^\mathbf{k}\mathbf{q}^\mathbf{l}}{d}\right)    \left( \frac{ \frac{-ab\mathbf{p}^\mathbf{k}}{\mathbf{q}^{\mathbf{l}}}}{s}\right)  
    \left( \frac{\mathbf{p}^\mathbf{k}\mathbf{q}^\mathbf{l}}{e}\right).
\end{align*}
For $v \in \mathbb{Z}\setminus\{0\}$
we define 
\begin{align}\label{def:f}
    f(v) \coloneqq \frac{\mu^2(v)\gamma(v)}{v^\frac{3}{2}}\prod_{p\mid v} \left(  1+ \left( \frac{-ab}{p}\right)\right)
\end{align}
with $\gamma(v) = \prod_{p \vert v}(1+p^{-1})^{-1}$. 
Then $f(v)\left( \frac{\mathbf{p}^{\mathbf{k}}\mathbf{q}^{\mathbf{l}}}{v}\right)$ is multiplicative in $v$ for fixed $\mathbf{p}^\mathbf{k}\mathbf{q}^\mathbf{l}$.
Furthermore, we define  
\begin{align*}
    S(v) \coloneqq \sum_{\mathbf{p}^{\mathbf{k}} \vert a }\sum_{\mathbf{q}^{\mathbf{l}} \vert b } \left( \frac{\frac{b}{\mathbf{q}^{\mathbf{l}}}}{\mathbf{p}^{\mathbf{k}}} \right) \left( \frac{\mathbf{q}^{\mathbf{l}}}{\frac{a}{\mathbf{p}^{\mathbf{k}}}}\right)
    \sum_{d \vert \frac{a}{\mathbf{p}^{\mathbf{k}}}}\sum_{e \vert \frac{b}{\mathbf{q}^{\mathbf{l}}}}\sum_{r\vert \mathbf{p}^{\mathbf{k}}} \sum_{s\vert \mathbf{q}^{\mathbf{l}}} (ders)^{\frac{-1}{2}}\\ \times 
    \left( \frac{ \frac{-ab\mathbf{q}^{\mathbf{l}}}{\mathbf{p}^{\mathbf{k}}}}{r}\right)
    \left( \frac{\mathbf{p}^{\mathbf{k}} \mathbf{q}^{\mathbf{l}}}{d}\right)   \left( \frac{ \frac{-ab\mathbf{p}^{\mathbf{k}}}{\mathbf{q}^{\mathbf{l}}}}{s}\right) \left( \frac{\mathbf{p}^{\mathbf{k}} \mathbf{q}^{\mathbf{l}}}{e}\right)\left( \frac{\mathbf{p}^{\mathbf{k}} \mathbf{q}^{\mathbf{l}}}{v}\right) 
\end{align*}
and obtain with (\ref{eq:new_conj_sum}) that
\begin{align}\label{eq:c_neu-finall}
    c_{\text{CLT-BT}} = \frac{4\overline{c}_{a,b}}{\pi^2} \prod_{p \vert ab }\left(1+\frac{1}{p} \right)^{-1}\sum_{\substack{v\geq 1\\ \gcd(v,2ab)=1}}f(v)S(v). 
\end{align}

\section{Comparison of the constants }\label{seq:comparison_constant}
The aim of this section is to rewrite the two constants, $C$ and $c_{\text{CLT-BT}}$, in different forms so that a direct comparison can be made.
We write 
\begin{align*}
\gamma_{(\mathbf{k},\mathbf{l}),v} = \gamma(v)\prod_{p \vert ab} \left( 1 + \frac{1}{p}\right)^{-1}
\end{align*}
with $\gamma(v)=\prod_{p \vert v} \left( 1 + \frac{1}{p}\right)^{-1}$. Then we obtain for $C$ in (\ref{eq:C_final})
\begin{align*}
C&=\frac{4\overline{c}_{a,b}}{\pi^2}\prod_{p \vert ab} \left( 1 + \frac{1}{p}\right)^{-1} \\ &\times \sum_{\mathbf{k}\in \{0,1\}^n}\sum_{\mathbf{l}\in \{0,1\}^m} \frac{1}{(\mathbf{p}^{\mathbf{k}}  \mathbf{q}^{\mathbf{l}})^{\frac{1}{2}}}\sum_{\substack{v\geq 1\\ v\mathbf{p}^{\mathbf{k}}  \mathbf{q}^{\mathbf{l}} \equiv 1 \text{ mod }4}} \frac{\mu^2(abv)\gamma(v)\varrho_{(\mathbf{k},\mathbf{l}),v}}{v^\frac{3}{2}}.
\end{align*}
We replace $\mu^2(abv)$ by $\mu^2(v)$ and the condition $gcd(v,ab)=1$. Furthermore let $R=\frac{4\overline{c}_{a,b}}{\pi^2} \prod_{p \vert ab }\left(1+\frac{1}{p} \right)^{-1} $.
Then we get 
\begin{align*}
    \frac{C}{R} = \sum_{\mathbf{k}\in \{0,1\}^n}\sum_{\mathbf{l}\in \{0,1\}^m} \frac{1}{(\mathbf{p}^{\mathbf{k}}  \mathbf{q}^{\mathbf{l}})^{\frac{1}{2}}}\sum_{\substack{v\geq 1 \\ v\mathbf{p}^{\mathbf{k}}  \mathbf{q}^{\mathbf{l}} \equiv 1 \text{ mod }4\\ \gcd(v,ab)=1}} \frac{\mu^2(v)\gamma(v)\varrho_{(\mathbf{k},\mathbf{l}),v}}{v^\frac{3}{2}}. 
\end{align*}
\begin{lemma}
If $\varrho_{(\mathbf{k},\mathbf{l}),v} \neq 0$, and $v$ odd, then
\begin{align}
    v \equiv \begin{cases}
         1 \mod 4 & \text{ if } \mathbf{p}^\mathbf{k}\mathbf{q}^\mathbf{l} \equiv 1 \mod 4\\
         3 \mod 4 & \text{ if } \mathbf{p}^\mathbf{k}\mathbf{q}^\mathbf{l} \equiv 3 \mod 4\\
    \end{cases}
\end{align}
\end{lemma}
\begin{proof}
Assume $\varrho_{(\mathbf{k},\mathbf{l}),v} \neq 0$. Then since none of the Legendre symbols can be zero we obtain 
\begin{align*}
1 &=\prod_{p \vert \frac{a}{\mathbf{p}^\mathbf{k}}} \left(\frac{bv\mathbf{p}^k\mathbf{q}^\mathbf{l}}{p} \right)\prod_{p \vert \frac{b}{\mathbf{q}^\mathbf{l}}} \left(\frac{av\mathbf{p}^k\mathbf{q}^\mathbf{l}}{p} \right)\prod_{p \vert v} \left(\frac{-ab}{p} \right)\prod_{p \vert \mathbf{q}^\mathbf{l}} \left( \frac{ \frac{bv\mathbf{p}^\mathbf{k}\mathbf{q}^{\mathbf{l}}}{p^2}}{p} \right)
\prod_{p \vert \mathbf{p}^\mathbf{k}} \left( \frac{ \frac{av\mathbf{p}^\mathbf{k}\mathbf{q}^{\mathbf{l}}}{p^2}}{p} \right)\\
&=\left(\frac{bv\mathbf{p}^k\mathbf{q}^\mathbf{l}}{\frac{a}{\mathbf{p}^\mathbf{k}}} \right) \left(\frac{av\mathbf{p}^k\mathbf{q}^\mathbf{l}}{\frac{b}{\mathbf{q}^\mathbf{l}}} \right)\left(\frac{-ab}{v} \right)\prod_{p \vert \mathbf{q}^\mathbf{l}} \left( \frac{ \frac{bv\mathbf{p}^\mathbf{k}\mathbf{q}^{\mathbf{l}}}{p^2}}{p} \right)
\prod_{p \vert \mathbf{p}^\mathbf{k}} \left( \frac{ \frac{av\mathbf{p}^\mathbf{k}\mathbf{q}^{\mathbf{l}}}{p^2}}{p} \right),
\end{align*}
where we have in the last line jacobi symbols. 
Together with quadratic reciprocity we get the lemma.
\end{proof}
Furthermore, we observe
\begin{align*}
    \prod_{p \vert \mathbf{q}^\mathbf{l}}\left(1 + \left( \frac{ \frac{bv\mathbf{p}^\mathbf{k}\mathbf{q}^{\mathbf{l}}}{p^2}}{p} \right)\right)\prod_{p \vert \mathbf{p}^\mathbf{k}}\left(1 + \left( \frac{ \frac{av\mathbf{p}^\mathbf{k}\mathbf{q}^{\mathbf{l}}}{p^2}}{p} \right)\right)\\=  \prod_{p \vert \mathbf{q}^\mathbf{l}}\left(1 + \left( \frac{\frac{bv \mathbf{p}^\mathbf{k}}{\mathbf{q}^\mathbf{l}}}{p} \right)\right)\prod_{p \vert \mathbf{p}^\mathbf{k}}\left(1 + \left( \frac{\frac{av\mathbf{q}^\mathbf{l}}{\mathbf{p}^\mathbf{k}}}{p} \right)\right).
\end{align*}
We recall the definition of $f$ in (\ref{def:f}) and obtain
\begin{align}
\frac{C}{R} =\sum_{\mathbf{k}\in \{0,1\}^n}\sum_{\mathbf{l}\in \{0,1\}^m} \frac{1}{(\mathbf{p}^{\mathbf{k}}  \mathbf{q}^{\mathbf{l}})^{\frac{1}{2}}}\sum_{\substack{v\geq 1 
\\  \gcd(v,2ab)=1}} f(v) \prod_{p|\frac{a}{\mathbf{p}^\mathbf{k}}}  \left(1 + \left( \frac{vb\mathbf{p}^{\mathbf{k}}\mathbf{q}^{\mathbf{l}}}{p}\right)\right) 
    \prod_{p|\frac{b}{\mathbf{q}^\mathbf{l}}}\\ \times    \left(1 + \left( \frac{va\mathbf{p}^{\mathbf{k}}\mathbf{q}^{\mathbf{l}}}{p}\right)\right)\prod_{p \vert \mathbf{q}^\mathbf{l}}\left(1 + \left( \frac{\frac{bv \mathbf{p}^\mathbf{k}}{\mathbf{q}^\mathbf{l}}}{p} \right)\right)
    \prod_{p \vert \mathbf{p}^\mathbf{k}}\left(1 + \left( \frac{\frac{av\mathbf{q}^\mathbf{l}}{\mathbf{p}^\mathbf{k}}}{p} \right)\right). \notag
    \end{align}
    We write the products in the last line as sums and obtain 
    \begin{align}\label{eq:C/R_1}
    \frac{C}{R}=\sum_{\mathbf{k}\in \{0,1\}^n}\sum_{\mathbf{l}\in \{0,1\}^m} \frac{1}{(\mathbf{p}^{\mathbf{k}}  \mathbf{q}^{\mathbf{l}})^{\frac{1}{2}}}\sum_{d\vert \frac{a}{\mathbf{p}^{\mathbf{k}}}} \sum_{e\vert \frac{b}{\mathbf{q}^{\mathbf{l}}}} \sum_{r\vert \mathbf{p}^\mathbf{k}}\sum_{s\vert \mathbf{q}^\mathbf{l}}\sum_{\substack{v\geq 1 \\ \gcd(v,2ab)=1}} f(v)\\ \times  \left( \frac{vb\mathbf{p}^{\mathbf{k}}\mathbf{q}^{\mathbf{l}}}{d} \right) \left( \frac{va\mathbf{p}^{\mathbf{k}}\mathbf{q}^{\mathbf{l}}}{e} \right)
    \left( \frac{\frac{av\mathbf{q}^\mathbf{l}}{\mathbf{p}^\mathbf{k}}}{r} \right)
    \left( \frac{\frac{bv \mathbf{p}^\mathbf{k}}{\mathbf{q}^\mathbf{l}}}{s} \right).\notag
\end{align}
 For $\mathbf{p}^{\mathbf{k}} \vert a, \mathbf{q}^{\mathbf{l}}\vert b, d \vert \frac{a}{\mathbf{p}^{\mathbf{k}}}, e \vert \frac{b}{\mathbf{q}^{\mathbf{l}}}, r \vert \mathbf{p}^{\mathbf{k}}, s \vert \mathbf{q}^{\mathbf{l}}$ and $v \in \mathbb{Z}\setminus \{0\}$ with $\gcd(v,2ab)=1$ we define 
    \begin{align*}
        \mathscr{A}(\mathbf{p}^{\mathbf{k}},\mathbf{q}^{\mathbf{l}},d,e,r,s,v)\coloneqq\left( \frac{vb\mathbf{p}^{\mathbf{k}}\mathbf{q}^{\mathbf{l}}}{d} \right) \left( \frac{va\mathbf{p}^{\mathbf{k}}\mathbf{q}^{\mathbf{l}}}{e} \right)
    \left( \frac{\frac{av\mathbf{q}^\mathbf{l}}{\mathbf{p}^\mathbf{k}}}{r} \right)
    \left( \frac{\frac{bv \mathbf{p}^\mathbf{k}}{\mathbf{q}^\mathbf{l}}}{s} \right) 
    \end{align*}
    and\begin{align}T_{\mathbf{p}^{\mathbf{k}},\mathbf{q}^{\mathbf{l}}}(v)& \coloneqq \sum_{d\vert \frac{a}{\mathbf{p}^{\mathbf{k}}}} \sum_{e\vert \frac{b}{\mathbf{q}^{\mathbf{l}}}} \sum_{r\vert \mathbf{p}^\mathbf{k}}\sum_{s\vert \mathbf{q}^\mathbf{l}} \mathscr{A}(\mathbf{p}^{\mathbf{k}},\mathbf{q}^{\mathbf{l}},d,e,r,s,v).
    \end{align}
    Then we obtain with (\ref{eq:C/R_1}) the following identity
    \begin{align}
        \frac{C}{R}&= 
        \sum_{\mathbf{p}^\mathbf{k}\vert a}\sum_{\mathbf{q}^\mathbf{l}\vert b} (\mathbf{p}^\mathbf{k}\mathbf{q}^\mathbf{l})^{\frac{-1}{2}}\sum_{\substack{v\geq 1\\ \gcd(v,2ab)=1}}f(v)\sum_{d\vert \frac{a}{\mathbf{p}^{\mathbf{k}}}} \sum_{e\vert \frac{b}{\mathbf{q}^{\mathbf{l}}}} \sum_{r\vert \mathbf{p}^\mathbf{k}}\sum_{s\vert \mathbf{q}^\mathbf{l}} \mathscr{A}(\mathbf{p}^{\mathbf{k}},\mathbf{q}^{\mathbf{l}},d,e,r,s,v)\\
        &= \sum_{\mathbf{p}^\mathbf{k}\vert a}\sum_{\mathbf{q}^\mathbf{l}\vert b} (\mathbf{p}^\mathbf{k}\mathbf{q}^\mathbf{l})^{\frac{-1}{2}}\sum_{\substack{v\geq 1 \\ \gcd(v,2ab)=1}}f(v)T_{\mathbf{p}^\mathbf{k},\mathbf{q}^\mathbf{l}}(v).\notag
    \end{align}
By (\ref{eq:c_neu-finall}) we obtain
\begin{align*}
    \frac{c_{\text{CLT-BT}}}{R} = \sum_{\substack{v\geq 1\\ (v,2ab)=1}}f(v)S(v). 
\end{align*}
We apply the bijective change of variables
\begin{align}\label{change_of_variables}
    (\mathbf{p}^\mathbf{k}, \mathbf{q}^\mathbf{l},d,e,r,s)\mapsto \left(dr,es,\frac{\mathbf{p}^\mathbf{k}}{r},\frac{\mathbf{q}^\mathbf{l}}{s},r,s\right)
\end{align}
to $S(v)$ and obtain 
\begin{align*}
       S(v) &= \sum_{\mathbf{p}^{\mathbf{k}} \vert a }\sum_{\mathbf{q}^{\mathbf{l}} \vert b } (\mathbf{p}^{\mathbf{k}}\mathbf{q}^{\mathbf{l}})^{\frac{-1}{2}} 
    \sum_{d \vert \frac{a}{\mathbf{p}^{\mathbf{k}}}}\sum_{e \vert \frac{b}{\mathbf{q}^{\mathbf{l}}}}\sum_{r\vert \mathbf{p}^{\mathbf{k}}} \sum_{s\vert \mathbf{q}^{\mathbf{l}}}\left( \frac{\frac{b}{es}}{dr} \right) 
    \left( \frac{es}{\frac{a}{dr}}\right)\\ 
    &\times
    \left( \frac{ \frac{-abes}{dr}}{r}\right)
    \left( \frac{dres}{\frac{\mathbf{p}^\mathbf{k}}{r}}\right) 
    \left( \frac{ \frac{-abdr}{es}}{s}\right) 
    \left( \frac{dres}{\frac{\mathbf{q}^\mathbf{l}}{s}}\right)
    \left( \frac{dres}{v}\right). 
\end{align*}
 For $\mathbf{p}^{\mathbf{k}} \vert a, \mathbf{q}^{\mathbf{l}}\vert b, d \vert \frac{a}{\mathbf{p}^{\mathbf{k}}}, e \vert \frac{b}{\mathbf{q}^{\mathbf{l}}}, r \vert \mathbf{p}^{\mathbf{k}}, s \vert \mathbf{q}^{\mathbf{l}}$ and $v \in \mathbb{Z}\setminus \{0\}$ with $\gcd(v,2ab)=1$ we define 
    \begin{align*}
        \mathscr{B}(\mathbf{p}^{\mathbf{k}},\mathbf{q}^{\mathbf{l}},d,e,r,s,v)\coloneqq\left( \frac{\frac{b}{es}}{dr} \right) 
    \left( \frac{es}{\frac{a}{dr}}\right)
    \left( \frac{ \frac{-abes}{dr}}{r}\right)
    \left( \frac{dres}{\frac{\mathbf{p}^\mathbf{k}}{r}}\right)\\ 
    \times \left( \frac{ \frac{-abdr}{es}}{s}\right) 
    \left( \frac{dres}{\frac{\mathbf{q}^\mathbf{l}}{s}}\right)
    \left( \frac{dres}{v}\right). 
    \end{align*}
    Then we obtain
    \begin{align}\label{eq:C_neu/R_1}
         \frac{c_{\text{CLT-BT}}}{R} = \sum_{\mathbf{p}^{\mathbf{k}} \vert a }\sum_{\mathbf{q}^{\mathbf{l}} \vert b } (\mathbf{p}^{\mathbf{k}}\mathbf{q}^{\mathbf{l}})^{\frac{-1}{2}} \sum_{\substack{v\\ \gcd(v,2ab)=1}}f(v)\\ \times 
    \sum_{d \vert \frac{a}{\mathbf{p}^{\mathbf{k}}}}\sum_{e \vert \frac{b}{\mathbf{q}^{\mathbf{l}}}}\sum_{r\vert \mathbf{p}^{\mathbf{k}}} \sum_{s\vert \mathbf{q}^{\mathbf{l}}}
   \mathscr{B}(\mathbf{p}^{\mathbf{k}},\mathbf{q}^{\mathbf{l}},d,e,r,s,v). \notag
    \end{align}

\begin{lemma}\label{lemma:curlyA,curlyB} For $\mathbf{p}^{\mathbf{k}} \vert a, \mathbf{q}^{\mathbf{l}}\vert b, d \vert \frac{a}{\mathbf{p}^{\mathbf{k}}}, e \vert \frac{b}{\mathbf{q}^{\mathbf{l}}}, r \vert \mathbf{p}^{\mathbf{k}}, s \vert \mathbf{q}^{\mathbf{l}}$ and $v \in \mathbb{Z}\setminus \{0\}$ with $\gcd(v,2ab)=1$ we get 
    \begin{align*}
    \mathscr{A}(\mathbf{p}^{\mathbf{k}},\mathbf{q}^{\mathbf{l}},d,e,r,s,v)= (-1)^{\frac{(dres-1)}{2}\frac{(v\mathbf{p}^\mathbf{k}\mathbf{q}^\mathbf{l}-1)}{2}}\mathscr{B}(\mathbf{p}^{\mathbf{k}},\mathbf{q}^{\mathbf{l}},d,e,r,s,v).
\end{align*} 
\end{lemma}
\begin{proof}
    We consider $\mathscr{A}(\mathbf{p}^{\mathbf{k}},\mathbf{q}^{\mathbf{l}},d,e,r,s,v)$ and obtain with quadratic reciprocity
     \begin{align*}
        \mathscr{A}(\mathbf{p}^{\mathbf{k}},\mathbf{q}^{\mathbf{l}},d,e,r,s,v) = \left(  \frac{e \cdot s \cdot \frac{b}{\mathbf{q}^\mathbf{l}e}\cdot \frac{\mathbf{p}^\mathbf{k}}{r}\cdot r \cdot s}{d}\right)
        \left(  \frac{ d \cdot  r \cdot \frac{a}{\mathbf{p}^\mathbf{k}d}\cdot \frac{\mathbf{q}^\mathbf{l}}{s}\cdot s \cdot  r}{e}\right)\\ \times 
        \left( \frac{\frac{ d \cdot \frac{\mathbf{p}^\mathbf{k}}{r} \cdot r \cdot \frac{a}{\mathbf{p}^\mathbf{k}d}\cdot \frac{\mathbf{q}^\mathbf{l}}{s} \cdot s }{\frac{\mathbf{p}^\mathbf{k}}{r}\cdot r}}{r}\right) \left( \frac{\frac{ e \cdot \frac{\mathbf{q}^\mathbf{l}}{s} \cdot s \cdot \frac{b}{\mathbf{q}^\mathbf{l}e}\cdot \frac{\mathbf{p}^\mathbf{k}}{r} \cdot r }{\frac{\mathbf{q}^\mathbf{l}}{s}\cdot s}}{s}\right)\left( \frac{v}{ders}\right).
    \end{align*}
For $\mathscr{B}(\mathbf{p}^{\mathbf{k}},\mathbf{q}^{\mathbf{l}},d,e,r,s,v)$ we obtain with quadratic reciprocity
   \begin{align*}
        \mathscr{B}(\mathbf{p}^{\mathbf{k}},\mathbf{q}^{\mathbf{l}},d,e,r,s,v) &= (-1)^{\frac{(dres-1)}{2}\left( \frac{(v-1)}{2} + \frac{\left(\frac{\mathbf{p}^\mathbf{k}}{r}-1\right)}{2} + \frac{\left(\frac{\mathbf{q}^\mathbf{l}}{s}-1\right)}{2} \right) + \frac{(es-1)}{2}\frac{\left(\frac{a}{rd}-1\right)}{2}}\\
        &\times  \left(\frac{\frac{\mathbf{q}^\mathbf{l}}{s}\cdot \frac{b}{\mathbf{q}^\mathbf{l}} \cdot \frac{\mathbf{p}^\mathbf{k}}{r} \cdot \frac{\mathbf{q}^\mathbf{l}}{s}}{d} \right)
        \left( \frac{\frac{\mathbf{p}^\mathbf{k}}{r} \cdot \frac{a}{\mathbf{p}^\mathbf{k}d}\cdot \frac{\mathbf{p}^\mathbf{k}}{r} \cdot \frac{\mathbf{q}^\mathbf{l}}{s}}{e}\right)\\ &\times 
        \left( \frac{\frac{-a}{\mathbf{p}^\mathbf{k}d} \cdot \frac{\mathbf{q}^\mathbf{l}}{s}}{r}\right)
        \left( \frac{\frac{\mathbf{p}^\mathbf{k}}{r}\cdot \frac{-b}{\mathbf{q}^\mathbf{l}e}}{s}\right)
        \left( \frac{v}{dres}\right).
    \end{align*}
We take 
    \begin{align*}
        \hat{R} \coloneqq 
        \left( \frac{\frac{b}{\mathbf{q}^\mathbf{l}e} \cdot \frac{\mathbf{p}^\mathbf{k}}{r} }{d}\right)
        \left( \frac{\frac{a}{\mathbf{p}^\mathbf{k}d} \cdot \frac{\mathbf{q}^\mathbf{l}}{s}}{e}\right)
        \left( \frac{\frac{a}{\mathbf{p}^\mathbf{k}d}\cdot \frac{\mathbf{q}^\mathbf{l}}{s}}{r}\right)
        \left( \frac{\frac{b}{\mathbf{q}^\mathbf{l}e} \cdot \frac{\mathbf{p}^\mathbf{k}}{r} }{s}\right)
        \left( \frac{v}{ders}\right).
    \end{align*}
This is the common factor in $\mathscr{A}(\mathbf{p}^{\mathbf{k}},\mathbf{q}^{\mathbf{l}},d,e,r,s,v)$ and $\mathscr{B}(\mathbf{p}^{\mathbf{k}},\mathbf{q}^{\mathbf{l}},d,e,r,s,v).$ We obtain 
\begin{align*}
          \frac{\mathscr{A}(\mathbf{p}^{\mathbf{k}},\mathbf{q}^{\mathbf{l}},d,e,r,s,v)}{\hat{R}} = \left( \frac{ds}{er}\right) \left( \frac{er}{ds}\right) = (-1)^{\frac{(er-1)}{2}\frac{(ds-1)}{2}} 
    \end{align*}
    and
\begin{align*}
          \frac{\mathscr{B}(\mathbf{p}^{\mathbf{k}},\mathbf{q}^{\mathbf{l}},d,e,r,s,v)}{\hat{R}}  
          &= (-1)^{\frac{(dres-1)}{2}\left( \frac{(v-1)}{2} + \frac{\left(\frac{\mathbf{p}^\mathbf{k}}{r}-1\right)}{2} + \frac{\left(\frac{\mathbf{q}^\mathbf{l}}{s}-1\right)}{2} \right) }\\ 
          &\times (-1)^{ \frac{(es-1)}{2}\frac{\left(\frac{a}{rd}-1\right)}{2} + \frac{s-1}{2}+\frac{r-1}{2}}.   
    \end{align*}
For $x,y \in \mathbb{N}$ odd, we know $\frac{x-1}{2} \equiv 1 \mod 2$ if $x\equiv 3  \mod 4$ and $\frac{x-1}{2} \equiv 0 \mod 2$ if $x\equiv 1  \mod 4$. Furthermore, $\frac{xy-1}{2}\equiv \frac{x-1}{2}+\frac{y-1}{2} \mod 2 $. Since $a \equiv d \cdot \frac{\mathbf{p}^\mathbf{k}}{r} \cdot r \cdot \frac{a}{\mathbf{p}^\mathbf{k}} \equiv 1 \mod 4$ we obtain 
$$\frac{\left(\frac{a}{\mathbf{p}^\mathbf{k}d }-1\right)}{2} \equiv \frac{\left(d-1\right)}{2} + \frac{\left(\frac{\mathbf{p}^\mathbf{k}}{r }-1\right)}{2} + \frac{\left( r-1 \right)}{2} \mod 2. $$ 
Thus 
\begin{align*}
    \frac{(es-1)}{2}\frac{(\frac{a}{rd}-1)}{2} &\equiv \frac{(es-1)}{2}\left( \frac{(\frac{\mathbf{p}^\mathbf{k}}{r}-1)}{2} + \frac{\left(\frac{a}{\mathbf{p}^\mathbf{k}d }-1\right)}{2} \right)\\ 
         &\equiv \frac{(es-1)}{2}\frac{(dr-1)}{2} \mod 2.  
\end{align*}
Then 
\begin{align*}
    \frac{\mathscr{A}(\mathbf{p}^{\mathbf{k}},\mathbf{q}^{\mathbf{l}},d,e,r,s,v)}{\hat{R}} \frac{\mathscr{B}(\mathbf{p}^{\mathbf{k}},\mathbf{q}^{\mathbf{l}},d,e,r,s,v)}{\hat{R}}  
    = (-1)^{\frac{(dres-1)}{2}\frac{(v\mathbf{p}^\mathbf{k}\mathbf{q}^\mathbf{l}-1)}{2}}.
\end{align*}

\end{proof}   
By (\ref{eq:C/R_1}) and (\ref{eq:C_neu/R_1}), to establish Theorem~\ref{thm:constants_equal}, it remains to show that
\begin{align}\label{eq:T_p,k=sumB}
T_{\mathbf{p}^\mathbf{k},\mathbf{q}^\mathbf{l}}(v) =  \sum_{d \vert \frac{a}{\mathbf{p}^{\mathbf{k}}}}\sum_{e \vert \frac{b}{\mathbf{q}^{\mathbf{l}}}}\sum_{r\vert \mathbf{p}^{\mathbf{k}}} \sum_{s\vert \mathbf{q}^{\mathbf{l}}}
   \mathscr{B}(\mathbf{p}^{\mathbf{k}},\mathbf{q}^{\mathbf{l}},d,e,r,s,v)
\end{align}
for all $\mathbf{p}^\mathbf{k} \vert a, \mathbf{q}^\mathbf{l}\vert b$ and $v \in \mathbb{N}$ with $\gcd(v,2ab)=1$.
We consider the two cases $v\mathbf{p}^\mathbf{k}\mathbf{q}^\mathbf{l} \equiv 1 \mod 4$ and $v\mathbf{p}^\mathbf{k}\mathbf{q}^\mathbf{l} \equiv 3 \mod 4$ separately. If $v\mathbf{p}^\mathbf{k}\mathbf{q}^\mathbf{l} \equiv 1 \mod 4$ we get by Lemma~\ref{lemma:curlyA,curlyB} that $\mathscr{A}(\mathbf{p}^{\mathbf{k}},\mathbf{q}^{\mathbf{l}},d,e,r,s,v)= \mathscr{B}(\mathbf{p}^{\mathbf{k}},\mathbf{q}^{\mathbf{l}},d,e,r,s,v).$ Thus (\ref{eq:T_p,k=sumB}) is true. 
Now we assume $v\mathbf{p}^\mathbf{k}\mathbf{q}^\mathbf{l} \equiv 3 \mod 4$.     Since 
    \begin{align*}
T_{\mathbf{p}^{\mathbf{k}},\mathbf{q}^{\mathbf{l}}}(v)
    &= \prod_{p|\frac{a}{\mathbf{p}^\mathbf{k}}}  \left(1 + \left( \frac{vb\mathbf{p}^{\mathbf{k}}\mathbf{q}^{\mathbf{l}}}{p}\right)\right) 
    \prod_{p|\frac{b}{\mathbf{q}^\mathbf{l}}}    \left(1 + \left( \frac{va\mathbf{p}^{\mathbf{k}}\mathbf{q}^{\mathbf{l}}}{p}\right)\right)\\ & \times \prod_{p \vert \mathbf{q}^\mathbf{l}}\left(1 + \left( \frac{\frac{bv \mathbf{p}^\mathbf{k}}{\mathbf{q}^\mathbf{l}}}{p} \right)\right)
    \prod_{p \vert \mathbf{p}^\mathbf{k}}\left(1 + \left( \frac{\frac{av\mathbf{q}^\mathbf{l}}{\mathbf{p}^\mathbf{k}}}{p} \right)\right),
    \end{align*}
    we get for the product of all the jacobi symbols in $T_{\mathbf{p}^\mathbf{k},\mathbf{q}^\mathbf{l}}(v)$ together with quadratic reciprocity and $a\equiv b \equiv1 \mod 4$
     \begin{align*}
        \left( \frac{vb\mathbf{p}^{\mathbf{k}}\mathbf{q}^{\mathbf{l}}}{\frac{a}{\mathbf{p}^\mathbf{k}}} \right) \left( \frac{va\mathbf{p}^{\mathbf{k}}\mathbf{q}^{\mathbf{l}}}{\frac{b}{\mathbf{q}^\mathbf{l}}} \right)
    \left( \frac{\frac{av\mathbf{q}^\mathbf{l}}{\mathbf{p}^\mathbf{k}}}{\mathbf{p}^\mathbf{k}} \right)
    \left( \frac{\frac{bv \mathbf{p}^\mathbf{k}}{\mathbf{q}^\mathbf{l}}}{\mathbf{q}^\mathbf{l}} \right)\\ 
    = (-1)^{\frac{(\mathbf{p}^\mathbf{k}-1)}{2}\frac{\left(\frac{a}{\mathbf{p}^\mathbf{k}}-1\right)}{2}+\frac{(\mathbf{q}^\mathbf{l}-1)}{2}\frac{\left(\frac{b}{\mathbf{q}^\mathbf{l}}-1\right)}{2} }
    \left( \frac{v}{ab} \right).
    \end{align*}
    Since $a \equiv \mathbf{p}^\mathbf{k} \cdot \frac{a}{\mathbf{p}^\mathbf{k}} \equiv 1 \mod 4$ we have $0 \equiv \frac{a-1}{2} \equiv \frac{\mathbf{p}^\mathbf{k}\cdot \frac{a}{\mathbf{p}^\mathbf{k}}-1}{2} \equiv \frac{\mathbf{p}^\mathbf{k}-1}{2}+ \frac{ \frac{a}{\mathbf{p}^\mathbf{k}}-1}{2} \mod 2$ and so $\frac{\mathbf{p}^\mathbf{k}-1}{2} \equiv \frac{ \frac{a}{\mathbf{p}^\mathbf{k}}-1}{2} \mod 2$. In the same way we get $\frac{\mathbf{q}^\mathbf{l}-1}{2} \equiv \frac{ \frac{b}{\mathbf{q}^\mathbf{l}}-1}{2} \mod 2$. Therefore, we get 
    \begin{align*}
        \left( \frac{vb\mathbf{p}^{\mathbf{k}}\mathbf{q}^{\mathbf{l}}}{\frac{a}{\mathbf{p}^\mathbf{k}}} \right) \left( \frac{va\mathbf{p}^{\mathbf{k}}\mathbf{q}^{\mathbf{l}}}{\frac{b}{\mathbf{q}^\mathbf{l}}} \right)
    \left( \frac{\frac{av\mathbf{q}^\mathbf{l}}{\mathbf{p}^\mathbf{k}}}{\mathbf{p}^\mathbf{k}} \right)
    \left( \frac{\frac{bv \mathbf{p}^\mathbf{k}}{\mathbf{q}^\mathbf{l}}}{\mathbf{q}^\mathbf{l}} \right) = \left( \frac{v}{ab}\right)(-1)^{\frac{\mathbf{p}^\mathbf{k}\mathbf{q}^\mathbf{l}-1}{2}}. 
    \end{align*}
If $f(v)\neq 0$ and $\gcd(v,2ab)=1,$ then $\left( \frac{-ab}{v}\right)=1$
    and so $\left( \frac{v}{ab}\right)= \left( \frac{ab}{v}\right)=(-1)^{\frac{v-1}{2}}.$ Thus,  $\left( \frac{v}{ab} \right)(-1)^{\frac{\mathbf{p}^\mathbf{k}\mathbf{q}^\mathbf{l}-1}{2}}=-1$ and at least one of the factors in $T_{\mathbf{p}^\mathbf{k}\mathbf{q}^\mathbf{l}}(v)$ is zero. Therefore, $T_{\mathbf{p}^\mathbf{k}\mathbf{q}^\mathbf{l}}(v)=0$. It is left to show that 
    \begin{align}\sum_{d\vert \frac{a}{\mathbf{p}^{\mathbf{k}}}} \sum_{e\vert \frac{b}{\mathbf{q}^{\mathbf{l}}}} \sum_{r\vert \mathbf{p}^\mathbf{k}}\sum_{s\vert \mathbf{q}^\mathbf{l}} \mathscr{B}(\mathbf{p}^{\mathbf{k}},\mathbf{q}^{\mathbf{l}},d,e,r,s,v)=0.
    \end{align}
    If $v\mathbf{p}^\mathbf{k}\mathbf{q}^\mathbf{l} \equiv 3 \mod 4$ we observe $
    \mathscr{B}(\mathbf{p}^{\mathbf{k}},\mathbf{q}^{\mathbf{l}},d,e,r,s,v) = \mathscr{A}(\mathbf{p}^{\mathbf{k}},\mathbf{q}^{\mathbf{l}},d,e,r,s,-v)$ and 
        \begin{align*}
        \sum_{d\vert \frac{a}{\mathbf{p}^{\mathbf{k}}}} \sum_{e\vert \frac{b}{\mathbf{q}^{\mathbf{l}}}} \sum_{r\vert \mathbf{p}^\mathbf{k}}\sum_{s\vert \mathbf{q}^\mathbf{l}} \mathscr{B}(\mathbf{p}^{\mathbf{k}},\mathbf{q}^{\mathbf{l}},d,e,r,s,v)
        = T_{\mathbf{p}^\mathbf{k}\mathbf{q}^\mathbf{l}}(-v).
    \end{align*}
 Since $ab\equiv 1 \mod 4$ replacing $v$ by $-v$ does not change the product of the jacobi symbols. Thus, $T_{\mathbf{p}^\mathbf{k}\mathbf{q}^\mathbf{l}}(-v)=0$ and we get 
\begin{align*}
T_{\mathbf{p}^\mathbf{k},\mathbf{q}^\mathbf{l}}(v) =  \sum_{d \vert \frac{a}{\mathbf{p}^{\mathbf{k}}}}\sum_{e \vert \frac{b}{\mathbf{q}^{\mathbf{l}}}}\sum_{r\vert \mathbf{p}^{\mathbf{k}}} \sum_{s\vert \mathbf{q}^{\mathbf{l}}}
   \mathscr{B}(\mathbf{p}^{\mathbf{k}},\mathbf{q}^{\mathbf{l}},d,e,r,s,v) = 0. 
\end{align*}
This completes the proof of Theorem~\ref{thm:constants_equal}.

\begin{remark}
We expect Theorem ~\ref{thm:constants_equal} to extend to a more general setting in which $S$ contains all the primes that divide $ab$, as well as a finite number of additional primes.
\end{remark}

\section*{Declaration of AI use}
With the help of GPT-5.6 Sol, the constructions in Subsection \ref{sec:new_parametrization} were developed.
In addition, a discussion with GPT-5.6 Sol was helpful in developing the final change of variables (\ref{change_of_variables}), the identity (\ref{eq:T_p,k=sumB}) and Lemma \ref{lemma:curlyA,curlyB}.
\section*{Acknowledgments}
The author would like to thank Tim Browning, Marta Pieropan, Julian Reichardt, Damaris Schindler and Florian Wilsch for discussions around this project. Moreover, the author thanks Alec Shute for his insightful comments and suggestions, which led to an improvement of the error term in the asymptotic formula for $N_S(B)$.

\printbibliography

\end{document}